\documentclass[12pt]{article}

\usepackage{color}

\title{\Large\bf 
A presentation of the torus-equivariant \\[2mm]
quantum $K$-theory ring of flag manifolds of type $A$, \\[2mm]
Part I: the defining ideal%
\footnote{Key words and phrases: (quantum) $K$-theory, (quantum) Schubert calculus, 
semi-infinite flag manifold, inverse Chevalley formula. \newline
2020 Mathematics Subject Classification: Primary 14M15, 14N35;
Secondary 14N15, 05E10, 20C08.}%
}
\author{%
Toshiaki Maeno \\
 \small Department of Mathematics, Faculty of Science and Technology, Meijo University, \\
 \small 1-501 Shiogamaguchi, Tempaku-ku, Nagoya 468-8502, Japan \\
 \small (e-mail: {\tt tmaeno@meijo-u.ac.jp}) \\[5mm]
Satoshi Naito \\ 
 \small Department of Mathematics, Tokyo Institute of Technology, \\
 \small 2-12-1 Oh-okayama, Meguro-ku, Tokyo 152-8551, Japan \\
 \small (e-mail: {\tt naito@math.titech.ac.jp}) \\[5mm]
and \\[3mm]
Daisuke Sagaki \\ 
 \small Department of Mathematics, \\
 \small Faculty of Pure and Applied Sciences, University of Tsukuba, \\
 \small 1-1-1 Tennodai, Tsukuba, Ibaraki 305-8571, Japan \\
 \small (e-mail: {\tt sagaki@math.tsukuba.ac.jp})
}
\date{}

\usepackage{amsmath, amssymb, amsthm, amscd}
\usepackage{bm}
\usepackage{xcolor}
\usepackage{mathrsfs}

\makeatletter
\renewcommand\section{\@startsection{section}{1}{0pt}
{-3.5ex plus -1ex minus -.2ex}{1.0ex plus .2ex}{\large\bf}}
\renewcommand\subsection{\@startsection{subsection}{1}{0pt}
{2.5ex plus 1ex minus .2ex}{-1em}{\bf}}
\makeatother

\numberwithin{equation}{section}

\theoremstyle{plain}
\newtheorem{thm}{Theorem}[section]
\newtheorem{lem}[thm]{Lemma}
\newtheorem{prop}[thm]{Proposition}
\newtheorem{cor}[thm]{Corollary}

\newtheorem{ithm}{Theorem}
\newtheorem{icor}[ithm]{Corollary}
\newtheorem{iprop}[ithm]{Proposition}

\theoremstyle{definition}
\newtheorem{dfn}[thm]{Definition}

\theoremstyle{remark}

\newtheorem{rem}[thm]{Remark}

\newcommand{\BZ}{\mathbb{Z}}
\newcommand{\BQ}{\mathbb{Q}}

\newcommand{\BC}{\mathbb{C}}
\newcommand{\BH}{\mathbb{H}}

\newcommand{\CO}{\mathcal{O}}
\newcommand{\CF}{\mathcal{F}}
\newcommand{\CI}{\mathcal{I}}

\newcommand{\ba}{\mathbf{a}}
\newcommand{\bb}{\mathbf{b}}
\newcommand{\be}{\mathbf{e}}
\newcommand{\bp}{\mathbf{p}}
\newcommand{\bw}{\mathbf{w}}
\newcommand{\bI}{\mathbf{I}}

\newcommand{\ff}{\bm{f}}  
\newcommand{\FFF}{\bm{F}}

\newcommand{\st}{\mathsf{t}}
\newcommand{\sT}{\mathsf{T}}
\newcommand{\sD}{\mathsf{D}}
\newcommand{\sX}{\mathsf{X}}

\newcommand{\Qq}{\mathscr{Q}}
\newcommand{\q}{\mathsf{q}}

\newcommand{\Fg}{\mathfrak{g}}
\newcommand{\Fh}{\mathfrak{h}}
\newcommand{\FG}{\mathfrak{G}}
\newcommand{\FF}{\mathfrak{F}}
\newcommand{\FE}{\mathfrak{E}}
\newcommand{\FH}{\mathfrak{H}}
\newcommand{\Fsl}{\mathfrak{sl}}

\newcommand{\vpi}{\varpi}
\newcommand{\eps}{\epsilon}

\newcommand{\lng}{w_{\circ}}

\newcommand{\af}{\mathrm{af}}

\DeclareMathOperator{\wt}{wt}
\DeclareMathOperator{\upp}{up}
\DeclareMathOperator{\dnn}{down}

\newcommand{\QG}{\mathbf{Q}_{G}}
\newcommand{\QGr}{\mathbf{Q}_{G}^{\mathrm{rat}}}
\newcommand{\Kr}{K_{H \times \BC^{\ast}}(\QGr)}

\newcommand{\walk}{\mathbf{QW}_{\vpi_{1},s_{1}s_{2} \cdots s_{k}}}
\newcommand{\walkg}[1]{\mathbf{QW}_{\vpi_{1},#1}}
\newcommand{\twalk}{\ti{\mathbf{QW}}_{\vpi_{1},s_{1}s_{2} \cdots s_{k}}}
\newcommand{\twalkg}[1]{\ti{\mathbf{QW}}_{\vpi_{1},#1}}
\newcommand{\walka}[2]{\mathbf{QW}_{#1,#2}}
\newcommand{\twalka}[2]{\ti{\mathbf{QW}}_{#1,#2}}

\newcommand{\Hom}{\mathrm{Hom}}
\newcommand{\QBG}{\mathrm{QBG}}
\newcommand{\QLS}{\mathrm{QLS}}
\newcommand{\BG}{\mathrm{BG}}
\newcommand{\LS}{\mathrm{LS}}
\newcommand{\Supp}{\mathrm{Supp}}

\newcommand{\bra}[1]{[\![#1]\!]}
\newcommand{\pra}[1]{(\!(#1)\!)}
\newcommand{\pair}[2]{\langle #1, #2 \rangle}

\newcommand{\J}{J}
\newcommand{\K}{L}
\newcommand{\WJ}{W^{\J}}
\newcommand{\WJs}{W_{\J}}
\newcommand{\DJp}{\Delta^{+} \setminus \Delta_{\J}^{+}}
\newcommand{\DJs}{\Delta_{\J}^{+}}
\newcommand{\QJ}{Q_{\J}}

\newcommand{\edge}[1]{ \xrightarrow{\hspace{2pt}#1\hspace{2pt}} }
\newcommand{\Qe}[1]{ \xrightarrow[\mathsf{Q}]{\hspace{2pt}#1\hspace{2pt}} }
\newcommand{\Be}[1]{ \xrightarrow[\mathsf{B}]{\hspace{2pt}#1\hspace{2pt}} }

\newcommand{\sDP}{{\sf (B)}}
\newcommand{\sQDP}{{\sf (QB)}}

\newcommand{\kap}[2]{\kappa(#1,#2)}
\newcommand{\io}[2]{\iota(#1,#2)}

\newcommand{\tb}[1]{\le_{#1}}
\newcommand{\dtb}[1]{\le_{#1}^{\ast}}

\newcommand{\tbmin}[3]{\min(#1W_{#2},\le_{#3}\nobreak)}
\newcommand{\tbmax}[3]{\max(#1W_{#2},\le_{#3}^{\ast}\nobreak)}
\newcommand{\up}[3]{\mathrm{up}(#1,#2W_{#3})}
\newcommand{\dn}[3]{\mathrm{dn}(#1,#2W_{#3})}

\newcommand{\mcr}[1]{\lfloor #1 \rfloor}

\newcommand{\ti}[1]{\widetilde{#1}}
\newcommand{\ha}[1]{\widehat{#1}}

\newcommand{\sprod}{\sideset{}{^\star}\prod}

\newenvironment{enu}{%
 \begin{enumerate}%
}{\end{enumerate}}

\allowdisplaybreaks[4]

\begin{document}

\maketitle


\begin{abstract}
We give a presentation of the torus-equivariant quantum $K$-theory ring of flag manifolds of type $A$, 
as a quotient of a polynomial ring by an explicit ideal. 
This is the torus-equivariant version of our previous result, which gives a presentation of 
the non-equivariant quantum $K$-theory ring of flag manifolds of type $A$.
However, the method of proof for the torus-equivariant one is completely different from that 
for the non-equivariant one; our proof is based on the result in the $Q = 0$ limit, and 
uses Nakayama-type arguments to upgrade it to the quantum situation.
Also, in contrast to the non-equivariant case in which we used the Chevalley formula, 
we make use of the inverse Chevalley formula for the torus-equivariant $K$-group of 
semi-infinite flag manifolds to obtain a relation which yields our presentation.
\end{abstract}


\section{Introduction.} 
\label{sec:intro}
Let $Fl_{n+1}$ denote the (full) flag manifold $G/B$ of type $A_{n}$, 
where $G = SL_{n+1}(\BC)$ is the connected, simply-connected simple algebraic group of type $A_{n}$ 
over the complex numbers $\BC$, with Borel subgroup $B$ consisting of the upper triangular matrices in $G = SL_{n+1}(\BC)$ 
and maximal torus $H \subset B$ consisting of the diagonal matrices in $G = SL_{n+1}(\BC)$. 
The purpose of this paper is to give a presentation of the $H$-equivariant quantum $K$-theory ring 
$QK_{H}(Fl_{n+1}) := K_{H}(Fl_{n+1}) \otimes_{R(H)} R(H)\bra{Q}$, 
defined by Givental \cite{Giv} and Lee \cite{Lee}, 
as a quotient of a polynomial ring by an explicit ideal, 
where $K_{H}(Fl_{n+1}) = \bigoplus_{w \in W} R(H)[\CO^{w}]$ denotes the $H$-equivariant (ordinary) $K$-theory ring of $Fl_{n+1}$ 
with the (opposite) Schubert classes $[\CO^{w}]$ indexed by the elements $w$ of the finite Weyl group $W = S_{n+1}$ of $G = SL_{n+1}(\BC)$ 
as a basis over $R(H)$, and where $R(H)\bra{Q} = R(H)\bra{Q_1, \ldots, Q_{n}}$ denotes the ring of formal power series 
in the Novikov variables $Q_{i} := Q^{\alpha_i^{\vee}}$ corresponding to the simple coroots $\alpha_i^{\vee}$, $1 \leq i \leq n$, 
with coefficients in the representation ring $R(H)$ of $H$;
we will identify the representation ring $R(H)$ with the group algebra $\BZ[P] = \bigoplus_{\nu \in P} \BZ \be^{\nu}$ of 
the weight lattice $P = \sum_{i =1}^{n} \BZ \vpi_{i}$ of $G = SL_{n+1}(\BC)$, where $\vpi_{i}$, $1 \leq i \leq n$, are the fundamental weights.

In \cite[Theorem 50]{LNS}, 
we gave a presentation of the (non-equivariant) quantum $K$-theory ring 
$QK(Fl_{n+1}) := K(Fl_{n+1}) \otimes_{\BZ} \BZ\bra{Q}$ in terms of generators and 
relations, where $K(Fl_{n+1})$ denotes the (non-equivariant) ordinary $K$-theory ring of $Fl_{n+1}$; 
cf. the conjectural presentation of $QK(Fl_{n+1})$, which is cited as \cite[Theorem~3.10]{LM}. 
This presentation was obtained by comparing a Chevalley-type multiplication formula for divisor classes in $QK(Fl_{n+1})$ 
and the corresponding one for quantum Grothendieck polynomials (\cite[Theorem~6.4]{LM}).

However, the strategy of our proof of the presentation of $QK_{H}(Fl_{n+1})$ is completely different from that of $QK(Fl_{n+1})$, 
and is based on a general principle that if one has a set of quantum relations 
such that the $Q = 0$ specialization gives the ideal of the classical defining 
relations, then these quantum relations generate the ideal of the quantum 
relations; see \cite{GMSZ} for the result in the case of Grassmannians. 
The main ingredient in our proof is one quite explicit identity in $QK_{H}(Fl_{n+1})$. In order to establish this identity, 
we first prove the corresponding one in the $H$-equivariant $K$-group $K_{H}(\QG) = \prod_{x \in W_{\af}^{\geq 0}} \BZ[P][\CO_{\QG(x)}]$ (direct product) 
of the semi-infinite flag manifold $\QG$ (see \cite{KNS, Kat2}) associated to $G = SL_{n+1}(\BC)$, with the semi-infinite Schubert classes $[\CO_{\QG(x)}]$ 
labeled by the elements $x \in W_{\af}^{\geq 0}$ of the affine Weyl group $W_{\af} \cong W \ltimes Q^{\vee}$ 
of the form $x = wt_{\xi}$, $w \in W$, $\xi \in Q^{\vee,+}$, as a topological basis over $\BZ[P]$; 
here, $Q^{\vee,+} := \sum_{i = 1}^{n} \BZ_{\geq 0} \alpha_i^{\vee} 
\subset Q^{\vee} := \sum_{i = 1}^{n} \BZ \alpha_i^{\vee}$.
We then transfer it to an identity in $QK_{H}(Fl_{n+1})$ through the $R(H)$-module isomorphism $\Phi$ 
from $QK_{H}(Fl_{n+1})$ onto $K_{H}(\QG)$ respecting the quantum multiplication $\star$ and the tensor product $\otimes$ 
with respect to the line bundle classes associated to anti-dominant fundamental weights $- \vpi_i$, $1 \leq i \leq n$, 
which was obtained in \cite{Kat1, Kat3} (see Section~\ref{sec:relation} for details) on the basis of \cite{BF, IMT} (see also \cite{ACT}). 
Here we should mention that in contrast to \cite{LNS}, our identity in $K_{H}(\QG)$ is 
deduced by using the inverse Chevalley formula in \cite{KNOS} (see Section~\ref{sec:identity} for details). 

To be more precise, we need the following notation. 
For $\xi \in Q^{\vee,+}$, we define a $\BZ[P]$-linear operator 
$\st_{\xi}$ on $K_{H}(\QG)$ by $\st_{\xi}[\CO_{\QG(x)}] :=
[\CO_{\QG(x t_{\xi})}]$ for $x \in W_{\af}^{\geq 0}$; note that 
\begin{equation}
(\st_{\xi}[\CO_{\QG(x)}]) \otimes [\CO_{\QG}(\nu)] = 
\st_{\xi}([\CO_{\QG(x)}] \otimes [\CO_{\QG}(\nu)])
\end{equation}
for $x \in W_{\af}^{\geq 0}$, $\xi \in Q^{\vee,+}$, and $\nu \in P$, 
where $[\CO_{\QG}(\nu)]$ denotes the line bundle class in $K_{H}(\QG)$ 
associated to $\nu \in P$. 
Also, for $k \in \BZ_{\geq 0}$, we set $[k] := \{ 1, 2, \ldots, k \}$, 
and for a subset $J \subset [k]$, we set $\epsilon_{J} := \sum_{j \in J} \eps_{j}$, 
where $\eps_{j} = \vpi_{j} - \vpi_{j-1}$ for $1 \leq j \leq n+1$; note that $\vpi_{0} := 0$, $\vpi_{n+1} := 0$ 
by convention, and hence that $\eps_1 + \cdots + \eps_{n+1} = \vpi_{n+1} = 0$. 
Then, for $0 \leq l \leq k \leq n+1$, we define an element $\FF^{k}_{l}$ of $K_{H}(\QG)$ by:
\begin{equation}
\FF^{k}_{l}:=
 \sum_{
   \begin{subarray}{c}
   J \subset [k] \\[1mm]
   |J|=l
   \end{subarray}}
 \left( \prod_{ j \notin J,\, j+1 \in J }
 (1-\st_{j}) \right)[\CO_{\QG}(\lng \eps_{J})], 
\end{equation}
where $\lng = \begin{pmatrix} n+1 & n & \cdots & 1 \end{pmatrix}$ (in one-line notation) denotes 
the longest element of the finite Weyl group $W = S_{n+1}$; 
note that $\FF^{k}_{0} = 1$ for $0 \le k \le n+1$. 

The following proposition gives an explicit description of the (special, but important) semi-infinite Schubert classes $[\CO_{\QG(s_1 s_2 \cdots s_{k-1} s_k)}] \in K_{H}(\QG)$, $0 \leq k \leq n+1$, in terms of the elements $\FF^k_l$, $0 \leq l \leq k$ (see Proposition~\ref{prop:ckk}). 

\begin{iprop} \label{iprop1}
Let $0 \le k \le n+1$. 
The following equality holds in $K_{H}(\QG)$: 
\begin{equation}
\underbrace{[ \CO_{\QG(s_{1}s_{2} \cdots s_{k-1}s_{k})} ]}_{
\begin{subarray}{c}
\text{\rm This term is understood  } \\[1mm]
\text{\rm to be $0$ if $k=n+1$.} \end{subarray} } 
= \sum_{0 \le l \le k} (-1)^{l} \be^{l\vpi_{1}} \FF^{k}_{l}. 
\end{equation}
\end{iprop}

From this proposition with $k = n+1$, we deduce that the elements $\FF^{n+1}_{l} \in K_{H}(\QG)$, $1 \leq l \leq n+1$, are in fact equivariant parameters, i.e., elements of $\BZ[P]$, given explicitly as follows (see Theorem~\ref{thm:FFn+1}). 
\begin{ithm} \label{ithm1}
For $1 \leq l \leq n+1$, the following equality holds in $K_{H}(\QG)$:
\begin{equation}
\FF^{n+1}_{l} = \sum_{
   \begin{subarray}{c}
   J \subset [n+1] \\[1mm]
   |J|=l
   \end{subarray}} \be^{-\eps_{J}}.
\end{equation}
\end{ithm}
Through the $R(H)$-module isomorphism $\Phi : QK_{H}(Fl_{n+1}) \to K_{H}(\QG)$, established in \cite{Kat1, Kat3}, 
which respects the quantum multiplication $\star$ with the line bundle class $[\CO_{Fl_{n+1}}(- \vpi_{i})]$ and 
the tensor product $\otimes$ with the line bundle class $[\CO_{\QG}(\lng \vpi_{i})]$ for $1 \leq i \leq n$, 
we can transfer the identity above in $K_{H}(\QG)$ to the one below in $QK_{H}(Fl_{n+1})$; 
note that the line bundle $\CO_{Fl_{n+1}}(- \nu)$ over $Fl_{n+1}$ for $\nu \in P$ denotes the $SL_{n+1}(\BC)$-equivariant line bundle constructed as the quotient space $SL_{n+1}(\BC) \times^{B} \BC_{\nu}$ of the product space $SL_{n+1}(\BC) \times \BC_{\nu}$ by the usual (free) left action of $B$, where $B$ is the Borel subgroup of $SL_{n+1}(\BC)$ consisting of the upper triangular matrices in $G$ and $\BC_{\nu}$ is the one-dimensional $B$-module of weight $\nu$. 
For this purpose, we note that in our notation, the map $\Phi$ sends the (opposite) Schubert class 
$\be^{\mu}[\CO^{w}][Q^{\xi}]$ to the semi-infinite Schubert class $\be^{-\mu}[\CO_{\QG(wt_{\xi})}]$ 
for $\mu \in P$, $w \in W = S_{n+1}$, and $\xi \in Q^{\vee,+}$, 
where $Q^{\xi} := \prod_{i=1}^{n} Q_{i}^{k_i}$ for $\xi = \sum_{i=1}^{n} k_{i} \alpha_i^{\vee} \in Q^{\vee,+}$.
Moreover, we know from Section~\ref{sec:relation} that if $G = SL_{n+1}(\BC)$, then for each $1 \leq k \leq n+1$, 
the quantum multiplication $\star$ with the class $\frac{1}{1 - Q_{k}}[\CO_{G/B}(\varepsilon_{k})]$ 
(resp., $\frac{1}{1 - Q_{k-1}}[\CO_{G/B}(- \varepsilon_{k})]$) in $QK_{H}(G/B)$ corresponds 
to the tensor product $\otimes$ with the line bundle class $[\CO_{\QG}(-\lng \eps_{k})]$ 
(resp., $[\CO_{\QG}(\lng \eps_{k})]$) in $K_{H}(\QG)$, where $Q_{0} := 0, Q_{n+1} := 0$ by convention.
Now, for $0 \le p \le k \le n+1$, we set
\begin{equation}
\CF^{k}_{p}:=
 \sum_{
   \begin{subarray}{c}
   J \subset [k] \\[1mm]
   |J|=p
   \end{subarray}} \ 
 \prod_{ \begin{subarray}{c} 1 \le j \le k \\[1mm] j,\,j+1 \in J \end{subarray} }
 \frac{1}{1-Q_{j}} \ 
 \sprod_{j \in J} [\CO_{G/B}(-\eps_{j})] \in QK_{H}(Fl_{n+1}), 
\end{equation}
where $\prod^{\star}$ denotes the product with respect to the quantum multiplication $\star$; 
note that $\CF^{k}_{0} = 1$ for $0 \le k \le n+1$. 
Then it follows that 
$\Phi(\CF^{k}_{p}) = \FF^{k}_{p} \in K_{H}(\QG)$ for $0 \le p \le k \le n+1$. 
Hence we obtain the following (see Theorem~\ref{thm:rel}).
%
%
\begin{icor} \label{icor2}
For $0 \leq l \leq n+1$, 
the following equality holds in $QK_{H}(Fl_{n+1})$: 
\begin{equation}
\CF^{n+1}_{l}= 
\sum_{
   \begin{subarray}{c}
   J \subset [n+1] \\[1mm]
   |J|=l
   \end{subarray}}
\be^{\eps_{J}}.
\end{equation}
\end{icor}
Once the identity above in $QK_{H}(Fl_{n+1})$ is obtained, we can deduce from it 
the following presentation of $QK_{H}(Fl_{n+1})$ by generators and relations 
by Nakayama-type arguments based on the well-known presentation of 
the ordinary $K$-theory ring $K_{H}(Fl_{n+1})$ by generators and relations; 
we follow the same line of arguments as those in \cite{GMSZ}. 
To be more precise, let $\CI$ denote the ideal of 
$R(H)[x_{1},\dots,x_{n},x_{n+1}]$ generated by 
\begin{equation*}
\sum_{
   \begin{subarray}{c}
   J \subset [n+1] \\[1mm]
   |J|=l
   \end{subarray}} \, 
\prod_{j \in J} (1-x_{j}) - 
\sum_{
   \begin{subarray}{c}
   J \subset [n+1] \\[1mm]
   |J|=l
   \end{subarray}}\be^{\eps_{J}}
\quad
\text{for $1 \le l \le n+1$}. 
\end{equation*}
It is well-known (see \cite[Introduction]{FL}; cf. \cite{PR99}) that there exists an $R(H)$-algebra isomorphism 
$\Psi$ from the quotient $R(H)[x_{1},\dots,x_{n},x_{n+1}]/\CI$ onto 
$K_{H}(Fl_{n+1})$ which maps the residue class of $1 - x_{j}$ modulo $\CI$ to 
$[\CO_{Fl_{n+1}}(-\eps_{j})]$ for $1 \le j \le n+1$; 
note that for $1 \leq j \leq n+1$, the line bundle $\CO_{Fl_{n+1}}(- \epsilon_{j})$ is just the quotient bundle $\mathcal{U}_{j}/\mathcal{U}_{j-1}$ over $Fl_{n+1}$, where $0 = \mathcal{U}_{0} \subset \mathcal{U}_{1} \subset \cdots \subset \mathcal{U}_{n} \subset \mathcal{U}_{n+1} = Fl_{n+1} \times \BC^{n+1}$ denotes the universal, or tautological, flag of subvector bundles of the trivial bundle $Fl_{n+1} \times \BC^{n+1}$. 
Now, let $\CI^{Q}$ be the ideal of 
$R(H)\bra{Q}[x_{1},\dots,x_{n},x_{n+1}]$ 
generated by 
\begin{equation*}
\sum_{
   \begin{subarray}{c}
   J \subset [n+1] \\[1mm]
   |J|=l
   \end{subarray}}
\prod_{\begin{subarray}{c} 
  1 \le j \le n+1 \\[1mm]
  j \in J,\,j+1 \notin J
  \end{subarray}} (1-Q_{j})
\prod_{j \in J}(1-x_{j}) - 
\sum_{
   \begin{subarray}{c}
   J \subset [n+1] \\[1mm]
   |J|=l
   \end{subarray}}\be^{\eps_{J}}
\quad
\text{for $1 \le l \le n+1$},
\end{equation*}
where we understand that $1 - Q_{n+1} = 1$. 
Based on the result above in the classical limit (i.e., the limit $Q_{i} = 0$, 
$1 \leq i \leq n$), Nakayama-type arguments yield the following presentation 
of $QK_{H}(Fl_{n+1})$ (see Theorem~\ref{thm:main}); 
for Nakayama-type arguments, we first need to show that the quotient ring $R(H)\bra{Q}[x_{1},\dots,x_{n},x_{n+1}]/\CI^{Q}$ is finitely generated as a module over $R(H)\bra{Q}$ (see Corollary~\ref{cor:fg} of Appendix~\ref{sec:B}). 
\begin{ithm} \label{ithm3}
There exists an $R(H)\bra{Q}$-algebra isomorphism 
\begin{equation*}
\Psi^{Q}: R(H)\bra{Q}[x_{1},\dots,x_{n},x_{n+1}]/\CI^{Q} 
\stackrel{\sim}{\rightarrow} 
QK_{H}(Fl_{n+1})
\end{equation*}
which maps the residue class of $(1-Q_{j})(1 -x_{j})$ 
modulo $\CI^{Q}$ to $[\CO_{Fl_{n+1}}(-\eps_{j})]$ for $1 \le j \le n$, 
and the residue class of $1 - x_{n+1}$ modulo $\CI^{Q}$ to $[\CO_{Fl_{n+1}}(-\eps_{n+1})]$. 
\end{ithm}

Also, in \cite[Theorem~50]{LNS}, we proved that quantum Grothendieck polynomials, 
introduced in \cite[Section~3]{LM}, represent the corresponding (opposite) Schubert classes 
in the non-equivariant quantum $K$-theory ring $QK(Fl_{n+1})$ under the presentation 
given in \cite[Theorem~50]{LNS}; namely, we proved \cite[Conjecture~7.1]{LM}. 
In \cite{MaNS}, we prove the $H$-equivariant version of this result. 
Namely, we prove that quantum double Grothendieck polynomials, 
introduced in \cite[Section~8]{LM}, 
represent the corresponding (opposite) Schubert classes 
in the $H$-equivariant quantum $K$-theory ring $QK_{H}(Fl_{n+1})$ 
under the presentation above. However, our proof in the $H$-equivariant case is completely different from 
that in the non-equivariant case, since we can make use of 
(quantum) left divided difference operators 
(i.e., (quantum) left Demazure operators) $\delta_{i}^{\vee}$ 
acting only on equivariant parameters, as proposed in \cite[Section~8]{MNS}.
Indeed, for our purpose, we first identify 
the (opposite) Schubert class $[\CO^{\lng}]$ in $QK_{H}(Fl_{n+1})$ 
associated to the longest element 
$\lng = \begin{pmatrix} n+1 & n & \cdots & 1 \end{pmatrix}$ 
(in one-line notation) 
of the finite Weyl group $W = S_{n+1}$ 
with the corresponding quantum double Grothendieck polynomial $\FG^{q}_{\lng}(y, x)$, 
given in \cite[Definition~8.2]{LM}. 
Next, by successively applying left Demazure operators, 
we can identify the (opposite) Schubert class $[\CO^{w^{-1}}] \in QK_{H}(Fl_{n+1})$ 
associated to an arbitrary element $w \in W = S_{n+1}$ 
with the corresponding quantum double Grothendieck polynomial $\FG^{q}_{w}(y, x)$ (cf. \cite[Corollary~8.10]{LM}); 
here we note that the quantum double Grothendieck polynomial $\FG^{q}_{w}(y, x)$ is obtained 
from $\FG^{q}_{w_{\circ}}(y, x)$ by successively applying 
(quantum) left Demazure operators $\pi_{i}^{(y)} = \delta_{i}^{\vee}$ 
acting only on the $y$-variables (under the identification $y_{i} = 1 - \be^{- \eps_{i}}$). 
\subsection*{Acknowledgments.}
The second and third authors would like to 
thank Cristian Lenart and Daniel Orr for related collaborations.
The second author would like to thank Leonardo C. Mihalcea for valuable discussions 
on the presentation of the quantum $K$-theory ring of Grassmannians, 
which inspired this work. 
The first and second authors thank Takafumi Kouno for pointing out a gap in the proof of Corollary~\ref{cor:fg} in an earlier version of this paper. 
S.N. was partly supported by JSPS Grant-in-Aid for Scientific Research (C) 21K03198. 
D.S. was partly supported by JSPS Grant-in-Aid for Scientific Research (C) 19K03145.
%
%
\section{Basics.} \label{sec:not}
In this section, we fix our basic notation, and also recall some basic notions and facts, 
which will be used in this paper; 
$G$ is a connected, simply-connected simple algebraic group over $\BC$, 
and is not necessarily assumed to be of type $A_{n}$ unless stated explicitly.

%
\subsection{Algebraic groups.} \label{subsec:alggrp}
Let $G$ be a connected, simply-connected simple algebraic group over $\BC$, 
$H$ a maximal torus of $G$. 
Set $\Fg := \mathrm{Lie}(G)$ and $\Fh := \mathrm{Lie}(H)$. 
Thus $\Fg$ is a finite-dimensional simple Lie algebra over $\BC$ and 
$\Fh$ is a Cartan subalgebra of $\Fg$. 
We denote by $\pair{\cdot\,}{\cdot} : \Fh^{\ast} \times \Fh \rightarrow \BC$ 
the canonical pairing, where $\Fh^{\ast} = \Hom_{\BC}(\Fh, \BC)$. 

It is known that $\Fg$ has a root system $\Delta \subset \Fh^{\ast}$. 
We take the set $\Delta^{+} \subset \Delta$ of positive roots, 
and the set $\{ \alpha_{i} \}_{i \in I} \subset \Delta^{+}$ of simple roots. 
We denote by $\alpha^{\vee} \in \Fh$ the coroot corresponding to $\alpha \in \Delta$. 
Also, we denote by $\theta \in \Delta^+$ the highest root of $\Delta^+$, 
and set $\rho := (1/2) \sum_{\alpha \in \Delta^{+}} \alpha$. 
The root lattice $Q$ and the coroot lattice $Q^{\vee}$ of $\Fg$ are defined by 
$Q := \sum_{i \in I} \BZ \alpha_{i}$ and $Q^{\vee} := \sum_{i \in I} \BZ \alpha_{i}^{\vee}$. 

For $i \in I$, the weight $\vpi_{i} \in \Fh^{\ast}$ 
which satisfies $\pair{\vpi_{i}}{\alpha_{j}^{\vee}} = \delta_{i,j}$ for all $j \in I$, 
where $\delta_{i,j}$ denotes the Kronecker delta, 
is called the fundamental weight. 
The weight lattice $P$ of $\Fg$ is defined by $P := \sum_{i \in I} \BZ \vpi_{i}$. 
We denote by $\BZ[P]$ the group algebra of $P$, that is, 
the associative algebra generated by formal elements $\be^{\nu}$, $\nu \in P$, 
where the product is defined by $\be^{\mu} \be^{\nu} := \be^{\mu + \nu}$ for $\mu, \nu \in P$. 

A reflection $s_{\alpha} \in GL(\Fh^{\ast})$, $\alpha \in \Delta$, 
is defined by $s_{\alpha} \mu := \mu - \pair{\mu}{\alpha^{\vee}} \alpha$ 
for $\mu \in \Fh^{\ast}$. We write $s_{i} := s_{\alpha_{i}}$ for $i \in I$. 
Then the (finite) Weyl group $W$ of $\Fg$ is defined to be the subgroup of $GL(\Fh^{\ast})$ 
generated by $\{ s_{i} \}_{i \in I}$, that is, $W := \langle s_{i} \mid i \in I \rangle$. 
For $w \in W$, there exist $i_{1}, \ldots, i_{r} \in I$ such that $w = s_{i_{1}} \cdots s_{i_{r}}$. 
If $r$ is minimal, then the product $s_{i_{1}} \cdots s_{i_{r}}$ is called a reduced expression for $w$, 
and $r$ is called the length of $w$; we denote by $\ell(w)$ the length of $w$. 
Note that a reduced expression for $w$ is not unique. However, the length is defined uniquely. 
Also, the affine Weyl group $W_{\af}$ of $\Fg$ is, by definition, the semi-direct product group 
$W \ltimes \{ t_{\xi} \mid \xi \in Q^{\vee} \}$ of $W$ and the abelian group 
$\{ t_{\xi} \mid \xi \in Q^{\vee} \} \cong Q^{\vee}$, 
where $t_{\xi}$ denotes the translation in $\Fh^{\ast}$ corresponding to $\xi \in Q^{\vee}$.

%
\subsection{The Bruhat graph and the quantum Bruhat graph.}
\label{subsec:QBG}

Let $\J$ be a subset of $I$. We set
$\QJ := \sum_{i \in \J} \BZ \alpha_i$, 
$\DJs := \Delta^{+} \cap \QJ$, 
$\rho_{\J}:=(1/2) \sum_{\alpha \in \DJs} \alpha$, and 
$\WJs := \langle s_{i} \mid i \in \J \rangle$. 
Let $\WJ$ denote the set of minimal(-length) coset representatives 
for the cosets in $W/\WJ$; we know from \cite[Sect.~2.4]{BB} that 
%
%
\begin{equation} \label{eq:mcr}
\WJ = \bigl\{ w \in W \mid 
\text{$w \alpha \in \Delta^{+}$ for all $\alpha \in \DJs$}\bigr\}.
\end{equation}
For $w \in W$, let $\mcr{w}^{\J} \in \WJ$ denote 
the minimal coset representative for the coset $w \WJs$ in $W/\WJs$.

\begin{dfn} \label{dfn:QBG} 
The (parabolic) \emph{quantum Bruhat graph} of $\WJ$, 
denoted by $\QBG(\WJ)$, is the ($\DJp$)-labeled
directed graph whose vertices are the elements of $\WJ$ and 
whose edges are of the following form: 
$x \edge{\alpha} y$, with $x, y \in \WJ$ and $\alpha \in \DJp$, 
such that $y = \mcr{x s_{\alpha}}^{\J}$ and either of the following holds: 
(B) $\ell(y) = \ell (x) + 1$; 
(Q) $\ell(y) = \ell (x) + 1 - 2 \pair{\rho-\rho_{\J}}{\alpha^{\vee}}$.
An edge satisfying (B) (resp., (Q)) is called a \emph{Bruhat edge} (resp., \emph{quantum edge}). 
When $\J=\emptyset$ (note that $W^{\emptyset}=W$, $\rho_{\emptyset}=0$, and $\mcr{x}^{\emptyset}=x$ for all $x \in W$), 
we write $\QBG(W)$ for $\QBG(W^{\emptyset})$.
\end{dfn}

For an edge $x \edge{\alpha} y$ in $\QBG(\WJ)$, 
we sometimes write $x \Be{\alpha} y$ (resp., $x \Qe{\alpha} y$) 
to indicate that the edge is a Bruhat (resp., quantum) edge. 

\begin{dfn} \label{dfn:BG} 
The (parabolic) \emph{Bruhat graph} of $\WJ$, 
denoted by $\BG(\WJ)$, is the subgraph of $\QBG(\WJ)$ 
with the same vertex set as $\QBG(\WJ)$ but having only 
the Bruhat edges. 
When $\J=\emptyset$, we write $\BG(W)$ for $\BG(W^{\emptyset})$.
\end{dfn}

For a directed path $\bp : x_{0} \edge{\gamma_{1}} x_{1} 
\edge{\gamma_{2}} \cdots \edge{\gamma_{s}} x_{s}$ in $\QBG(W)$, 
we set $\ell(\bp) := s$, and
\begin{align*}
\wt(\bp) := \sum_{\substack{1 \le i \le s \\ 
   \text{$x_{i-1} \edge{\gamma_{i}} x_{i}$ is a quantum edge}}} \gamma_{i}^{\vee}. 
\end{align*}
For $w, v \in W$, there exists a directed path in $\QBG(W)$ which starts from $w$ and ends at $v$ (see \cite[Lemma~1]{Postnikov}). 
Hence we can take a shortest(-length) directed path $\bp$ from $w$ to $v$ in $\QBG(W)$. 
We define $\ell(w \Rightarrow v):=\ell(\bp)$ and $\wt(w \Rightarrow v) := \wt(\bp)$. 
The definition of $\wt(w \Rightarrow v)$ does not 
depend on the choice of $\bp$ (see \cite[Lemma~1]{Postnikov}). 

Now, let $\lhd$ be a reflection order on $\Delta^{+}$. 
A directed path $\bp : x_{0} \edge{\gamma_{1}} x_{1} 
\edge{\gamma_{2}} \cdots \edge{\gamma_{s}} x_{s}$ in $\QBG(W)$ or in $\BG(W)$ 
is said to be label-increasing in the reflection order $\lhd$ if 
the sequence $\gamma_{1},\gamma_{2},\dots,\gamma_{s}$ of labels in $\bp$ 
is strictly increasing in $\lhd$. 
%
%
\begin{prop}[{\cite{BFP}; see also \cite[Theorem~7.3]{LNSSS1}}] \label{prop:QBG-LI}
Let $\lhd$ be an arbitrary reflection order on $\Delta^{+}$. 
For each pair of elements $v,w \in W$, there exists a unique label-increasing 
directed path from $v$ to $w$ with respect to $\lhd$ in the quantum Bruhat graph $\QBG(W)$. 
Moreover, this label-increasing directed path is a shortest(-length) directed path from $v$ to $w$. 
\end{prop}
%
%
\begin{prop}[{see \cite[Lemma~2.7.4]{BB}}] \label{prop:BG-LI}
Let $\lhd$ be an arbitrary reflection order on $\Delta^{+}$. 
For each pair of elements $v,w \in W$ such that $v \le w$ in the Bruhat order $\le$, 
there exists a unique label-increasing directed path from $v$ to $w$ 
with respect to $\lhd$ in the Bruhat graph $\BG(W)$. 
\end{prop}

Denote by $\le$ the (ordinary) Bruhat order on $W$; 
recall that for $v,w \in W$, $v \le w$ if and only if 
there exists a directed path from $v$ to $w$ in $\BG(W)$. 
%
%
\begin{lem} \label{lem:BG-QBG}
Let $v,w \in W$, and let $\bp$ be a shortest(-length) directed path 
from $v$ to $w$ in $\QBG(W)$. Then, all the edges in $\bp$ are Bruhat edges 
{\rm(}or equivalently, $\bp$ is a directed path in $\BG(W)${\rm)} if and 
only if $v \le w$ in the Bruhat order $\le$ on $W$. 
\end{lem}

\begin{proof}
The ``only if'' part is obvious by the comment preceding this lemma. 
Let us show the ``if'' part. Assume that $v \le w$. Then 
there exists a directed path $\bp'$ from $v$ to $w$ in $\BG(W)$; 
notice that $\ell(w)=\ell(v)+\ell(\bp')$. It is obvious that 
$\bp'$ is also a directed path in $\QBG(W)$ 
since $\BG(W)$ is a subgraph of $\QBG(W)$. Let $\bp$ be a shortest(-length) 
directed path from $v$ to $w$ in $\QBG(W)$. Since 
$\ell(\bp) \le \ell(\bp')$, it follows that 
\begin{equation*}
\ell(w) \le \ell(v) + \ell(\bp) \le 
\ell(v) + \ell(\bp') = \ell(w). 
\end{equation*}
Hence we obtain $\ell(w) = \ell(v) + \ell(\bp)$, which implies that 
all the edges in $\bp$ are Bruhat edges. This proves the lemma. 
\end{proof}

%
\subsection{Minimal or maximal elements of cosets in the tilted Bruhat order.}
\label{subsec:tilted}
Let $\J$ be a subset of $I$. 
%
%
\begin{dfn}[tilted Bruhat order] \label{dfn:tilted}
For each $v \in W$, we define the $v$-tilted Bruhat order $\tb{v}$ on $W$ as follows:
for $w_{1},w_{2} \in W$, 
%
%
\begin{equation} \label{eq:tilted}
w_{1} \tb{v} w_{2} \iff \ell(v \Rightarrow w_{2}) = 
 \ell(v \Rightarrow w_{1}) + \ell(w_{1} \Rightarrow w_{2}).
\end{equation}
Namely, $w_{1} \tb{v} w_{2}$ if and only if 
there exists a shortest(-length) directed path 
from $v$ to $w_{2}$ in $\QBG(W)$ passing through $w_{1}$; 
or equivalently, if and only if 
the concatenation of a shortest(-length) directed path 
from $v$ to $w_{1}$ and one from $w_{1}$ to $w_{2}$ 
is one from $v$ to $w_{2}$. 
\end{dfn}
%
%
\begin{prop}[{\cite[Theorem~7.1]{LNSSS1}}] \label{prop:tbmin}
Let $\J$ be a subset of $I$, and $v \in W$. 
Then, each coset $u\WJ$ for $u \in W$ has a unique minimal element
with respect to $\tb{v}${\rm;} we denote it by $\tbmin{u}{\J}{v}$. 
\end{prop}

%
\begin{dfn}[dual tilted Bruhat order] \label{dfn:dtilted}
For each $v \in W$, we define the dual $v$-tilted Bruhat order 
$\dtb{v}$ on $W$ as follows:
for $w_{1},w_{2} \in W$, 
%
%
\begin{equation} \label{eq:dtilted}
w_{1} \dtb{v} w_{2} \iff \ell(w_{1} \Rightarrow v) = 
 \ell(w_{1} \Rightarrow w_{2}) + \ell(w_{2} \Rightarrow v).
\end{equation}
Namely, $w_{1} \dtb{v} w_{2}$ if and only if 
there exists a shortest(-length) directed path 
from $w_{1}$ to $v$ in $\QBG(W)$ passing through $w_{2}$; 
or equivalently, if and only if the concatenation of a shortest(-length) directed path 
from $w_{1}$ to $w_{2}$ and one from $w_{2}$ to $v$ 
is one from $w_{1}$ to $v$. 
\end{dfn}
%
%
\begin{prop}[{\cite[Proposition~2.25]{NOS}}] \label{prop:tbmax}
Let $v \in W$. 
Then, each coset $u\WJs$ for $u \in W$ has a unique maximal element
with respect to $\dtb{v}$\,{\rm;} 
we denote it by $\tbmax{u}{\J}{v}$.
\end{prop}

Recall that $\le$ denotes the (ordinary) Bruhat order on $W$. 
For $w \in W$ and $x \in \WJ$ such that $x \ge \mcr{w}^{\J}$, 
the set $\bigl\{ y \in x\WJs \mid y \ge w \bigr \}$ has a unique 
minimal element in the Bruhat order $\le$, which is denoted by $\up{w}{x}{\J}$
(see, e.g., \cite[Proposition~3.1\,(1)]{LS}). Similarly, 
for $w \in W$ and $x \in \WJ$ such that $x \le \mcr{w}^{\J}$, 
the set $\bigl\{ y \in x\WJs \mid y \le w \bigr \}$ has a unique 
maximal element in the Bruhat order, which is denoted by $\dn{w}{x}{\J}$
(see, e.g., \cite[Proposition~3.1\,(2)]{LS}). 

%
\begin{lem} \label{lem:Deo}
Keep the notation and setting above. 
\begin{enu}
\item Let $w \in W$ and $x \in \WJ$ be such that $x \ge \mcr{w}^{\J}$ in the Bruhat order. 
Then, we have $\tbmin{x}{\J}{w} \ge w$ and $\tbmin{x}{\J}{w} = \up{w}{x}{\J}$. 

\item Let $w \in W$ and $x \in \WJ$ be such that $x \le \mcr{w}^{\J}$ in the Bruhat order.
Then, we have $\tbmax{x}{\J}{w} \le w$ and $\tbmax{x}{\J}{w} = \dn{w}{x}{\J}$. 
\end{enu}
\end{lem}
\begin{proof}
We give a proof only for part (2), since part (1) follows from part (2) and 
\cite[Proposition~2.25]{NOS}. 
We take and fix a reflection order $\lhd$ such that 
%
%
\begin{equation} \label{eq:refa}
\beta \lhd \gamma \quad 
\text{for all $\beta \in \DJs$ and $\gamma \in \DJp$}.
\end{equation}
We know (see, e.g., \cite[Lemma 2.15]{KoNS}) that 
$w \in W$ and $x' \in x\WJs$ satisfies the condition that $x'=\tbmax{x}{\J}{w}$ 
if and only if all the labels in the label-increasing directed path 
from $x'$ to $w$ in the quantum Bruhat graph $\QBG(W)$ 
are contained in $\DJp$; in particular, $\wt (x' \Rightarrow w) = 0$ or 
$\wt (x' \Rightarrow w) \notin \sum_{i \in \J} \BZ \alpha_{i}^{\vee}$. 
Also, we deduce from \cite[Lemma~7.2]{LNSSS2} that 
$\wt (x' \Rightarrow w) \equiv \wt(x \Rightarrow \mcr{w}^{\J})$ modulo 
$\sum_{i \in \J} \BZ \alpha_{i}^{\vee}$. Since $x \le \mcr{w}^{\J}$ in the Bruhat order, 
we have $\wt(x \Rightarrow \mcr{w}^{\J}) = 0$, which implies that 
$\wt (x' \Rightarrow w) \in \sum_{i \in \J} \BZ \alpha_{i}^{\vee}$. 
Hence we obtain $\wt (x' \Rightarrow w) = 0$, which implies that $\tbmax{x}{\J}{w}=x' \le w$. 

Now, by the definition of $\dn{w}{x}{\J}$, we have 
$x' \le \dn{w}{x}{\J}$. Since $\dn{w}{x}{\J} \in x\WJs$, we see that 
$\dn{w}{x}{\J} \dtb{w} x'$. It follows from the definition of $\dtb{w}$ that 
\begin{equation*}
\wt (\dn{w}{x}{\J} \Rightarrow w) = \wt (\dn{w}{x}{\J} \Rightarrow x') + \wt(x' \Rightarrow w).
\end{equation*}
Here we have $\wt(x' \Rightarrow w) = 0$, as seen above. 
Also, since $\dn{w}{x}{\J} \le w$ in the Bruhat order, 
we have $\wt (\dn{w}{x}{\J} \Rightarrow w) = 0$. Hence we see that 
$\wt (\dn{w}{x}{\J} \Rightarrow x') = 0$, which implies that 
$\dn{w}{x}{\J} \le x'$. Combining these inequalities, we obtain 
$\dn{w}{x}{\J} = x' = \tbmax{x}{\J}{w}$, as desired. 
\end{proof}
%
%
\subsection{Quantum Lakshmibai-Seshadri paths.}
\label{subsec:QLS}

Let $\nu \in P^{+}$, and set 
\begin{equation} \label{eq:J}
\J=\J_{\nu}:= 
\bigl\{ i \in I \mid \pair{\nu}{\alpha_{i}^{\vee}}=0 \bigr\}.
\end{equation}
%
%
\begin{dfn} \label{dfn:QBa}
For a rational number $0 < a < 1$, 
we define $\QBG_{a\nu}(\WJ)$ (resp., $\BG_{a\nu}(\WJ)$) 
to be the subgraph of $\QBG(\WJ)$ (resp., $\BG(\WJ)$) 
with the same vertex set but having only those directed edges of 
the form $x \edge{\alpha} y$ for which 
$a\pair{\nu}{\alpha^{\vee}} \in \BZ$ holds.
\end{dfn}
%
%
\begin{dfn}[{\cite{Lit94} and \cite{Lit95}}] \label{dfn:LS}
A \emph{Lakshmibai-Seshadri path} (LS path for short) 
of shape $\nu$ is a pair 
%
%
\begin{equation} \label{eq:LS}
\eta = (\bw \,;\, \ba) = 
(w_{1},\,\dots,\,w_{s} \,;\, a_{0},\,a_{1},\,\dots,\,a_{s}), \quad s \ge 1, 
\end{equation}
of a sequence $w_{1},\,\dots,\,w_{s}$ 
of elements in $\WJ$, with $w_{l} \ne w_{l+1}$ 
for any $1 \le l \le s-1$, and an increasing sequence 
$0 = a_0 < a_1 < \cdots  < a_s =1$ of rational numbers 
satisfying the condition that there exists a directed path 
in $\BG_{a_{l}\nu}(\WJ)$ from $w_{l+1}$ to  $w_{l}$ 
for each $l = 1,\,2,\,\dots,\,s-1$. 
\end{dfn}

Let $\LS(\nu)$ denote 
the set of all Lakshmibai-Seshadri paths of shape $\nu$. 
%
%
\begin{dfn}[{\cite[Section~3.2]{LNSSS2}}] \label{dfn:QLS}
A \emph{quantum Lakshmibai-Seshadri path} (QLS path for short) 
of shape $\nu$ is a pair 
%
%
\begin{equation} \label{eq:QLS}
\eta = (\bw \,;\, \ba) = 
(w_{1},\,\dots,\,w_{s} \,;\, a_{0},\,a_{1},\,\dots,\,a_{s}), \quad s \ge 1, 
\end{equation}
of a sequence $w_{1},\,\dots,\,w_{s}$ 
of elements in $\WJ$, with $w_{l} \ne w_{l+1}$ 
for any $1 \le l \le s-1$, and an increasing sequence 
$0 = a_0 < a_1 < \cdots  < a_s =1$ of rational numbers 
satisfying the condition that there exists a directed path 
in $\QBG_{a_{l}\nu}(\WJ)$ from $w_{l+1}$ to  $w_{l}$ 
for each $l = 1,\,2,\,\dots,\,s-1$. 
\end{dfn}

Let $\QLS(\nu)$ denote 
the set of all quantum LS paths of shape $\nu$; 
note that $\LS(\nu) \subset \QLS(\nu)$. 
For $\eta \in \QLS(\nu)$ of the form \eqref{eq:QLS}, 
we set $\kappa(\eta):=w_{s} \in \WJ$, and 
\begin{equation} \label{eq:wt}
\wt (\eta) := \sum_{l=1}^{s} (a_{l}-a_{l-1}) w_{l}\nu \in P, 
\end{equation}
\begin{equation} \label{eq:deg}
\deg(\eta):= - \sum_{l=1}^{s-1} a_{l} \pair{\nu}{\wt(w_{l+1} \Rightarrow w_{l})} \in \BZ_{\le 0}; 
\end{equation}
see, e.g., \cite[Remark~3.13]{NOS}. 
Also, following \cite[(3.26) and (3.27)]{NOS}, 
for $\eta \in \QLS(\nu)$ of the form \eqref{eq:QLS} and $v \in W$, 
we define $\kap{\eta}{v} \in W$ by the following recursive formula: 
%
%
\begin{equation} \label{eq:haw}
\begin{cases}
\ha{w}_{0}:=v, & \\[2mm]
\ha{w}_{l}:=\tbmax{w_{l}}{\J_{\nu}}{\ha{w}_{l-1}} & \text{for $1 \le l \le s$}, \\[2mm]
\kap{\eta}{v}:=\ha{w}_{s}, 
\end{cases}
\end{equation}
and then we set
%
%
\begin{equation} \label{eq:zetav}
\zeta(\eta,v):= 
\sum_{l=0}^{s-1} \wt (\ha{w}_{l+1} \Rightarrow \ha{w}_{l}). 
\end{equation}
Similarly, following \cite[(C.2)--(C.4)]{NOS},
we define $\io{\eta}{v} \in W$ 
by the following recursive formula: 
%
%
\begin{equation} \label{eq:tiw}
\begin{cases}
\ti{w}_{s+1}:=v, & \\[2mm]
\ti{w}_{l}:=\tbmin{w_{l}}{\J_{\mu}}{\ti{w}_{l+1}} & \text{for $1 \le l \le s$}, \\[2mm]
\io{\eta}{v}:=\ti{w}_{1}, 
\end{cases}
\end{equation}
and then we set
%
%
\begin{equation} \label{eq:xiv}
\xi(\eta,v):= 
\sum_{l=1}^{s} \wt (\ti{w}_{l+1} \Rightarrow \ti{w}_{l}), 
\end{equation}
%
%
%
\begin{equation} \label{eq:degx}
\deg_{v}(\eta):= - \sum_{l=1}^{s} a_{l} \pair{\nu}{\wt(\ti{w}_{l+1} \Rightarrow \ti{w}_{l})}. 
\end{equation}

Recall that $\theta$ is the highest root of $\Delta^{+}$, with 
$\theta^{\vee}$ its coroot.
%
%
\begin{lem} \label{lem:QLS=LS}
For $i \in I$, 
$\QLS(\vpi_{i}) = \LS(\vpi_{i})$ if and only if $\pair{\vpi_{i}}{\theta^{\vee}}=1$. 
In particular, if $\vpi_{i}$ is minuscule, or if $\Fg$ is of type $C$, then 
$\QLS(\vpi_{i}) = \LS(\vpi_{i})$. 
\end{lem}

\begin{proof}
Let $i \in I$; recall that $\J=\J_{\vpi_{i}}=I \setminus \{i\}$. 
First we prove the ``if'' part. 
Let $i \in I$ be such that $\pair{\vpi_{i}}{\theta^{\vee}}=1$. 
Suppose, for a contradiction, that $\QLS(\vpi_{i}) \supsetneq \LS(\vpi_{i})$. 
Then, there exist $v,w \in \WJ$, $\beta \in \Delta^{+} \setminus \Delta_{\J}^{+}$, 
and $a \in \BQ$ with $0 < a < 1$ such that $w \edge{\beta} v$ is a quantum edge 
in $\QBG(\WJ)$ and $a\pair{\vpi_{i}}{\beta^{\vee}} \in \BZ_{> 0}$; 
we necessarily have $\pair{\vpi_{i}}{\beta^{\vee}} \ge 2$. Also, 
it follows from \cite[(2)' in Section 4.3 and (2) in Section 4.2]{LNSSS1} that 
$\beta$ is a quantum root in the sense of \cite[Lemma~4.2]{LNSSS1}.
Here we claim that $\theta^{\vee}-\beta^{\vee} \in Q^{\vee,+}$. 
If $\Fg$ is simply-laced, or if $\beta$ is a long root, then 
the claim is well-known. 
Assume that $\Fg$ is not simply-laced, and $\beta$ is a short root. 
We deduce from \cite[Lemma~4.2\,(2)]{LNSSS1} that 
the Dynkin subdiagram corresponding to 
the support $\Supp(\beta) := \bigl\{ i \in I \mid 
\pair{\vpi_{i}}{\beta^{\vee}} > 0 \bigr\}$ of $\beta$ is of type $A$. 
Hence it follows that $\beta = \sum_{j \in \Supp(\beta)} \alpha_{j}$, and 
$\beta^{\vee} = \sum_{j \in \Supp(\beta)} \alpha_{j}^{\vee}$ 
(see, e.g., \cite[(5.1.1)]{Kac}). 
Therefore, also in this case, we see that $\theta^{\vee} - \beta^{\vee} \in Q^{\vee,+}$, as desired. 
From this claim, we obtain 
$\pair{\vpi_{i}}{\theta^{\vee}} \ge \pair{\vpi_{i}}{\beta^{\vee}} \ge 2$, 
which is a contradiction. 
This proves the ``if'' part. 

Next we prove the ``only if'' part. 
Assume that $b:=\pair{\vpi_{i}}{\theta^{\vee}} \ge 2$.
Since $\theta$ is a quantum root by \cite[Lemma~4.2\,(1)]{LNSSS1}, 
we see that $s_{\theta} \edge{\theta} e$ is a quantum edge in $\QBG_{(1/b)\vpi_{i}}(W)$. 
It follows from \cite[Lemma~6.2]{LNSSS2} that there exists a directed path from 
$\mcr{ s_{\theta} }^J$ to $\mcr{ e }^J = e$ in $\QBG_{(1/b)\vpi_{i}}(\WJ)$. 
Hence we deeduce that $\eta \in (e, \mcr{s_{\theta}}^J \,;\,0,1/b,1) \in \QLS(\vpi_{i})$, 
but $\eta \notin LS(\vpi_{i})$. Thus we have $\QLS(\vpi_{i}) \supsetneq \LS(\vpi_{i})$. 
This proves the ``only if'' part, and completes the proof of the lemma. 
\end{proof}

%
\subsection{The root system of type $A$.}

We recall the root system of type $A$. In the rest of this paper, if we assume that $G$ is of type $A$, 
then we use the notation introduced in this subsection. 

Assume that $G$ is of type $A_{n}$, i.e., $G = SL_{n+1}(\BC)$. 
Then $\Fg = \Fsl_{n+1}(\BC)$, 
and $\Fh := \{ h \in \Fg \mid \text{$h$ is a diagonal matrix}\}$ is 
a Cartan subalgebra of $\Fg$. 
We let $\{\eps_{k} \mid 1 \le k \le n+1\}$ be the standard basis of $\BZ^{n+1}$ 
and realize the weight lattice as $P=\BZ^{n+1}/\BZ(\eps_1+\cdots+\eps_n+\eps_{n+1})$. 
By abuse of notation, we continue to denote the image of $\eps_k$ in $P$ by the same symbol. 
Thus $\varpi_k := \eps_1+\dotsm+\eps_k$, for $k \in \{1, \ldots, n\}$, are the fundamental weights of $\Fg$, 
and $\eps_1+\cdots+\eps_n+\eps_{n+1} = 0$. 
We set $\alpha_{i} := \eps_{i} - \eps_{i+1}$ for $i \in \{ 1, \ldots, n \}$ 
and $\alpha_{i, j} := \alpha_{i} + \alpha_{i+1} + \cdots + \alpha_{j}=\eps_{i}-\eps_{j+1}$ for $i, j \in \{1, \ldots, n\}$ with $i \le j$. 
Then $\Delta = \{ \pm \alpha_{i, j} \mid 1 \le i \le j \le n \}$ forms a root system of $\Fg$, 
with the set of positive roots $\Delta^+ = \{ \alpha_{i, j} \mid 1 \le i \le j \le n \}$ 
and the set of simple roots $\{ \alpha_{1}, \ldots, \alpha_{n} \}$. 

Let us review the Weyl group of $\Fg=\Fsl_{n+1}(\BC)$. 
By the definition of the Weyl group, we have $W = \langle s_{1}, \ldots, s_{n} \rangle$, 
where $s_{1}, \ldots, s_{n}$ are simple reflections. 
For $i,j \in \{ 1, \ldots, n+1 \}$ with $i < j$, 
we denote by $(i,j)$ the transposition of $i$ and $j$. 
It is known that the correspondence 
$s_{1} \mapsto (1, 2),\, 
 s_{2} \mapsto (2, 3),\,\dots,\,
 s_{n} \mapsto (n, n+1)$ 
defines a group isomorphism $W \xrightarrow{\sim} S_{n+1}$, 
where $S_{n+1}$ denotes the symmetric group of degree $n+1$. 
With this isomorphism, we regard $x \in W$ as a permutation on $\{ 1, \ldots, n+1 \}$. 
By this identification, we see that $w \eps_{i} = \eps_{w(i)}$ for $i \in \{ 1, \ldots, n+1 \}$. 
Also, the longest element of $W$, denoted by $\lng$, is regarded as the permutation
\begin{equation*}
\begin{pmatrix}
1 & 2 & \cdots & n & n+1 \\ 
n+1 & n & \cdots & 2 & 1
\end{pmatrix},
\end{equation*}
that is, $\lng$ is considered to be the permutation defined 
by $\lng(k) = n+2-k$ for $k \in \{ 1, \ldots, n+1 \}$. 

%
\section{Formulas for the class of a line bundle in $K_{H}(\QG)$.}
\label{sec:formulalb}
%
%
In this section, $G$ is a connected, simply-connected simple algebraic group over $\BC$, 
and is not necessarily assumed to be of type $A_{n}$ unless stated explicitly.

\subsection{The $H$-equivariant $K$-group of semi-infinite flag manifolds.}
\label{subsec:equivgroup}

Let $\QGr$ denote the semi-infinite flag manifold associated to $G$, 
which is an ind-scheme of infinite type whose set of 
$\BC$-valued points is $G(\BC\pra{z})/(H(\BC) \cdot N(\BC\pra{z}))$ 
(see \cite{KNS, Kat2} for details), 
where $G$ is 
a connected, simply-connected simple algebraic group over $\BC$, $B=HN$ is a Borel subgroup, 
$H$ is a maximal torus, and $N$ is the unipotent radical of $B$; 
note that $\QGr$ is defined as an inductive limit of copies of the (reduced) closed subscheme $\QG \subset \prod_{i \in I} \mathbb{P}(L(\varpi_{i}) \otimes_{\BC} \BC\bra{z})$ of infinite type (introduced in \cite[Section~4.1]{FM}), 
where $L(\varpi_{i})$ is the irreducible highest weight $G$-module of highest weight $\varpi_{i}$. 
One has the semi-infinite Schubert (sub)variety 
$\QG(x) \subset \QGr$ associated to each element $x$ of 
the affine Weyl group $W_{\af} \cong W \ltimes Q^{\vee}$, 
with $W = \langle s_i \mid i \in I \rangle$ 
the (finite) Weyl group and $Q^{\vee} = \sum_{i \in I} \BZ \alpha_i^{\vee}$ 
the coroot lattice of $G$; note that $\QG(x)$ is, by definition, the closure of the orbit under the Iwahori subgroup $\bI \subset G(\BC\bra{z})$, 
which is the pre-image of $B$ under the evaluation map $G(\BC\bra{z}) \to G$ at $z = 0$, 
through the ($H \times \BC^{*}$)-fixed point labeled by $x \in W_{\af}$ (in the same way as in \cite[Section~4.2]{KNS} and \cite[Section~2.3]{Orr}), and that $\QG(x)$ is contained in $\QG(e) = \QG$ for all $x \in W_{\af}^{\geq 0} := \{ x = w t_{\xi} \in W_{\af} \mid w \in W, \xi \in Q^{\vee,+} \}$, where $Q^{\vee, +} := \sum_{i \in I} \BZ_{\geq 0} \alpha_{i}^{\vee} \subset Q^{\vee}$. 
Also, for each weight $\nu = \sum_{i \in I} m_{i} \varpi_{i} \in P$ with $m_{i} \in \BZ$, we have a $G(\BC\bra{z}) \rtimes \BC^{*}$-equivariant line bundle $\CO_{\QG}(\nu)$ over $\QG$, which is given by the restriction of the line bundle $\boxtimes_{i \in I} \CO(m_{i})$ on $\prod_{i \in I} \mathbb{P}(L(\varpi_{i}) \otimes_{\BC} \BC\bra{z})$. 

The $(H \times \BC^{\ast})$-equivariant $K$-group $K_{H \times \BC^{\ast}}(\QG)$ 
is a module over $\BZ[q,q^{-1}][P]$ (equivariant parameters), 
with the semi-infinite Schubert classes $[\CO_{\QG(x)}]$ associated to 
$x \in W_{\af}^{\geq 0} \simeq W \times Q^{\vee,+}$ as a topological basis (in the sense of \cite[Proposition~5.11]{KNS}) over $\BZ[q,q^{-1}][P]$, 
where $P = \sum_{i \in I} \BZ \varpi_{i}$ is the weight lattice of $G$, $\BZ[P] = \bigoplus_{\nu \in P} \BZ\be^{\nu} \cong R(H)$, 
and $q \in R(\BC^{*})$ corresponds to loop rotation. 
More precisely, $K_{H \times \BC^{\ast}}(\QG)$ is defined to be the $\BZ[q, q^{-1}][P]$-submodule 
of the Laurent series (in $q^{-1}$) extension 
$\BZ\pra{q^{-1}}[P] \otimes_{\BZ\bra{q^{-1}}[P]} K_{\bI \rtimes \BC^{*}}^{\prime}(\QG)$ 
of the equivariant, with respect to the Iwahori subgroup $\bI$ and loop rotation, 
$K$-group $K_{\bI \rtimes \BC^{*}}^{\prime}(\QG)$ (see \cite[Section~5]{KNS}) consisting of 
all convergent (in the sense of \cite[Proposition~5.11]{KNS}) infinite linear combinations of the classes $[\CO_{\QG(x)}]$, $x \in W_{\af}^{\geq 0}$, 
of the structure sheaves $\CO_{\QG(x)}$ of the semi-infinite Schubert varieties $\QG(x) \subset \QG$ 
with coefficients $a_{x} \in \BZ[q, q^{-1}][P]$; 
briefly speaking, convergence holds if the sum $\sum_{x \in W_{\af}^{\geq 0}} \vert a_{x} \vert$ 
of the absolute values $|a_{x}|$ lies in $\BZ_{\geq 0}[P]\pra{ q^{-1} }$. 
For each $x \in W_{\af}^{\geq 0}$ and $\nu \in P$, it follows from \cite[Corollary~5.12]{KNS} and \cite[Theorem~5.16]{KLN} that the twisted semi-infinite Schubert class $[\CO_{\QG(x)}(\nu)] \in K_{H \times \BC^{\ast}}(\QG)$ corresponding to the tensor product sheaf $\CO_{\QG(x)} \otimes \CO_{\QG}(\nu)$ lies in $K_{H \times \BC^{\ast}}(\QG)$; 
in particular, we have $[\CO_{\QG}] \otimes [\CO_{\QG}(\nu)] = [\CO_{\QG}(\nu)] \in K_{H \times \BC^{\ast}}(\QG)$ for all $\nu \in P$. 

Now, we first define the $H$-equivariant $K$-group $K_{H}(\QG)$ of the semi-infinite flag manifold $\QG$ to be the $\BZ[P]$-module 
$\prod_{x \in W_{\af}^{\geq 0}} \BZ[P][\CO_{\QG(x)}]$ (direct product), which 
consists of all (formal) infinite linear combinations of 
the semi-infinite Schubert classes $[\CO_{\QG(x)}]$, $x \in W_{\af}^{\geq 0}$, 
with coefficients in $\BZ[P]$. 
Note that by definition, the semi-infinite Schubert classes $[\CO_{\QG(x)}]$, $x \in W_{\af}^{\geq 0}$, form a topological basis of $K_{H}(\QG)$; that is, an arbitrary element of $K_{H}(\QG)$ can be written uniquely as an infinite linear combination of the $[\CO_{\QG(x)}]$, $x \in W_{\af}^{\geq 0}$, with coefficients in $\BZ[P]$. Then, for each $\nu \in P$, a $\BZ[P]$-linear operator $\bullet \otimes [\CO_{\QG}(\nu)]$ on $K_{H}(\QG)$ is induced (without ambiguity) from the $\BZ[q, q^{-1}][P]$-linear operator $\bullet \otimes [\CO_{\QG}(\nu)]$ on $K_{H \times \BC^{\ast}}(\QG)$ by the specialization at $q = 1$ (of coefficients). 
Finally, we set $[\CO_{\QG}(\nu)] := [\CO_{\QG}] \otimes [\CO_{\QG}(\nu)]$ for $\nu \in P$; that is, the element $[\CO_{\QG}(\nu)] \in K_{H}(\QG)$ is defined to be the image of $[\CO_{\QG}]$ by the $\BZ[P]$-linear operator $\bullet \otimes [\CO_{\QG}(\nu)]$. 
In addition, for $\xi \in Q^{\vee,+}$, a $\BZ[q, q^{-1}][P]$-linear operator 
$\st_{\xi}$ on $K_{H \times \BC^{\ast}}(\QG)$ can be defined by 
$\st_{\xi}[\CO_{\QG(x)}] := 
[\CO_{\QG(x t_{\xi})}]$ for $x \in W_{\af}^{\geq 0}$ (cf. \cite[Proposition~2.4]{Orr}); note that in our notation, we have in $K_{H \times \BC^{\ast}}(\QG)$, 
\begin{equation*}
(\st_{\xi}[\CO_{\QG(x)}]) \otimes [\CO_{\QG}(\nu)] = 
q^{-\langle \nu, -w_{\circ} \xi \rangle} \st_{\xi}([\CO_{\QG(x)}] \otimes [\CO_{\QG}(\nu)]) 
\end{equation*}
for $x \in W_{\af}^{\geq 0}$ and $\nu \in P$; this equation also follows directly from \cite[Proposition D.1]{KNS}. 
Hence, for each $\xi \in Q^{\vee,+}$, the $\BZ[P]$-linear operator $\st_{\xi}$, given by $\st_{\xi}[\CO_{\QG(x)}] = [\CO_{\QG(x t_{\xi})}]$ for $x \in W_{\af}^{\geq 0}$, on $K_{H}(\QG)$ is induced by the specialization $q = 1$, and we have 
\begin{equation}\label{eq:shift}
(\st_{\xi}[\CO_{\QG(x)}]) \otimes [\CO_{\QG}(\nu)] = 
\st_{\xi}([\CO_{\QG(x)}] \otimes [\CO_{\QG}(\nu)])
\end{equation}
for $x \in W_{\af}^{\geq 0}$ and $\nu \in P$. 
Since the semi-infinite Schubert classes $[\CO_{\QG(x)}]$, $x \in W_{\af}^{\geq 0}$, 
form a topological basis of $K_{H}(\QG)$ over $\BZ[P]$ in the sense above, it follows that
\begin{equation}
(\st_{\xi} \, \bullet) \otimes [\CO_{\QG}(\nu)] = 
\st_{\xi}(\bullet \otimes [\CO_{\QG}(\nu)])
\end{equation}
for an arbitrary element $\bullet \in K_{H}(\QG)$ and $\xi \in Q^{\vee,+}$, $\nu \in P$.

%
%
\subsection{Formulas for the class of a line bundle in $K_{H}(\QG)$.}
\label{subsec:fomulalb}

Let $\K$ be an arbitrary subset of $I$, and set 
%
%
\begin{equation} \label{eq:mu}
\mu:=\sum_{i \in \K} \vpi_{i} \in P^{+};
\end{equation}
note that $\J=\J_{\mu}=
\bigl\{ i \in I \mid \pair{\mu}{\alpha_{i}^{\vee}}=0 \bigr\}$ 
is identical to $I \setminus \K$. 

In this subsection, we prove a formula which expresses 
the line bundle class $[\CO_{\QG}(\lng \mu)]$ associated to 
the weight $\lng \mu$ as an explicit finite sum of semi-infinite 
Schubert classes in $K_{H}(\QG)$;
this formula is a generalization of the formula 
$[\CO_{\QG}(\lng \varpi_{i})] = \be^{-\varpi_{i}}([\CO_{\QG(e)}] - [\CO_{\QG(s_{i})}])$, 
obtained in \cite{NOS}. 
Namely, we prove the following.
%
%
\begin{prop} \label{prop:e2}
Let $\mu \in P^{+}$ be of the form \eqref{eq:mu}. 
Then, in $K_{H \times \BC^{*}}(\QG)$ {\rm(}and hence in $K_{H}(\QG)${\rm)}, 
the following equality holds\,{\rm:}
\begin{equation} \label{eq:e2b}
[\CO_{\QG}(\lng \mu)] = 
\sum_{v \in W_{\K}} 
 (-1)^{\ell(v)} \be^{-\mu} [\CO_{\QG(v)}]. 
\end{equation}
\end{prop}

\begin{rem}
Equation \eqref{eq:e2b} holds in $K_{H \times \BC^{*}}(\QG)$ without the specialization at $q = 1$; that is, even in $K_{H \times \BC^{*}}(\QG)$, the equivariant parameter $q$ does not appear on the right-hand side of the equation. 
This is because the equivariant parameter $q$ does not appear on the right-hand side of equation \eqref{eq:e1} below (see also the proof of Lemma~\ref{lem:e1} below).
\end{rem}

In order to prove this proposition, we need some lemmas. 
%
%
\begin{lem} \label{lem:e}
Let $\mu \in P^{+}$ be of the form \eqref{eq:mu}.
If $\eta \in \QLS(\mu)$ satisfies $\kappa(\eta) = e$, 
then $\eta = (e \,;\,0,1)$. 
\end{lem}

\begin{proof}
Let $\eta = (w_{1},\,\dots,w_{s-1},w_{s}\,;\,\sigma_{0},\,\dots,\,\sigma_{s-1},\sigma_{s}) \in \QLS(\mu)$ 
be such that $\kappa(\eta) = w_{s} = e$. 
Suppose, for a contradiction, that $s > 1$; note that 
$0=\sigma_{0} < \sigma_{s-1} < \sigma_{s} = 1$. 
By the definition of quantum Lakshmibai-Seshadri paths 
of shape $\mu$, there exists a directed path 
\begin{equation*}
w_{s} = x_{0} \edge{\beta_{1}} x_{1} \edge{\beta_{2}} \cdots \edge{\beta_{u}} x_{u}=w_{s-1}
\end{equation*}
from $w_{s}$ to $w_{s-1}$ in the parabolic quantum Bruhat graph 
$\QBG(\WJ)$ satisfying the condition that 
$\sigma_{s-1}\pair{\mu}{\beta_{t}^{\vee}} \in \BZ_{> 0}$ 
for all $1 \le t \le u$. 
Since $w_{s} = e$ by the assumption, we see that 
the first edge $e = x_{0} \edge{\beta_{1}} x_{1}$ is a Bruhat edge, and 
$\beta_{1}=\alpha_{i}$ for some $i \in \K$. 
Since $\mu$ is of the form \eqref{eq:mu}, 
we have $\sigma_{s-1}\pair{\mu}{\beta_{1}^{\vee}} = 
\sigma_{s-1}\pair{\mu}{\alpha_{i}^{\vee}} = \sigma_{s-1} \notin \BZ$, 
which is a contradiction. Therefore, we deduce that $s=1$, and hence 
$\eta = (e \,;\,0,1)$, as desired. 
\end{proof}

We take and fix a reflection order $\lhd$ satisfying \eqref{eq:refa}, 
with $\J = \J_{\mu}=I \setminus \K$.
%
%
\begin{lem} \label{lem:e1}
Let $\mu \in P^{+}$ be of the form \eqref{eq:mu}. 
Then, in $K_{H \times \BC^{\ast}}(\QG)$ {\rm(}and hence in $K_{H}(\QG)${\rm)}, 
the following equality holds\,{\rm:}
\begin{equation} \label{eq:e1}
[\CO_{\QG}(\lng \mu)] = 
\sum_{\begin{subarray}{c} v \in W \\[1mm] \text{\sDP} \end{subarray} } 
 (-1)^{\ell(v)} \be^{-\mu} [\CO_{\QG(v)}], 
\end{equation}
where condition $\text{\sDP}$ is as follows\,{\rm:}
\begin{quote}
\sDP \ all the labels in the {\rm(}unique{\rm)} 
label-increasing directed path from $e$ to $v$ in the Bruhat graph $\BG(W)$ 
{\rm(}see Proposition~\ref{prop:BG-LI}{\rm)} are contained in $\DJp$. 
\end{quote}
\end{lem}

\begin{proof}
By \cite[Theorem~9.1]{NOS} (with $\lambda=-\lng \mu$ and $x=w=e$, $\xi=0$), we have
\begin{align}
[\CO_{\QG}(\lng \mu)] & = [\CO_{\QG}(\lng \mu)] \otimes [\CO_{\QG(e)}] \nonumber \\
& = \sum_{v \in W} \sum_{ 
    \begin{subarray}{c} \eta \in \QLS(\mu) \\ \kappa(\eta,v)=e \end{subarray} }
    (-1)^{\ell(v)-\ell(e)} q^{-\deg(\eta)} \be^{-\wt(\eta)}
    [ \CO_{\QG(vt_{\zeta(\eta,v)})} ] \label{eq:e1z}
\end{align}
in $K_{H \times \BC^{*}}(\QG)$. 
Let $\eta \in \QLS(\mu)$ be such that $\kappa(\eta,v)=e$ for some $v \in W$. 
We see by the definition \eqref{eq:haw} of $\kappa(\eta,v)$ that 
$\kappa(\eta) = \mcr{\kappa(\eta,v)}^J = e$. 
Hence it follows that $\eta= (e \,;\,0,1)$ by Lemma~\ref{lem:e}. 
For $\eta = (e \,;\,0,1)$, we have $\wt(\eta)=\mu$ by \eqref{eq:wt}, and 
$\deg (\eta) = 0$ by \eqref{eq:deg}; 
this is the reason why the equivariant parameter $q$ does not appear on the right-hand side of equation \eqref{eq:e1a} below, 
and hence on the right-hand side of the asserted equation of the lemma. 
Also, we have $\kap{\eta}{v}=\tbmax{e}{\J}{v}$ by \eqref{eq:haw}, and 
$\zeta(\eta,v) = \wt(e \Rightarrow v) = 0$ by \eqref{eq:zetav}. 
Substituting these equalities into \eqref{eq:e1z}, we obtain 
\begin{equation} \label{eq:e1a}
[\CO_{\QG}(\lng \mu)] = 
\sum_{ \begin{subarray}{c} v \in W \\[1mm] \tbmax{e}{\J}{v}=e \end{subarray}}
\underbrace{(-1)^{\ell(v)-\ell(e)}}_{=(-1)^{\ell(v)}} 
 \be^{-\mu} [\CO_{\QG(vt_{\wt (e \Rightarrow v)})}]. 
\end{equation}
We know (see, e.g., \cite[Lemma~2.15]{KoNS}) that 
$v \in W$ satisfies the condition that $\tbmax{e}{\J}{v}=e$ 
if and only if the following condition is satisfied: 
\begin{quote}
\sQDP \ all the labels in the (unique) label-increasing directed path 
from $e$ to $v$ in the quantum Bruhat graph $\QBG(W)$ 
(see Proposition~\ref{prop:QBG-LI}) are contained in $\DJp$. 
\end{quote}
By using Lemma~\ref{lem:BG-QBG} and the fact that 
a (label-increasing) directed path in $\BG(W)$ is a (label-increasing) directed path in $\QBG(W)$, 
we deduce that $v$ satisfies \sQDP \ if and only if $v$ satisfies \sDP. 
This proves the lemma.
\end{proof}

\begin{proof}[Proof of Proposition~\ref{prop:e2}]
By Lemma~\ref{lem:e1}, it suffices to show that 
\begin{equation} \label{eq:e2a}
\bigl\{v \in W \mid \text{\rm $v$ satisfies \sDP} \bigr\} = W_{\K}. 
\end{equation}
Let $w_{\K,\circ}$ denote the longest element in $W_{\K}$. 
Since $\J=I \setminus \K$, we see that 
$w_{\K,\circ}\beta \in \Delta^{+}$ for all $\beta \in \DJs$, 
which implies that $w_{\K,\circ} \in \WJ$, and hence 
$\mcr{\lng}^{\J} = uw_{\K,\circ}$ 
for some $u \in W$ such that $\ell(\mcr{\lng}^{\J})=\ell(uw_{\K,\circ}) = 
\ell(u)+\ell(w_{\K,\circ})$. 
Hence there exists a reflection order $\lhd$ such that 
\begin{equation} \label{eq:e2c}
\beta \lhd \gamma \lhd \zeta \quad 
\text{for all $\beta \in \DJs$, 
$\gamma \in (\DJp) \setminus \Delta_{\K}^{+}$, 
and $\zeta \in \Delta_{\K}^{+}$}. 
\end{equation}

First we show the inclusion $\supset$ in \eqref{eq:e2a}. Let $v \in W_{\K}$. 
Observe that the restriction of the reflection order $\lhd$ above to 
$\Delta_{\K}^{+}$ is a reflection order on $\Delta_{\K}^{+}$. 
Hence, by Proposition~\ref{prop:BG-LI}, 
there exist $\beta_{1},\,\dots,\,\beta_{k} \in \Delta_{\K}^{+}$, 
with $\beta_{1} \lhd \cdots \lhd \beta_{k}$, such that 
$e \edge{\beta_{1}} \cdots \edge{\beta_{k}} v$ 
in the Bruhat graph $\BG(W_{\K})$ for $W_{\K}$, 
and hence in the Bruhat graph $\BG(W)$. 
Thus we conclude that $v$ satisfies \sDP; 
recall that $\J=I \setminus \K$. 
Next we show the opposite inclusion $\subset$ in \eqref{eq:e2a}. 
Assume that $v \in W$ satisfies \sDP, and let 
\begin{equation*}
e \edge{\beta_{1}} \cdots \edge{\beta_{k}} v
\end{equation*}
be the label-increasing directed path from $e$ to $v$ 
in the Bruhat graph $\BG(W)$; note that $\beta_{1},\,\dots,\,\beta_{k} \in \DJp$. 
We deduce that the label $\beta_{1}$ of the first edge is $\alpha_{i}$ 
for some $i \in \K$; in particular, $\beta_{1} \in \Delta_{\K}$, 
which implies that $\beta_{l} \in \Delta_{\K}$ for all $1 \le l \le k$ 
by \eqref{eq:e2c}. Therefore, we obtain $v=s_{\beta_{1}} \cdots s_{\beta_{k}} \in W_{\K}$, 
as desired. This completes the proof of Proposition~\ref{prop:e2}. 
\end{proof}

Let $\mu \in P^{+}$ be of the form \eqref{eq:mu}, and let $m \in I$. 
In $K_{H \times \BC^{*}}(\QG)$, we have
\begin{align*}
& [\CO_{\QG}(\lng (\mu-\vpi_{m}))] = \sum_{v \in W_{\K}} 
 (-1)^{\ell(v)} \be^{-\mu} [\CO_{\QG(v)}(-\lng \vpi_{m})] \qquad \text{by \eqref{eq:e2b}} \\[3mm]
& = \sum_{ n \in \BZ_{\ge 0} } \sum_{v \in W_{\K}} (-1)^{\ell(v)} \be^{-\mu} 
   \sum_{ \eta \in \QLS(\vpi_{m}) } \be^{\wt(\eta)} q^{\deg_{v}(\eta) - n} 
   [ \CO_{\QG( \io{\eta}{v} t_{\xi(\eta,v)} + n \alpha_{m}^{\vee} )} ] \\
& \hspace*{100mm} \text{by \cite[Corollary~C.3]{NOS}}.
\end{align*} 
In the following, we assume that $m$ is an element of 
$\J=\J_{\mu}=I \setminus \K$ such that 
$\QLS(\vpi_{m}) = \LS(\vpi_{m})$ (see Lemma~\ref{lem:QLS=LS}).
Let $\eta = (w_{1},\,\dots,\,w_{s} \,;\, 
a_{0},\,a_{1},\,\dots,\,a_{s}) \in \QLS(\vpi_{m}) = \LS(\vpi_{m})$ and $v \in W_{\K}$. 
Since $m \in I \setminus \K$, we have $\K \subset I \setminus \{m\}$, and hence 
$\mcr{v}^{I \setminus \{m\}} = e$. 
Since $w_{s} \ge e = \mcr{v}^{I \setminus \{m\}}$, we see by Lemma~\ref{lem:Deo} that 
$\ti{w}_{s} = \up{v}{w_{s}}{I \setminus \{m\}} \ge v$ in the Bruhat order. 
Also, since $\eta \in \LS(\vpi_{m})$, it follows that $w_{s-1} > w_{s} = \mcr{\ti{w}_{s}}^{I \setminus \{m\}}$ in the Bruhat order. 
Therefore, by Lemma~\ref{lem:Deo}, we deduce that $\ti{w}_{s-1} = \up{\ti{w}_{s}}{w_{s-1}}{I \setminus \{m\}} \ge \ti{w}_{s-1}$ 
in the Bruhat order. Similarly, we deduce that 
$\ti{w}_{l} = \up{\ti{w}_{l+1}}{w_{l}}{I \setminus \{m\}} \ge \ti{w}_{l+1}$ for all $1 \le l \le s$; 
in particular, we have $\io{\eta}{v} = \ti{w}_{1} = \upp(v,\eta)$ in the notation of \cite[\S3.1]{LS}. 
Also, since $\ti{w}_{l} \ge \ti{w}_{l+1}$ in the Bruhat order for all $1 \le l \le s$, we have
$\deg_{v}(\eta) = 0$ and $\xi(\eta,v)=0$. Therefore, we conclude that
in $K_{H \times \BC^{*}}(\QG)$,
%
%
\begin{equation} \label{eq:a}
\begin{split}
& [\CO_{\QG}(\lng (\mu-\vpi_{m}))] = \\
& \hspace*{10mm}
  \sum_{n \in \BZ_{\ge 0}} 
  \sum_{v \in W_{\K}} (-1)^{\ell(v)} q^{-n} \be^{-\mu} 
   \sum_{\eta \in \LS(\vpi_{m})} \be^{\wt(\eta)} 
  [\CO_{\QG( \upp(v,\eta) t_{ n\alpha_{m}^{\vee} } ) }]. 
\end{split}
\end{equation}
From this equality, by specializing (coefficients) at $q = 1$, 
we obtain the following. 
%
%
\begin{prop} \label{prop:mixed} 
Let $\mu \in P^{+}$ be of the form \eqref{eq:mu}, and let $m \in I$ 
be an element of 
$\J=\J_{\mu}=I \setminus \K$ such that 
$\QLS(\vpi_{m}) = \LS(\vpi_{m})$. 
Then, the following equality holds in $K_{H}(\QG)$\,{\rm:}
\begin{equation} \label{eq:equivariant}
\begin{split}
& [\CO_{\QG}(\lng(\mu-\vpi_{m}))] = \\
& \hspace*{10mm}
  \sum_{n \in \BZ_{\ge 0}} 
  \sum_{v \in W_{\K}} (-1)^{\ell(v)} \be^{-\mu} 
  \sum_{\eta \in \LS(\vpi_{m})} \be^{\wt(\eta)} 
  [\CO_{\QG( \upp(v,\eta) t_{ n\alpha_{m}^{\vee} } )}]. 
\end{split}
\end{equation}
\end{prop}

Finally, we will deduce a cancellation-free formula from \eqref{eq:a} 
in the case that $\vpi_{m}$ is minuscule;
in this case, $\QLS(\vpi_{m})=\LS(\vpi_{m})=\bigl\{(w\,;\,0,1) \mid 
w \in W^{I \setminus \{m\}} \bigr\}$. Hence, by \eqref{eq:a}, we have
%
%
\begin{equation} \label{eq:c}
\begin{split}
& [\CO_{\QG}(\lng(\mu-\vpi_{m}))] = \\
& \hspace*{10mm}
  \sum_{n \in \BZ_{\ge 0}} 
  \sum_{v \in W_{\K}} (-1)^{\ell(v)}q^{-n} \be^{-\mu} 
  \sum_{ w \in W^{I \setminus \{m\}} } \be^{w\vpi_{m}}
  [\CO_{\QG( \up{v}{w}{I \setminus \{m\}} t_{n\alpha_{m}^{\vee}} )}].
\end{split}
\end{equation}
%
%
\begin{lem} \label{lem:can1}
Let $w \in W^{I \setminus \{m\}}$. 
If there exists $k \in \K$ such that 
$s_{k}w < w$, then 
\begin{equation*}
\up{v}{w}{I \setminus \{m\}} = \up{s_{k}v}{w}{I \setminus \{m\}} \quad 
\text{\rm for all $v \in W_{\K}$}. 
\end{equation*}
\end{lem}

\begin{proof}
Let us show that for $v \in W_{\K}$, 
\begin{equation*}
\bigl\{ x \in wW_{I \setminus \{m\}} \mid x \ge v \bigr\} = 
\bigl\{ x \in wW_{I \setminus \{m\}} \mid x \ge s_{k}v \bigr\}.
\end{equation*}
To show this assertion, let $v \in W_{\K}$; we may assume that $s_{k}v > v$. 
For the inclusion $\subset$, 
let $x \in \text{LHS}$; since $s_{k}w < w$, we see that 
$s_{k}x < x$. Hence, by \cite[Proposition~2.2.7]{BB}
and the assumption that $x \ge v$, we deduce that $x \ge s_{k}v$, 
which implies that $x \in \text{RHS}$. 

For the opposite inclusion $\supset$, let $x \in \text{RHS}$. 
Since $s_{k}x < x$, as seen above, and 
$s_{k}v > v = s_{k}(s_{k}v)$, 
we deduce that $s_{k}x \ge s_{k}(s_{k}v) = v$
by \cite[Proposition~2.2.7]{BB} 
together with the inequality $x \ge s_{k}v$. Since 
$x > s_{k}x$, we have $x > s_{k}x \ge v$, which implies that 
$x \in \text{LHS}$. This proves the lemma. 
\end{proof}
%
%
\begin{lem} \label{lem:can2}
Assume that $w \in W^{I \setminus \{m\}}$ has a reduced expression 
$w = s_{i_{r}}s_{i_{r-1}} \cdots s_{i_{2}}s_{i_{1}}$ such that 
$\bigl\{ 1 \le t \le r \mid i_{t} \in \K \bigr\} \ne \emptyset$. 
If we set $p:=\max \bigl\{ 1 \le t \le r \mid i_{t} \in \K \bigr\}$, then 
\begin{equation*}
\up{v}{w}{I \setminus \{m\}} = \up{s_{i_{p}}v}{w}{I \setminus \{m\}} \quad 
\text{\rm for all $v \in W_{\K}$}. 
\end{equation*}
\end{lem}

\begin{proof}
We prove the assertion of the lemma by induction on $r-p \ge 0$. 
If $r-p = 0$, then the assertion follows from Lemma~\ref{lem:can1}. 
Assume that $r-p > 0$, and set $w'=s_{i_{r}}w=
s_{i_{r-1}} \cdots s_{i_{2}}s_{i_{1}} \in W^{I \setminus \{m\}}$; 
by our induction hypothesis, we have 
\begin{equation} \label{eq:canih}
\up{v}{w'}{I \setminus \{m\}} = \up{s_{i_{p}}v}{w'}{I \setminus \{m\}} \quad 
\text{\rm for all $v \in W_{\K}$}. 
\end{equation}
Now we claim that for all $v \in W_{\K}$, 
\begin{equation} \label{eq:can1a}
\bigl\{ x \in wW_{I \setminus \{m\}} \mid x \ge v \bigr\} = 
s_{i_{r}} \bigl\{ y \in w' W_{I \setminus \{m\}} \mid y \ge v \bigr\}.
\end{equation}
To show this claim, let $v \in W_{\K}$; 
note that $s_{i_{r}}v > v$ since $i_{r} \notin \K$.
For the inclusion $\subset$, let $x \in \text{LHS}$; 
note that $s_{i_{r}}x \in w' W_{I \setminus \{m\}}$. 
Since $s_{i_{r}}w < w$, we see that $s_{i_{r}}x < x$. 
Hence it follows from \cite[Proposition~2.2.7]{BB}
that $s_{i_{r}}x \ge v$, 
which implies that $x=s_{i_{r}}(s_{i_{r}}x) \in \text{RHS}$. 
For the opposite inclusion $\supset$, 
let $y \in w' W_{I \setminus \{m\}}$ be such that $y \ge v$; 
note that $s_{i_{r}}y \in wW_{I \setminus \{m\}}$. 
Since $w=s_{i_r}w' > w'$, we see that $s_{i_r}y > y$, and hence 
$s_{i_{r}}y > y \ge v$. Hence we conclude that $s_{i_{r}}y \in \text{LHS}$. 
Thus we have shown \eqref{eq:can1a}. Also, we deduce by 
\cite[Proposition~2.2.7]{BB} that 
for $x, x'$ in the set on the LHS of \eqref{eq:can1a}, 
$x \ge x'$ if and only if 
$s_{i_{r}}x = s_{i_{r}}x'$. Hence it follows that
\begin{equation} \label{eq:can2a}
\up{v}{w}{I \setminus \{m\}} = s_{i_{r}} \up{v}{w'}{I \setminus \{m\}} \quad 
\text{\rm for all $v \in W_{\K}$}.
\end{equation}
Here, for $v \in W_{\K}$, we see that
\begin{align*}
\up{v}{w}{I \setminus \{m\}} 
 & = s_{i_{r}} \up{v}{w'}{I \setminus \{m\}} \quad \text{by \eqref{eq:can2a}} \\
 & = s_{i_{r}} \up{s_{i_{p}}v}{w'}{I \setminus \{m\}} 
   \quad \text{by our induction hypothesis \eqref{eq:canih}} \\
 & = \up{s_{i_{p}}v}{w}{I \setminus \{m\}}  \quad \text{by \eqref{eq:can2a}}.
\end{align*}
This proves the lemma. 
\end{proof}

From Lemmas~\ref{lem:can1} and \ref{lem:can2}, we deduce that 
\begin{equation} \label{eq:y}
\begin{split}
& [\CO_{\QG}(\lng(\mu-\vpi_{m}))] = \\
& \hspace*{10mm}
  \sum_{n \in \BZ_{\ge 0}}
  \sum_{v \in W_{\K}} (-1)^{\ell(v)} q^{-n} \be^{-\mu} 
   \sum_{ w \in W^{I \setminus \{m\}} \cap W_{I \setminus \K} } \be^{w\vpi_{m}} 
   [\CO_{\QG ( \up{v}{w}{I \setminus \{m\}} t_{n\alpha_{m}^{\vee}} ) }]. 
\end{split}
\end{equation}
%
%
\begin{lem} \label{lem:can3}
For $w \in W^{I \setminus \{m\}} \cap W_{I \setminus \K}$ and $v \in W_{\K}$, 
we have $\up{v}{w}{I \setminus \{m\}} = wv$. 
\end{lem}

\begin{proof}
Recall that 
$W_{\K} \subset W_{I \setminus \{m\}}$ (since $\K \subset I \setminus \{m\}$), 
and that $\up{v}{w}{I \setminus \{m\}}$ is the minimum element 
of $\bigl\{ y \in w W_{I \setminus \{m\}} \mid y \ge v \bigr\}$
in the Bruhat order. Since $v \in W_{\K} \subset W_{I \setminus \{m\}}$, 
we have $wv \in w W_{I \setminus \{m\}}$. Also, since 
$w \in W^{I \setminus \{m\}}$, we have $\ell(wv) = \ell(w) + \ell(v)$. 
Hence, if $w=s_{i_1} \cdots s_{i_a}$ and $v = s_{j_1} \cdots s_{j_b}$ 
are reduced expression for $w$ and $v$, respectively, then 
$wv= s_{i_1} \cdots s_{i_a}s_{j_1} \cdots s_{j_b}$ is a reduced 
expression for $wv$; in particular, $wv \ge v$ by the Subword Property 
(see \cite[Theorem~2.2.2]{BB}). Therefore, it follows that 
$wv \in \bigl\{ y \in w W_{I \setminus \{m\}} \mid y \ge v \bigr\}$, 
and hence $wv \ge \up{v}{w}{I \setminus \{m\}}$. 

Let us write $\up{v}{w}{I \setminus \{m\}}$ as $wz$, 
with $z \in W_{I \setminus \{m\}}$. By the same reasoning as above, 
if $z = s_{k_1} \cdots s_{k_c}$ is a reduced expression for $z$, then 
$\up{v}{w}{I \setminus \{m\}} = wz = s_{i_1} \cdots s_{i_a}s_{k_1} \cdots s_{k_c}$ 
is a reduced expression. Since $\up{v}{w}{I \setminus \{m\}} = wz \ge v$, 
it follows from the Subword Property that there exists a subsequence 
$l_{1},\dots,l_{d}$ of the sequence $i_1,\dots,i_{a},k_{1},\dots,k_{c}$ such that 
$v = s_{l_1} \cdots s_{l_d}$ is a reduced expression for $v$. 
Here note that 
$l_{1},\dots,l_{d} \in \K$ since $v \in W_{\K}$, 
and that $i_1,\dots,i_{a} \notin \K$ since $w \in W_{I \setminus \K}$. 
Hence $l_{1},\dots,l_{d}$ is a subsequence of $k_{1},\dots,k_{c}$, 
which implies that $\up{v}{w}{I \setminus \{m\}} = wz \ge wv$ by the Subword Property. 
Therefore, we conclude that $\up{v}{w}{I \setminus \{m\}} = wv$, as desired. 
\end{proof}

Substituting the equalities $\up{v}{w}{I \setminus \{m\}} = wv$ for 
$w \in W^{I \setminus \{m\}} \cap W_{I \setminus \K}$ and $v \in W_{\K}$ 
in Lemma~\ref{lem:can3} into equation \eqref{eq:y}, we see that 
%
%
\begin{equation} \label{eq:z}
\begin{split}
& [\CO_{\QG}(\lng(\mu-\vpi_{m}))] = \\
& \hspace*{10mm}
  \sum_{n \in \BZ_{\ge 0}} 
  \sum_{v \in W_{\K}} (-1)^{\ell(v)} q^{-n} \be^{-\mu} 
  \sum_{ w \in W^{I \setminus \{m\}} \cap W_{I \setminus \K} } \be^{w\vpi_{m}} 
  [\CO_{\QG (wv t_{n\alpha_{m}^{\vee}} ) }]. 
\end{split}
\end{equation}
%
From this equality, by specializing (coefficients) at $q = 1$, 
we obtain the following. 
\begin{prop}
Let $\mu \in P^{+}$ be of the form \eqref{eq:mu}, and let $m \in I$ 
be an element of 
$\J=\J_{\mu}=I \setminus \K$ such that 
$\QLS(\vpi_{m}) = \LS(\vpi_{m})$. 
Then, we have the following cancellation-free formula in $K_{H}(\QG)$\,{\rm:}
\begin{equation} \label{eq:z2}
\begin{split}
& [\CO_{\QG}(- \lng(\vpi_{m} - \mu))] = \\
& \hspace*{10mm}
  \sum_{n \in \BZ_{\ge 0}} 
  \sum_{ v \in W_{\K}} (-1)^{\ell(v)} \be^{-\mu} 
  \sum_{ w \in W^{I \setminus \{m\}} \cap W_{I \setminus \K} } \be^{w\vpi_{m}} 
  [\CO_{\QG( wv t_{ n\alpha_{m}^{\vee} } )}]. 
\end{split}
\end{equation}
\end{prop}

\begin{proof}
It remains to prove that there are no cancellations 
on the right-hand side of \eqref{eq:z2}. 
For this, it suffices to show that 
for $n,n' \in \BZ_{\ge 0}$, $v,v' \in W_{\K}$, and 
$w,w' \in W^{I \setminus \{m\}} \cap W_{I \setminus \K}$, 
\begin{equation*}
w v t_{ n \alpha_{m}^{\vee} } = 
w'v' t_{ n'\alpha_{m}^{\vee} } \quad \text{only if} \quad
(n,v,w) = (n',v',w'). 
\end{equation*}
Since $W=W \ltimes Q^{\vee}$, 
it is obvious that $n=n'$. Hence we have $w v = w' v'$. 
Since $w,w' \in W^{I \setminus \{m\}}$ and 
$v,v' \in W_{\K} \subset W_{I \setminus \{m\}}$, 
it follows that $w = \mcr{wv}^{I \setminus \{m\}} = 
\mcr{w'v'}^{I \setminus \{m\}} = w'$, and hence 
$v = v'$. This proves the proposition. 
\end{proof}
%
%
\section{An identity for $K$-theoretic quantum shifted elementary polynomials.}
\label{sec:identity}
%
%
\subsection{Inverse Chevalley formula.}
\label{subsec:inv}

In this subsection, we assume that $G$ is simply-laced, i.e., of type $A$, $D$, or $E$. 
Let $\vpi_{i}$ be a minuscule weight, and set $\J=\J_{\vpi_{i}} = I \setminus \{i\}$. 
We fix $x \in \WJ$, and set $\lambda=x\vpi_{i}$. Let $y \in W$ be a unique element of $W$ 
such that $\mcr{\lng}^{\J} = yx$. Let $x=s_{j_l} \cdots s_{j_1}$ and 
$y=s_{i_1} \cdots s_{i_m}$ be reduced expressions for $x$ and $y$, respectively.
We set
\begin{align*}
& \beta_{r} = s_{j_{l}} s_{j_{l-1}} \cdots s_{j_{r+1}} \alpha_{j_{r}} \quad 
  \text{for $1 \le r \le l$}, \\ 
& \gamma_{r} = s_{i_{m}} s_{i_{m-1}} \cdots s_{i_{r+1}} \alpha_{i_{r}} \quad 
  \text{for $1 \le r \le m$}, 
\end{align*}
and then set
\begin{equation*}
\vec{\eta} := (\eta_{1},\dots,\eta_{l+m}) = 
(\beta_{l},\dots,\beta_{1},\gamma_{1},\gamma_{2},\dots,\gamma_{m}). 
\end{equation*}

\begin{dfn}[{\cite[\S3.3.3]{KNOS}}] \label{dfn:QW}
Let $w \in W$. A sequence $\bw=(w_{0},w_{1},\dots,w_{l+m})$ is called a \emph{quantum walk}
starting from $w$ if $w_{0} = w$, and either of the following holds for each $1 \le r \le l+m$:
\begin{enu}
\item[(i)] $w_{r} = w_{r-1}$; 
\item[(ii)] $w_{r} = s_{\eta_{r}}w_{r-1}$, and there exists a directed edge 
from $w_{r-1}$ to $w_{r}$ in $\QBG(W)$. 
\end{enu}
Let $\walka{\lambda}{w}$ denote the set of all quantum walks starting from $w$. 
\end{dfn}

For $\bw=(w_{0},w_1,\dots,w_{l+m}) \in \walka{\lambda}{w}$, 
let $S^{-}(\bw)$ denote the set of steps $r$, for $1 \le r \le l$, such that 
$w_{r}=w_{r-1}$ and $w_{r-1}^{-1}\eta_{r}$ is a simple root. 
Similarly, let $S^{+}(\bw)$ denote the set of steps $r$, for $l < r \le l+m$, 
such that $w_{r} = w_{r-1}$ and $- w_{r-1}^{-1}\eta_{r}$ is a simple root. 
We set $S(\bw):=S^{-}(\bw)\cup S^{+}(\bw)$
%
%
\begin{dfn}[{\cite[\S3.4.1]{KNOS}}] \label{dfn:tQW}
A \emph{decorated quantum walk} is a pair $(\bw,\bb)$ 
of an element $\bw \in \walka{\lambda}{w}$ and a $\{0,1\}$-valued function $\bb$ on $S(\bw)$. 
Denote by $\twalka{\lambda}{w}$ the set of decorated quantum walks. 
\end{dfn}

For $(\bw,\bb) \in \twalka{\lambda}{w}$, we define
\begin{equation*} 
(-1)^{(\bw,\bb)} := 
 \prod_{ \substack{1 \le r \le l \\[1mm] w_{r} < w_{r-1}} } (-1)
 \prod_{ \substack{l < r \le l+m \\[1mm] w_{r} > w_{r-1}} } (-1)
 \prod_{ r \in S(\bw) }(-1)^{\bb(r)}. 
\end{equation*}
Also, we define the partial weight $\wt_{r}(\bw,\bb)$, $0 \le r \le l+m$, 
of $(\bw,\bb)$ by the following recursive formula: 
\begin{align*}
\wt_0(\bw,\bb) & := 0, \\
\wt_{r}(\bw,\bb) & := \wt_{r-1}(\bw,\bb)+
\begin{cases}
-\bb(r)\lng w_{r-1}^{-1}\eta_{r} & \text{if $r \in S^{-}(\bw)$}, \\
\lng w_{r-1}^{-1}\eta_{r} & \text{if $w_{r} < w_{r-1}$}, \\
0 & \text{otherwise}, 
\end{cases} \qquad \text{for $1\le r \le l$}, \\
\wt_{r}(\bw,\bb) & := \wt_{r-1}(\bw,\bb)+
\begin{cases}
\bb(r)\lng w_{r-1}^{-1}\eta_{r} & \text{if $r \in S^{+}(\bw)$}, \\
\lng w_{r-1}^{-1}\eta_{r} & \text{if $w_{r} < w_{r-1}$}, \\
0 & \text{otherwise}, 
\end{cases} \qquad \text{for $l < r \le l+m$}.
\end{align*}
Define the weight of $(\bw,\bb)$ to be $\wt(\bw,\bb)=\wt_{l+m}(\bw,\bb) \in Q^{+}$. 
Also, we define $\wt^{\vee}(\bw,\bb) \in Q^{\vee,+}$ by the same recursive formula 
as the one for $\wt(\bw,\bb)$ above, with $\eta_{r}$ replaced by its coroot $\eta_{r}^{\vee}$ 
for $1 \le r \le l+m$. 

By specializing at $q = 1$ in \cite[Theorem~3.14]{KNOS}, 
we obtain the following inverse Chevalley formula 
for $\be^{\lambda} [\CO_{\QG(w)}]$, with $\lambda = x \varpi_{i}$ and $w \in W$, 
for the $H$-equivariant $K$-group $K_{H}(\QG)$. 
%
%
\begin{thm}[inverse Chevalley formula for $K_{H}(\QG)$] \label{thm:invchev}
Keep the notation and setting above. 
The following equality holds in the $H$-equivariant $K$-group $K_{H}(\QG)$\,{\rm:} 
%
%
\begin{equation} \label{eq:invchev}
\begin{split}
& \be^{\lambda} [\CO_{\QG(w)}] = \\
& \sum_{(\bw=(w_{0},\dots,w_{l+m}),\,\bb) \in \twalka{\lambda}{w}}
(-1)^{(\bw,\bb)} 
[\CO_{\QG( w_{l+m} t_{-\lng(\wt^{\vee}(\bw,\bb))} )} (- \lng w_{l}^{-1}\lambda + \wt(\bw,\bb))].
\end{split}
\end{equation}
\end{thm}

In the rest of this subsection, we assume that $G$ is of type $A_{n}$, 
i.e., $G = SL_{n+1}(\BC)$, and hence $\Fg = \Fsl_{n+1}(\BC)$. 
Now, let us apply formula \eqref{eq:invchev} to 
the special case that $\lambda=\vpi_{1}$; 
in this case, we have $x=e$ (and hence $l=0$), and 
$y = \mcr{\lng}^{\J_{\vpi_1}} = s_{n}s_{n-1} \cdots s_{2}s_{1}$ 
(and hence $m=n=l+m$). Also, we see that 
\begin{equation} \label{eq:gamma}
\begin{cases}
\gamma_{1} = s_{1}s_{2} \cdots s_{n-1} \alpha_{n} = \alpha_{1} + \cdots + \alpha_{n} = \alpha_{1,n}, \\[1mm]
\gamma_{2} = s_{1}s_{2} \cdots s_{n-2} \alpha_{n-1} = \alpha_{1} + \cdots + \alpha_{n-1} = \alpha_{1,n-1}, \\[1mm]
\cdots \cdots \\[1mm]
\gamma_{n-1} = s_{1} \alpha_{2} = \alpha_{1}+\alpha_{2} = \alpha_{1,2}, \\[1mm]
\gamma_{n} = \alpha_{1}, 
\end{cases}
\end{equation}
and $\vec{\eta} = (\gamma_{1},\gamma_{2},\dots,\gamma_{n}) = (\alpha_{1,n},\alpha_{1,n-1},\dots,\alpha_{1})$.
Note that $S^{-}(\bw)=\emptyset$ for any $\bw \in \walkg{w}$ since $l=0$. 
Hence it follows that $S(\bw)=S^{-}(\bw) \sqcup S^{+}(\bw) = S^{+}(\bw)$. 
Also, for $(\bw,\bb) \in \twalkg{w}$, we have 
\begin{equation*}
(-1)^{(\bw,\bb)} = 
 \prod_{ \substack{1 \le r \le n \\ w_r > w_{r-1}} } (-1)
 \prod_{ r \in S(\bw) }(-1)^{\bb(r)}, 
\end{equation*}
and
\begin{align*}
\wt_0(\bw,\bb) & = 0, \\
\wt_r(\bw,\bb) & = \wt_{r-1}(\bw,\bb)+
\begin{cases}
\bb(u)\lng w_{r-1}^{-1}\eta_r & \text{if $r \in S^{+}(\bw)$}, \\
\lng w_{r-1}^{-1}\eta_r & \text{if $w_r < w_{r-1}$}, \\
0 & \text{otherwise}, 
\end{cases} \qquad \text{for $1 \le r \le n$}.
\end{align*}
By \eqref{eq:invchev}, the following equality holds 
in the $H$-equivariant $K$-group $K_{H}(\QG)$: 
%
%
\begin{equation} \label{eq:invchev2}
\begin{split}
& \be^{\varpi_{1}} [\CO_{\QG(w)}] = \\
& \sum_{(\bw = (w_{0},\dots,w_{n}),\,\bb) \in \twalkg{w}}
(-1)^{(\bw,\bb)} 
[\CO_{\QG( w_{n} t_{-\lng(\wt^{\vee}(\bw,\bb))} )} (- \lng \vpi_{1} + \wt(\bw,\bb))].
\end{split}
\end{equation}

Let us fix $0 \le k \le n$ arbitrarily. 
We rewrite equation \eqref{eq:invchev2} more explicitly in the special case 
that $w = s_{1}s_{2} \cdots s_{k}$. In order to do this, 
we need to determine the set $\walk$.
%
%
\begin{prop} \label{prop:QW1}
Keep the setting above. 
Each element $\bw = (w_{0},w_{1},\ldots,w_{n}) \in \walk$ is of one of 
the following forms\,{\rm:}
\begin{enu}
\item \label{enu1} $w = w_{0} = \cdots = w_{n} = w$\,{\rm;}
\item \label{enu2} $w=w_{0}=\cdots=w_{n-k-1} \Be{\alpha_{k+1}} w_{n-k} = \cdots = w_{n}=ws_{k+1}$\,{\rm;} 
\item \label{enu3} $w=w_{0}=\cdots=w_{n-k} \Qe{\alpha_{k}} \cdots \Qe{\alpha_{m}} 
w_{n-m+1} = \cdots = w_{n} = ws_{k} \cdots s_{m}$ for $1 \le m \le k$. 
\end{enu}
If $k=n$, then there does not exist an element $\bw$ of the form \eqref{enu2}.
If $k=0$, then there does not exist an element $\bw$ of the form \eqref{enu3}. 
\end{prop}

\begin{proof} 
Let $\bw = (w_{0},w_{1},\dots,w_{n}) \in \walk$, 
where $w_{0} = w = s_{1}s_{2} \cdots s_{k}$ by the definition. Since 
\begin{equation*}
s_{\alpha_{1,p}} = s_{1} s_{2} \cdots s_{p-1} s_{p} s_{p-1} \cdots s_{2}s_{1} 
 \qquad \text{for $1 \le p \le n$}, 
\end{equation*}
we see that 
\begin{equation*}
s_{\alpha_{1,p}} w = s_{1} s_{2} \cdots s_{p-1} s_{p} s_{p-1} \cdots s_{k+1} 
 \qquad \text{for $k+1 \le p \le n$}. 
\end{equation*}
It follows that $\ell(s_{\gamma_{t}} w) = \ell(s_{\alpha_{1,n-t+1}} w) > \ell(w) + 1$ 
for $1 \le t \le n-k-1$, and $\ell(s_{\gamma_{n-k}} w) = \ell(w)+1$ (for $t=n-k$). 
Therefore, we have 
\begin{equation*}
\begin{split}
& w=w_{0}=w_{1}=\cdots=w_{n-k-1}, \\
& w_{n-k} \in \bigl\{w=s_{1}s_{2} \cdots s_{k}, \, 
  \underbrace{s_{\gamma_{n-k}} w = s_{1}s_{2} \cdots s_{k}s_{k+1}}_{%
  \text{This is omitted if $k=n$.} } \bigr\}.
\end{split}
\end{equation*}
We divide the proof into two cases: 
in Case 1, we consider the case that $w_{n-k}=s_{\gamma_{n-k}} w$ 
(note that $k < n$ in this case), and prove that $\bw$ is of the form (2); 
in Case 2, we consider the case that $w_{n-k}=w$, and 
prove that $\bw$ is either of the form (1) or of the form (3). 
%
\paragraph{Case 1.}
%
Assume that ($k < n$, and)
$w_{n-k} = s_{\gamma_{n-k}} w = s_{1}s_{2} \cdots s_{k}s_{k+1}$. 
If $k=0$, then $\bw = (w_{0},w_{1},\dots,w_{n}) = (e,\dots,e,s_{1})$, 
which is of the form (2) (with $k=0$). 
Assume that $k \ge 1$. We see that for $n-k+1 \le t \le n$, 
\begin{align*}
s_{\gamma_{t}}w_{n-k} & = 
( s_{1} s_{2} \cdots s_{n-t} s_{n-t+1} s_{n-t} \cdots s_{2}s_{1})(s_{1}s_{2} \cdots s_{k}s_{k+1}) \\
& = s_{1} s_{2} \cdots s_{n-t} s_{n-t+2} \cdots s_{k} s_{k+1},
\end{align*}
which implies that $\ell(s_{\gamma_{t}}w_{n-k}) = \ell(w_{n-k})-1$. 
Also, we see that for $n-k+1 \le t \le n$,
\begin{align*}
w_{n-k}^{-1}\gamma_{t} & = 
s_{k+1}s_{k} \cdots s_{2}s_{1}(\alpha_{1} + \cdots + \alpha_{n-t+1}) \\
& = - (\alpha_{n-t+1} + \cdots + \alpha_{k} + \alpha_{k+1}), 
\end{align*}
which implies that 
$\pair{2\rho}{-w_{n-k}^{-1}\gamma_{t}^{\vee}}-1 \ge 3$. 
Therefore, there does not exist a (quantum) edge from $w_{n-k}$ to $s_{\gamma_{t}}w_{n-k}$ 
for any $n-k+1 \le t \le n$. Hence we obtain
$w_{n-k} = w_{n-k+1} = \cdots = w_{n}$.
Thus, $\bw = (w_{0},w_{1},\dots,w_{n})$ is of the form (2). 

\paragraph{Case 2.}
%
Assume that ($0 \le k \le n$, and) $w_{n-k} = w = s_{1}s_{2} \cdots s_{k}$. 
If $k=0$, then $\bw = (w_{0},w_{1},\dots,w_{n}) = (e,\dots,e)$, 
which is of the form (1) (with $w = e$). 
Assume that $k \ge 1$. We see that for $n-k+1 \le t \le n$, 
\begin{align*}
s_{\gamma_{t}}w_{n-k} & = 
( s_{1} s_{2} \cdots s_{n-t} s_{n-t+1} s_{n-t} \cdots s_{2}s_{1})(s_{1}s_{2} \cdots s_{k}) \\
& = s_{1} s_{2} \cdots s_{n-t} s_{n-t+2} \cdots s_{k},
\end{align*}
which implies that $\ell(s_{\gamma_{t}}w_{n-k}) = \ell(w_{n-k})-1$. 
Also, we see that for $n-k+1 \le t \le n$,
\begin{align*}
w_{n-k}^{-1}\gamma_{t} & = 
s_{k} \cdots s_{2}s_{1}(\alpha_{1} + \cdots + \alpha_{n-t+1}) \\
& = - (\alpha_{n-t+1} + \cdots + \alpha_{k}), 
\end{align*}
which implies that 
$\pair{2\rho}{-w_{n-k}^{-1}\gamma_{t}^{\vee}}-1 \ge 3$ 
for $n-k+2 \le t \le n$, and 
$\pair{2\rho}{-w_{n-k}^{-1}\gamma_{n-k+1}^{\vee}}-1 = 1$ 
(for $t = n-k+1$). 
Therefore, for $n-k+1 \le t \le n$, there exists an edge 
from $w_{n-k}$ to $s_{\gamma_{t}}w_{n-k}$ if and only 
$t = n-k+1$; in this case, the edge is a quantum one whose label is $\alpha_{k}$, 
and $s_{\gamma_{n-k+1}}w_{n-k} = s_{1}s_{2} \cdots s_{k-1}$. 
Thus, we deduce that 
\begin{itemize}
\item $w_{n-k+1}$ is either $w_{n-k}=w$ or 
$s_{\gamma_{n-k+1}}w_{n-k} = s_{1}s_{2} \cdots s_{k-1}$; 

\item if $w_{n-k+1}=w_{n-k}$, then 
$w_{t}=w_{n-k}=w$ for all $n-k+1 \le t \le n$; 
in this case, $\bw$ is of the form (1). 
\end{itemize}
Assume that $w_{n-k+1} = s_{\gamma_{n-k+1}}w_{n-k} = s_{1}s_{2} \cdots s_{k-1}$. 
By the same argument as above, we see that 
\begin{itemize}
\item $w_{n-k+2}$ is either $w_{n-k+1}$ or 
$s_{\gamma_{n-k+2}}w_{n-k+1} = s_{1}s_{2} \cdots s_{k-2}$; 

\item if $w_{n-k+2}=w_{n-k+1}$, then 
$w_{t}=w_{n-k+1}$ for all $n-k+2 \le t \le n$; 
in this case, $\bw$ is of the form (3) with $m=k$. 
\end{itemize}
In addition, assume that 
$w_{n-k+2} = s_{\gamma_{n-k+2}}w_{n-k+1} = s_{1}s_{2} \cdots s_{k-2}$. 
By the same argument as above, we see that 
\begin{itemize}
\item $w_{n-k+3}$ is either $w_{n-k+2}$ or 
$s_{\gamma_{n-k+3}}w_{n-k+2} = s_{1}s_{2} \cdots s_{k-3}$; 

\item if $w_{n-k+3}=w_{n-k+2}$, then 
$w_{t}=w_{n-k+2}$ for all $n-k+3 \le t \le n$; 
in this case, $\bw$ is of the form (3) with $m=k-1$. 
\end{itemize}
Continuing in this way, we conclude that 
$\bw$ is either of the form (1) or of the form (3). 
This proves the proposition. 
\end{proof}

Let $(\bw,\bb) \in \twalk$. 
We can easily observe the following. 
\begin{enu}
\item[(i)] 
Assume that $\bw$ is of the form (1) in Proposition~\ref{prop:QW1}.
If $k=0$, then $S(\bw)=S^{+}(\bw)=\emptyset$, and 
\begin{align*}
& (-1)^{(\bw,\bb)} = 1, \quad 
  \wt(\bw,\bb)=0. \quad 
\end{align*}
If $k \ge 1$, then 
$S(\bw)=S^{+}(\bw)=\{ n-k+1 \}$. If $\bb(n-k+1)=0$, then 
\begin{align*}
& (-1)^{(\bw,\bb)} = 1, \quad 
  \wt(\bw,\bb)=0. \quad 
\end{align*}
If $\bb(n-k+1)=1$, then 
\begin{align*}
& (-1)^{(\bw,\bb)} = -1, \quad 
  \wt(\bw,\bb)=-\lng \alpha_{k}. \quad 
\end{align*}

\item[(ii)] 
If $\bw$ is of the form (2) in Proposition~\ref{prop:QW1}, then 
$S(\bw)=S^{+}(\bw)= \emptyset$, and
\begin{align*}
& (-1)^{(\bw,\bb)} = - 1, \quad 
  \wt(\bw,\bb)=0. \quad 
\end{align*}

\item[(iii)] If $\bw$ is of the form (3) 
in Proposition~\ref{prop:QW1}, with $2 \le m \le k$ (resp., $m=1$), then 
$S(\bw)=S^{+}(\bw)=\{ n-m+2 \}$ (resp., $=\emptyset$). 
If $2 \le m \le k$ and $\bb(n-m+2)=0$, or if $m=1$, then 
\begin{align*}
& (-1)^{(\bw,\bb)} = 1, \quad 
  \wt(\bw,\bb)=-\lng \alpha_{m,k}. \quad 
\end{align*}
If $2 \le m \le k$ and $\bb(n-m+2)=1$, then 
\begin{align*}
& (-1)^{(\bw,\bb)} = - 1, \quad 
  \wt(\bw,\bb)=-\lng \alpha_{m-1,k}. \quad 
\end{align*}
\end{enu}
By using Proposition~\ref{prop:QW1} and 
the observations (i), (ii), (iii) above, 
we deduce the following proposition from equation \eqref{eq:invchev2}.
%
%
\begin{prop} \label{prop:inv1}
For $0 \le k \le n$, the following equality holds in $K_{H}(\QG)$\,{\rm:}
\begin{align*}
& \be^{\vpi_{1}} [ \CO_{\QG(s_{1} \cdots s_{k})} ] \\
& = [ \CO_{\QG(s_{1} \cdots s_{k})}] \otimes [\CO_{\QG}(-\lng s_{k} \cdots s_{1} \vpi_{1}) ] \\
& -  [ \CO_{\QG(s_{1} \cdots s_{k}t_{\alpha_{k}})}] \otimes [\CO_{\QG}(-\lng s_{k} \cdots s_{1} \vpi_{1}-\lng \alpha_{k}) ] \\
& - \underbrace{%
  [ \CO_{\QG(s_{1} \cdots s_{k}s_{k+1})}] \otimes [\CO_{\QG}(-\lng s_{k} \cdots s_{1} \vpi_{1}) ]}_{%
  \text{\rm This term is understood to be $0$ if $k=n$.} } \\
& + \sum_{1 \le m \le k}  [ \CO_{\QG(s_{1} \cdots s_{m-1} t_{\alpha_{m,k}^{\vee}})}] \otimes 
    [\CO_{\QG}(-\lng s_{k} \cdots s_{1} \vpi_{1}-\lng \alpha_{m,k}) ] \\
& - \sum_{2 \le m \le k}  [ \CO_{\QG(s_{1} \cdots s_{m-1} t_{\alpha_{m-1,k}^{\vee}})}] \otimes 
    [\CO_{\QG}(-\lng s_{k} \cdots s_{1} \vpi_{1}-\lng \alpha_{m-1,k}) ]. 
\end{align*}
\end{prop}
Multiplying both sides of the equality in this proposition 
by $[\CO_{\QG}(\lng s_{k} \cdots s_{1} \vpi_{1}) ]$, 
we obtain the following.
%
%
\begin{cor} \label{cor:inv1}
For $0 \le k \le n$, the following equality holds in $K_{H}(\QG)$\,{\rm:}
\begin{align*}
& \overbrace{[ \CO_{\QG( s_{1} \cdots s_{k}s_{k+1} )} ]}^{%
  \begin{subarray}{c}
  \text{\rm This term is understood} \\[1mm]
  \text{\rm to be $0$ if $k=n$.} \end{subarray}}  \\
& = - \be^{\vpi_{1}} [ \CO_{\QG(s_{1} \cdots s_{k})}] \otimes [ \CO_{\QG}(\lng \eps_{k+1}) ] \\
& + [ \CO_{\QG(s_{1} \cdots s_{k})} ] \\
& + \sum_{1 \le m \le k}  [ \CO_{\QG(s_{1} \cdots s_{m-1} t_{\alpha_{m,k}^{\vee}})} ] 
    \otimes [\CO_{\QG}(\lng \eps_{k+1} - \lng \eps_{m}) ] \\
& - \sum_{1 \le m \le k}  [ \CO_{\QG(s_{1} \cdots s_{m} t_{\alpha_{m,k}^{\vee}})}
    \otimes [\CO_{\QG}(\lng \eps_{k+1} - \lng \eps_{m}) ]. 
\end{align*}
\end{cor}
%
%
\subsection{An identity for $K$-theoretic quantum shifted elementary polynomials.}
\label{subsec:identity}

In this subsection, we assume that $G$ is of type $A_{n}$, i.e., $G = SL_{n+1}(\BC)$. 
For a nonnegative integer $m \in \BZ_{\ge 0}$, we set $[m]:=\{1,2,\dots,m\}$, 
and for a subset $J$ of $[n+1]=\bigl\{1,2,\dots,n+1\bigr\}$, 
we set $\eps_{J}:=\sum_{j \in J} \eps_{j}$; note that $\eps_{\emptyset}=0$. 
Recall that for $\xi \in Q^{\vee,+}$, the $\BZ[P]$-linear operator $\st_{\xi}$ on $K_{H}(\QG)$ is defined by $\st_{\xi}[O_{Q_{G}(x)}] := [O_{Q_{G}(x t_{\xi})}]$ for $x \in W_{\af}^{\geq 0}$; 
for $j \in I$, we set $\st_{j}:=\st_{\alpha_{j}^{\vee}}$. 
Also, recall that 
\begin{equation} \label{eq:txi}
(\st_{\xi}[\CO_{\QG(x)}]) \otimes [\CO_{\QG}(\nu)] = 
\st_{\xi}([\CO_{\QG(x)}] \otimes [\CO_{\QG}(\nu)])
\end{equation}
for $x \in W_{\af}^{\geq 0}$, $\xi \in Q^{\vee,+}$, and $\nu \in P$.
Now, for $0 \le l \le k \le n+1$, we set 
\begin{equation}
\FF^{k}_{l}:=
 \sum_{
   \begin{subarray}{c}
   J \subset [k] \\[1mm]
   |J|=l
   \end{subarray}}
 \left( \prod_{ j \notin J,\, j+1 \in J }
 (1-\st_{j}) \right)[\CO_{\QG}(\lng \eps_{J})]; 
\end{equation}
note that $\FF^{k}_{0}=1$ for all $0 \le k \le n+1$. 
%
%
\begin{prop} \label{prop:ckk}
Let $0 \le k \le n+1$. 
The following equality holds in $K_{H}(\QG)$\,{\rm:} 
\begin{equation} \label{eq:k}
\underbrace{[ \CO_{\QG(s_{1}s_{2} \cdots s_{k-1}s_{k})} ]}_{
\begin{subarray}{c}
\text{\rm This term is understood  } \\[1mm]
\text{\rm to be $0$ if $k=n+1$.} \end{subarray} } 
= \sum_{0 \le l \le k} (-1)^{l} \be^{l\vpi_{1}} \FF^{k}_{l}. 
\end{equation}
\end{prop}

\begin{proof}
We prove \eqref{eq:k} by induction on $k$. 
If $k = 0$, then the assertion is obvious. Assume that $0 \le k < n+1$. 
By Corollary~\ref{cor:inv1} and the induction hypothesis, together with \eqref{eq:txi}, 
we deduce that
\begin{align*}
[ \CO_{\QG(s_{1}s_{2} \cdots s_{k}s_{k+1})} ]
& = - \sum_{0 \le l \le k} (-1)^{l} \be^{(l+1)\vpi_{1}}
 \sum_{
   \begin{subarray}{c}
   J \subset [k] \\[1mm]
   |J|=l
   \end{subarray}}
 \left( \prod_{ j \notin J,\, j+1 \in J }
 (1-\st_{j})\right)  [\CO_{\QG}(\lng \eps_{J \cup \{ k+1 \}})] \\[3mm]
& + \sum_{0 \le l \le k} (-1)^{l} \be^{l\vpi_{1}}
 \sum_{
   \begin{subarray}{c}
   J \subset [k] \\[1mm]
   |J|=l
   \end{subarray}}
 \left( \prod_{ j \notin J,\, j+1 \in J }
 (1-\st_{j})\right)  [\CO_{\QG}(\lng \eps_{J})] \\[3mm]
& + \sum_{1 \le m \le k}  
 \sum_{0 \le l \le m-1} (-1)^{l} \be^{l\vpi_{1}} \times \\
& \hspace*{15mm}
 \sum_{
   \begin{subarray}{c}
   J \subset [m-1] \\[1mm]
   |J|=l
   \end{subarray}}
 \st_{\alpha_{m,k}^{\vee}}
 \left( \prod_{ j \notin J,\, j+1 \in J }
 (1-\st_{j})\right) [\CO_{\QG}(\lng \eps_{J} + \lng \eps_{k+1} - \lng \eps_{m}) ] \\[3mm]
& + \sum_{1 \le m \le k}  
 \sum_{0 \le l \le m} (-1)^{l} \be^{l\vpi_{1}} \times \\
& \hspace*{15mm}
 \sum_{
   \begin{subarray}{c}
   J \subset [m] \\[1mm]
   |J|=l
   \end{subarray}}
 \st_{\alpha_{m,k}^{\vee}}
 \left( \prod_{ j \notin J,\, j+1 \in J }
 (1-\st_{j})\right) [\CO_{\QG}(\lng \eps_{J} + \lng \eps_{k+1} - \lng \eps_{m}) ].
\end{align*}
Here, observe that for each $1 \le m \le k$, 
\begin{align*}
& 
\sum_{0 \le l \le m-1} (-1)^{l} \be^{l\vpi_{1}}
 \sum_{
   \begin{subarray}{c}
   J \subset [m-1] \\[1mm]
   |J|=l
   \end{subarray}} \st_{\alpha_{m,k}^{\vee}}
 \left( \prod_{ j \notin J,\, j+1 \in J }
(1-\st_{j})\right) [\CO_{\QG}(\lng \eps_{J}+ \lng \eps_{k+1} - \lng \eps_{m} ) ] \\[3mm]
& - \sum_{0 \le l \le m} (-1)^{l} \be^{l\vpi_{1}}
 \sum_{
   \begin{subarray}{c}
   J \subset [m] \\[1mm]
   |J|=l
   \end{subarray}} \st_{\alpha_{m,k}^{\vee}}
 \left( \prod_{ j \notin J,\, j+1 \in J }
(1-\st_{j})\right) [\CO_{\QG}(\lng \eps_{J}+\lng \eps_{k+1} - \lng \eps_{m}) ] \\[5mm]
& = - \sum_{0 \le l \le m} (-1)^{l} \be^{l\vpi_{1}}
 \sum_{
   \begin{subarray}{c}
   J \subset [m] \\[1mm]
   |J|=l, \, \max J = m
   \end{subarray}} \st_{\alpha_{m,k}^{\vee}}
 \left( \prod_{ j \notin J,\, j+1 \in J }
(1-\st_{j})\right) [\CO_{\QG}(\lng \eps_{J} + \lng \eps_{k+1} - \lng \eps_{m}) ]. 
\end{align*}
Therefore, we find that 
\begin{align}
& [ \CO_{\QG}(s_{1}s_{2} \cdots s_{k}s_{k+1}) ] \nonumber \\[3mm]
& = 
- \sum_{0 \le l \le k} (-1)^{l} \be^{(l+1)\vpi_{1}}
 \sum_{
   \begin{subarray}{c}
   J \subset [k] \\[1mm]
   |J|=l
   \end{subarray}}
 \left( \prod_{ j \notin J,\, j+1 \in J }
(1-\st_{j})\right)  [\CO_{\QG}(\lng \eps_{J \cup \{ k+1 \}})] \nonumber \\[3mm]
& + \sum_{0 \le l \le k} (-1)^{l} \be^{l\vpi_{1}}
 \sum_{
   \begin{subarray}{c}
   J \subset [k] \\[1mm]
   |J|=l
   \end{subarray}}
 \left( \prod_{ j \notin J,\, j+1 \in J }
(1-\st_{j})\right)  [\CO_{\QG}(\lng \eps_{J})] \nonumber \\[3mm]
& - \sum_{1 \le m \le k} 
\sum_{0 \le l \le m} (-1)^{l} \be^{l\vpi_{1}}
 \sum_{
   \begin{subarray}{c}
   J \subset [m] \\[1mm]
   |J|=l, \, \max J = m
   \end{subarray}} \st_{\alpha_{m,k}^{\vee}}
 \left( \prod_{ j \notin J,\, j+1 \in J } 
(1-\st_{j})\right) [\CO_{\QG}(\lng \eps_{(J \setminus \{m\}) \cup \{k+1\}}) ]. \label{eq1}
\end{align}

Let us fix $1 \le p \le k$, and let $1 \le j_{1} < \cdots < j_{p} \le k$, 
with $j_{p} \ne k$; 
note that $p < k$. We set $J:=\bigl\{j_{1},\dots,j_{p}\bigr\}$. 
We see that the $\BZ[P]$-linear operator applied to
$[\lng \eps_{J \cup \{k+1\}}]$ on the right-hand side 
of \eqref{eq1} is identical to
\begin{align*}
& - (-1)^{p} \be^{(p+1)\vpi_{1}}
\left( \prod_{ j \notin J,\,j+1 \in J } 
(1-\st_{j}) \right) \nonumber \\
& -  \sum_{j_{p} < m \le k} 
(-1)^{p+1} \be^{(p+1)\vpi_{1}}
\st_{\alpha_{m,k}^{\vee}} \left( \prod_{ j \notin J, \, j+1 \in J } 
(1-\st_{j}) \right) \\[3mm]
& = (-1)^{p+1}\be^{(p+1)\vpi_{1}}
\left( \prod_{ j \notin J,\,j+1 \in J } 
(1-\st_{j}) \right) \\
& \times
\left( 1 - \left( \st_{ \alpha_{j_{p}+1,k}^{\vee} } + 
\sum_{j_{p}+1 < m \le k} (1-\st_{ \alpha_{m-1}^{\vee} }) \st_{ \alpha_{m,k}^{\vee} } \right) \right) \\[3mm]
& = (-1)^{p+1} \be^{(p+1)\vpi_{1}}
\left( \prod_{ j \notin J,\,j+1 \in J } 
(1-\st_{j}) \right)(1-\st_{k}). 
\end{align*}
Substituting this equality into \eqref{eq1}, 
we obtain \eqref{eq:k}, with $k$ replaced by $k+1$.
This completes the proof of the proposition. 
\end{proof}
%
%
\begin{prop} \label{prop:ctk}
For $0 \le t \le n$, 
the following equality holds in $K_{H}(\QG)$\,{\rm:}
\begin{equation} \label{eq:k=n+1-t}
\begin{split}
0 = & \sum_{0 \le l \le n+1-t} (-1)^{l} \times \\[2mm]
& \qquad
  \sum_{p_{1}=t-1}^{n-l} \sum_{p_{2}=t-2}^{p_{1}-1} \cdots \sum_{p_{t}=0}^{p_{t-1}-1}
  \be^{(l-n)\vpi_{1}+\vpi_{2}+\cdots+\vpi_{t}+p_{1}\alpha_{1}+\cdots+p_{t}\alpha_{t}} \FF^{n+1}_{l}.
\end{split}
\end{equation}
\end{prop}

\begin{proof}
We prove \eqref{eq:k=n+1-t} by induction on $t$. If $t=0$, 
then \eqref{eq:k=n+1-t} is just \eqref{eq:k} with $k=n+1$. 
Assume that $0 \le t < n$. 
Multiplying both sides of \eqref{eq:k=n+1-t} by $\be^{\vpi_{t+1}}$, and then 
applying the Demazure operator $\sD_{t+1}$ (see Appendix~\ref{sec:Demazure}), 
we see that 
\begin{equation} \label{eq:k=n+1-t-1}
\begin{split}
& 0 = \sum_{0 \le l \le n-t} (-1)^{l} \times \\[2mm]
& \sum_{p_{1}=t}^{n-l} \sum_{p_{2}=t-1}^{p_{1}-1} \cdots \sum_{p_{t}=1}^{p_{t-1}-1}\sum_{p_{t+1}=0}^{p_{t}-1}
  \be^{(l-n)\vpi_{1}+\vpi_{2}+\cdots+\vpi_{t}+\vpi_{t+1}+p_{1}\alpha_{1}+\cdots+p_{t}\alpha_{t}+p_{t+1}\alpha_{t+1}} \FF^{n+1}_{l}; 
\end{split}
\end{equation}
here we have used formula \eqref{eq:leib2} and the fact that 
\begin{equation*}
\begin{split}
& D_{t+1}(\be^{(l-n)\vpi_{1}+\vpi_{2}+\cdots+\vpi_{t}+\vpi_{t+1}+p_{1}\alpha_{1}+\cdots+p_{t}\alpha_{t}}) \\[3mm]
& = 
\begin{cases}
\sum_{p_{t+1}=0}^{p_{t}-1}
\be^{(l-n)\vpi_{1}+\vpi_{2}+\cdots+\vpi_{t}+\vpi_{t+1}+p_{1}\alpha_{1}+\cdots+p_{t}\alpha_{t}+p_{t+1}\alpha_{t+1}}
& \text{if $p_{t} > 0$}, \\[2mm]
0 & \text{if $p_{t}=0$}. 
\end{cases}
\end{split}
\end{equation*}
This proves the proposition. 
\end{proof}

For $0 \le l \le n+1$, we set 
\begin{equation} \label{eq:FE}
\FE^{n+1}_{l}:= \sum_{
   \begin{subarray}{c}
   J \subset [n+1] \\[1mm]
   |J|=l
   \end{subarray}} \be^{-\eps_{J}}, 
\end{equation}
which is the fundamental symmetric polynomial of degree $l$ 
in the variables $\be^{-\eps_{i}}$ for $1 \le i \le n+1$. 
Also, for $1 \le m \le n+1$ and $k \ge 0$, we set

\begin{equation} \label{eq:FH}
\FH^{m}_{k}:= \sum_{
  \begin{subarray}{c}
  r_{1},r_{2},\dots,r_{m} \ge 0 \\[1mm]
  r_{1}+r_{2}+\cdots+r_{m}=k
  \end{subarray} }
\be^{-(r_{1}\eps_{1}+r_{2}\eps_{2}+\cdots+r_{m}\eps_{m})}, 
\end{equation}
which is the complete symmetric polynomial of degree $k$ 
in the variables $\be^{-\eps_{i}}$ for $1 \le i \le m$. 
%
%
\begin{thm} \label{thm:FFn+1}
For $0 \le l \le n+1$, the equality $\FF^{n+1}_{l} = \FE^{n+1}_{l}$ holds in $K_{H}(\QG)$. 
\end{thm}

\begin{proof}
We deduce from \eqref{eq:k=n+1-t} that 
\begin{equation} \label{eq:k=n+1-t2}
\begin{cases}
\displaystyle{\sum_{0 \le l \le n+1-m}} (-1)^{l - (n+1-m)} \FH^{m+1}_{n+1-m-l} \FF^{n+1}_{l}=0 
& \text{for $0 \le m \le n$}, \\[8mm]
\FF^{n+1}_{n+1}= \be^{-(\eps_{1}+\eps_{2}+\cdots+\eps_{n+1})} = \be^{0} = 1. 
\end{cases}
\end{equation}
By \eqref{eq:k=n+1-t2}, 
we can show by induction on $l$ that $\FF^{n+1}_{l} \in \BZ[P]$ 
for all $0 \le l \le n+1$; note that $\FF^{n+1}_{0} = \be^{0} = 1$. 
Also, we see that $\FF^{n+1}_{l}$ for $0 \le l \le n+1$ are the unique elements 
in $\BZ[P]$ satisfying \eqref{eq:k=n+1-t2}. Hence, in order to prove 
Theorem~\ref{thm:FFn+1}, it suffices to show that 
\begin{equation} \label{eq:k=n+1-t2a}
\sum_{0 \le l \le n+1-m} (-1)^{l - (n+1-m)} \FH^{m+1}_{n+1-m-l} \FE^{n+1}_{l}=0 \qquad 
\text{for $0 \le m \le n$},
\end{equation}
\begin{equation} \label{eq:k=n+1-t2b}
\FE^{n+1}_{n+1}= \be^{-(\eps_{1}+\eps_{2}+\cdots+\eps_{n+1})}. 
\end{equation}
Equation~\eqref{eq:k=n+1-t2b} is obvious by definition \eqref{eq:FE}. 
Let us show equation \eqref{eq:k=n+1-t2a}. 
Let $0 \le m \le n$, and let $x$ be a formal variable. 
Observe that
\begin{equation*}
\sum_{l=0}^{n+1} \FE^{n+1}_{l} x^{l} = 
\prod_{i=1}^{n+1}(1+\be^{-\eps_{i}}x), \qquad 
\sum_{k=0}^{\infty} \FH^{m+1}_{k} x^{k} = 
\prod_{i=1}^{m+1}\frac{1}{1-\be^{-\eps_{i}}x}. 
\end{equation*}
Therefore, we have 
\begin{equation} \label{eq:k=n+1-t2c}
\left(\sum_{k=0}^{\infty} (-1)^{k}\FH^{m+1}_{k} x^{k}\right) 
\left(\sum_{l=0}^{n+1} \FE^{n+1}_{l} x^{l}\right) = 
\prod_{i=m+2}^{n+1}(1+\be^{-\eps_{i}}x).
\end{equation}
Let us compare the coefficients of $x^{n+1-m}$ on both sides. 
We see that the coefficient of $x^{n+1-m}$ on the left-hand side is 
equal to that on the left-hand side of \eqref{eq:k=n+1-t2a}. 
Since the right-hand side of \eqref{eq:k=n+1-t2c} 
is a polynomial in $x$ of degree $(n+1)-(m+2)+1= n-m$, 
it follows that the coefficient of $x^{n+1-m}$ on the right-hand side is 
equal to $0$. From these, we obtain equation \eqref{eq:k=n+1-t2a}, as desired. 
\end{proof}


\section{Relation between $K_{H}(\QG)$ and $QK_{H}(G/B)$.} 
\label{sec:relation}

Let $G$ be a connected, simply-connected simple algebraic group over $\BC$, 
with Borel subgroup $B \subset G$ and maximal torus $H \subset B$; 
$G$ is not necessarily assumed to be of type $A_{n}$, unless stated explicitly. Let $QK_{H}(G/B) := K_{H}(G/B) \otimes_{R(H)} R(H)\bra{Q^{\vee,+}}$ 
denote the $H$-equivariant quantum $K$-theory ring of the ordinary flag manifold $G/B$, 
defined by Givental~\cite{Giv} and Lee~\cite{Lee}, where $R(H)\bra{Q^{\vee,+}}$ is 
the ring of formal power series in the Novikov variables $Q_i = Q^{\alpha_i^{\vee}}$, $i \in I$, with coefficients in the representation ring $R(H)$ of $H$; 
for $\xi = \sum_{i \in I} k_i \alpha_i^{\vee} \in 
Q^{\vee,+} = \sum_{i \in I} \BZ_{\geq 0} \alpha_i^{\vee}$, we set 
$Q^{\xi} := \prod_{i \in I} Q_i^{k_i} \in R(H)\bra{Q^{\vee,+}}$.
The quantum $K$-theory ring $QK_{H}(G/B)$ is a free module over $R(H)\bra{Q^{\vee,+}}$ 
with the (opposite) Schubert classes $[\CO^{w}]$, $w \in W$, as a basis; 
also, the quantum multiplication $\star$ in $QK_{H}(G/B)$ is a deformation of 
the classical tensor product in $K_{H}(G/B)$, and is defined in terms of the $2$-point and $3$-point 
(genus zero, equivariant) $K$-theoretic Gromov-Witten invariants; 
see \cite{Giv} and \cite{Lee} for details.

In \cite{Kat1,Kat3}, based on \cite{BF,IMT} (see also \cite{ACT}), 
Kato established an $R(H)$-module isomorphism $\Phi$ from $QK_{H}(G/B)$ onto 
the $H$-equivariant $K$-group $K_{H}(\QG) = \prod_{x \in W_{\af}^{\geq 0}} \BZ[P][\CO_{\QG(x)}]$ (direct product) of the semi-infinite flag manifold $\QG$, 
in which tensor product operation with an arbitrary line bundle class is 
induced from that in $K_{H\times\BC^*}(\QG)$ by the specialization $q = 1$; 
in our notation, 
the map $\Phi$ sends 
the (opposite) Schubert class $\be^{\mu}[\CO^{w}] Q^{\xi}$ in $QK_{H}(G/B)$ 
to the corresponding semi-infinite Schubert class $\be^{-\mu}[\CO_{\QG(wt_{\xi})}]$ 
in $K_{H}(\QG)$ for $\mu \in P$, $w \in W$, and $\xi \in Q^{\vee,+}$. 
The isomorphism $\Phi$ also respects, in a sense, quantum multiplication $\star$ 
in $QK_{H}(G/B)$ and tensor product in $K_{H}(\QG)$.
More precisely, one has the commutative diagram:
%
%
\begin{equation} \label{eq:qdiagram}
\begin{CD}
QK_{H}(G/B) @>{\sim}>> K_{H}(\QG) \\
@V{\bullet \,\, \star [\CO_{G/B}(- \varpi_i)]}VV  
@VV{\bullet \,\, \otimes [\CO_{\QG}(\lng \varpi_i)]}V \\
QK_{H}(G/B) @>>{\sim}> K_{H}(\QG). 
\end{CD}
\end{equation}
Here note that the line bundle $\CO_{G/B}(- \nu)$ over $G/B$ for $\nu \in P$ denotes the $G$-equivariant line bundle constructed as the quotient space $G \times^{B} \BC_{\nu}$ of the product space $G \times \BC_{\nu}$ by the usual (free) left action of $B$, given by $b. (g, v) := (g b^{-1}, b v)$ for $b \in B$ and $(g, v) \in G \times \BC_{\nu}$, where 
$\BC_{\nu}$ is the one-dimensional $B$-module of weight $\nu \in P$. 
(We warn the reader that the conventions of \cite{Kat1} 
differs from those of \cite{KNS} and this paper,
by the twist coming from the involution $-\lng$.)
Also, we know from \cite{Kat1} that for $w \in W$ and $\xi \in Q^{\vee,+}$, 
$\Phi([\CO^{w}] Q^{\xi}) = \st_{\xi} \, \Phi([\CO^{w}])$ holds, 
and hence that for an arbitrary element $\bullet$ of 
$QK_{H}(G/B)$ and $\xi \in Q^{\vee,+}$, 
\begin{equation}\label{eq:push}
\Phi(\bullet \, Q^{\xi}) = \st_{\xi} \, \Phi(\bullet)
\end{equation}
holds. 
%

We know from \cite[Corollary~5.14]{BCMP} that 
the quantum multiplicative structure over $R(H)\bra{Q^{\vee,+}}$ 
of $QK_{H}(G/B)$ 
is completely determined by the operators of quantum multiplication 
by $[\CO_{G/B}(-\vpi_{i})]$ for $i \in I$; 
in fact, by using Nakayama's Lemma (i.e., \cite[Corollary~4.8 b.]{E}) together with \cite[Exercise~7.3]{E}, 
we can show that $QK_{H}(G/B)$ is generated as an algebra (with quantum multiplication $\star$) 
over $R(H)\bra{Q^{\vee,+}}$ by $[\CO_{G/B}(-\varpi_{i})]$, $i \in I$, 
since $K_{H}(G/B)$ is known to be generated by the same line bundle classes 
as an algebra (with tensor product $\otimes$) over $R(H)$ (\cite{Mi}). 
Here we should mention that by using Nakayama's lemma together with \cite[Chapter 3, Exercise 2]{AM}, we can also show that the submodule $QK_{H}(G/B)_{\mathrm{loc}} := K_{H}(G/B) \otimes_{R(H)} R(H)[Q^{\vee,+}]_{\mathrm{loc}}$ of $QK_{H}(G/B) = K_{H}(G/B) \otimes_{R(H)} R(H)\bra{Q^{\vee,+}}$ is generated as an algebra over $R(H)[Q^{\vee,+}]_{\mathrm{loc}}$ by the line bundle classes $[\CO_{G/B}(-\varpi_{i})]$, $i \in I$, 
where $R(H)[Q^{\vee,+}]$ is the polynomial ring in the variables $Q_i$, $i \in I$, with coefficients in $R(H)$, and $R(H)[Q^{\vee,+}]_{\mathrm{loc}}$ is its localization with respect to the multiplicative set $1 + (Q)$, with $(Q)$ the ideal in $R(H)[Q^{\vee,+}]$ generated by the $Q_i$, $i \in I$. 
Therefore, from \cite[Proposition~9]{ACT}, we see that the submodule $QK_{H}(G/B)_{\mathrm{loc}}$ of $QK_{H}(G/B)$ is, in fact, a subalgebra under the quantum multiplication $\star$. 
%

Let $\K$ be a subset of $I$, and set $\mu = \sum_{i \in \K} \vpi_{i} \in P^{+}$ 
as in \eqref{eq:mu}; 
note that $\J=\J_{\mu}=
\bigl\{ i \in I \mid \pair{\mu}{\alpha_{i}^{\vee}}=0 \bigr\} 
= I \setminus \K$.
%
%
Let us rewrite (in our notation) \cite[(3.6)]{LS} 
in the case that $\lambda = \mu$ and $z = e$. For $v \in W$, 
if $\eta \in \LS(\mu)$ satisfies $\dnn(v,\eta)=z=e$ (in the notation of 
\cite[\S3.1]{LS}), then we have $\kappa(\eta)=e$. 
Hence, noting that $\LS(\mu) \subset \QLS(\mu)$, 
we see that $\eta=(e\,;\,0,1)$ by Lemma~\ref{lem:e}. 
Therefore, in this case, \cite[(3.6)]{LS} 
can be rewritten as follows (cf. \eqref{eq:e1a}): 
\begin{equation*}
[\CO_{G/B}(-\mu)] = 
\sum_{ \begin{subarray}{c} v \in W \\[1mm] \dn{v}{e}{\J}=e \end{subarray} }
\underbrace{(-1)^{\ell(v)-\ell(e)}}_{=(-1)^{\ell(v)}} 
 \be^{\mu} [\CO^{v}]; 
\end{equation*}
here we warn the reader that the sign convention 
for the $R(H)$-module structure of $K_{H}(G/B)$ in \cite{LS} is 
opposite to that in this paper. 
Hence, using Lemma~\ref{lem:Deo}\,(2), 
we obtain 
\begin{equation} \label{eq:LS0}
[\CO_{G/B}(-\mu)] = 
\sum_{ v \in W_{\K} } (-1)^{\ell(v)} \be^{\mu} [\CO^{v}] 
\end{equation}
in $K_{H}(G/B)$ by the same argument as for Proposition~\ref{prop:e2}. 

By comparing the RHS of \eqref{eq:e2b} and 
that of \eqref{eq:LS0}, we see that under the $R(H)$-module isomorphism $\Phi$ 
from $QK_{H}(G/B)$ onto 
$K_{H}(\QG)$, the class $[\CO_{G/B}(-\mu)]$ 
in $QK_{H}(G/B)$ corresponds to the class $[\CO_{\QG}(\lng \mu)]$ in $K_{H}(\QG)$,
where $\mu \in P^{+}$ is as in \eqref{eq:mu}. 
Here we recall that $QK_{H}(G/B)$ is generated as an algebra 
(with quantum multiplication $\star$) 
over $R(H)\bra{Q^{\vee,+}}$ by the line bundle classes $[\CO_{G/B}(-\varpi_{i})]$, $i \in I$. 
Therefore, by virtue of the commutative diagram \eqref{eq:qdiagram} 
together with \eqref{eq:shift} and \eqref{eq:push}, we deduce the following. 

\begin{prop} \label{prop:corresp}
Keep the notation above. 
Let $\mu \in P^{+}$ be of the form \eqref{eq:mu}. 
Then, under the $R(H)$-module isomorphism $\Phi$ 
from $QK_{H}(G/B)$ onto $K_{H}(\QG)$, we have 
\begin{equation}
\Phi(\bullet \star [\CO_{G/B}(-\mu)]) = \Phi(\bullet) \otimes [\CO_{\QG}(\lng \mu)] 
\end{equation}
for an arbitrary element $\bullet \in QK_{H}(G/B)$. 
\end{prop}

In particular, for distinct $i_{1}, \ldots, i_{l} \in I$, we conclude that
\begin{equation*}
[\CO_{G/B}(-\varpi_{i_1})] \star \cdots \star [\CO_{G/B}(-\varpi_{i_l})]
= [\CO_{G/B}(-(\varpi_{i_1} + \cdots + \varpi_{i_l}))]
\end{equation*}
in $QK_{H}(G/B)$, since we have
\begin{equation*}
[\CO_{\QG}(\lng \varpi_{i_1})] \otimes \cdots \otimes [\CO_{\QG}(\lng \varpi_{i_l})]
= [\CO_{\QG}(\lng(\varpi_{i_1} + \cdots + \varpi_{i_l}))]
\end{equation*}
in $K_{H}(\QG)$.
This result can be thought of as a generalization of 
\cite[Lemma~6+]{ACT} in simply-laced types to arbitrary types.


Let $\mu \in P^{+}$ be as in \eqref{eq:mu}, 
and let $m \in I \setminus \K$ be such that 
$\QLS(\vpi_{m}) = \LS(\vpi_{m})$. 
Let $\eta = (w_{1},\,\dots,\,w_{s} \,;\, 
a_{0},\,a_{1},\,\dots,\,a_{s}) \in \QLS(\vpi_{m}) = \LS(\vpi_{m})$ and $v \in W_{\K}$. 
Since $m \in I \setminus \K$, we have $\K \subset I \setminus \{m\}$, and hence 
$\mcr{v}^{I \setminus \{m\}} = e$. 

We see from \cite[(3.5)]{LS} that for $v \in W_{\K}$, 
\begin{equation} \label{eq:LSa1}
[\CO_{G/B}(\vpi_{m})][\CO^{v}] = 
\sum_{ \begin{subarray}{c} \eta \in \LS(\vpi_{m}) \\[1mm] 
       \kappa(\eta) \ge \mcr{v}^{I \setminus \{m\}} \end{subarray} }
 \be^{-\wt(\eta)} [\CO^{\upp(v,\eta)}] = 
\sum_{ \eta \in \LS(\vpi_{m}) }
 \be^{-\wt(\eta)} [\CO^{\upp(v,\eta)}] 
\end{equation}
in $K_{H}(G/B)$, where the second equality follows from the fact that 
$\mcr{v}^{I \setminus \{m\}} = e$, as seen above; 
here, again we warn the reader that the sign convention for 
the $R(H)$-module structure of $K_{H}(G/B)$ in \cite{LS} is 
opposite to that in this paper.
Combining \eqref{eq:LS0} and \eqref{eq:LSa1}, we obtain
\begin{equation}
\begin{split} \label{eq:ordinary}
[\CO_{G/B}(\vpi_{m}-\mu)] & = 
[\CO_{G/B}(\vpi_{m})][\CO_{G/B}(-\mu)] = 
\sum_{ v \in W_{K} } (-1)^{\ell(v)} \be^{\mu} [\CO_{G/B}(\vpi_{m})][\CO^{v}] \\
& = \sum_{ v \in W_{K} } (-1)^{\ell(v)} \be^{\mu} \sum_{ \eta \in \LS(\vpi_{m}) }
 \be^{-\wt(\eta)} [\CO^{\upp(v,\eta)}] 
\end{split}
\end{equation}
in $K_{H}(G/B)$.

By comparing the RHS of \eqref{eq:equivariant} and the rightmost-hand side of \eqref{eq:ordinary},
we see easily that under the $R(H)$-module isomorphism $\Phi$ 
from $QK_{H}(G/B)$ 
onto $K_{H}(\QG)$, 
the class $\frac{1}{1 - Q_{m}}[\CO_{G/B}(\varpi_{m} - \mu)]$ 
corresponds to the line bundle class $[\CO_{\QG}(-\lng(\varpi_{m} - \mu))]$,
where $\mu \in P^{+}$ is as in \eqref{eq:mu}. 
Therefore, by the same reasoning as for Proposition~\ref{prop:corresp}, we deduce the following.

\begin{prop} \label{prop:mixedcorresp} Keep the notation above. 
Let $\mu \in P^{+}$ be of the form \eqref{eq:mu}, and let $m \in I$ 
be an element of 
$\J=\J_{\mu}=I \setminus \K$ such that 
$\QLS(\vpi_{m}) = \LS(\vpi_{m})$. 
Then, under the $R(H)$-module isomorphism $\Phi$ from $QK_{H}(G/B)$ onto $K_{H}(\QG)$, we have 
\begin{equation}
\Phi \left(\bullet \star \frac{1}{1 - Q_{m}}[\CO_{G/B}(\vpi_{m} - \mu)]\right) = \Phi(\bullet) \otimes [\CO_{\QG}(-\lng(\vpi_{m} - \mu))]
\end{equation}
for an arbitrary element $\bullet \in QK_{H}(G/B)$. 
\end{prop}

In particular, if $G$ is of type $A_{n}$, i.e., $G = SL_{n+1}(\BC)$, then for each $1 \leq k \leq n+1$, 
the quantum multiplication with the class $\frac{1}{1 - Q_{k}}[\CO_{Fl_{n+1}}(\eps_{k})]$ 
(resp., $\frac{1}{1 - Q_{k-1}}[\CO_{Fl_{n+1}}(- \eps_{k})]$) in $QK_{H}(Fl_{n+1})$ 
corresponds to the tensor product with the line bundle class 
$[\CO_{\QG}(-\lng \eps_{k})]$ (resp., $[\CO_{\QG}(\lng \eps_{k})]$) 
in $K_{H}(\QG)$ since $\eps_{k} = \varpi_{k} - \varpi_{k-1}$, 
where $\varpi_{0} := 0, \varpi_{n+1} := 0$, 
and $Q_{0} := 0, Q_{n+1} := 0$ by convention.
Also, if $G$ is of type $C_{n}$, then for each $1 \leq k \leq n$, 
the quantum multiplication with the class $\frac{1}{1-Q_{k}}[\CO_{G/B}(\eps_{k})]$ 
(resp., $\frac{1}{1-Q_{k-1}}[\CO_{G/B}(- \eps_{k})]$) in $QK_{H}(G/B)$ 
corresponds to the tensor product with the line bundle class $[\CO_{\QG}(\eps_{k})]$ 
(resp., $[\CO_{\QG}(- \eps_{k})]$) since $\eps_{k} = \varpi_{k} - \varpi_{k-1}$, 
where $\varpi_{0} := 0$ and $Q_{0} := 0$ by convention.

In the rest of this section, we assume that $G$ is of type $A_{n}$, i.e., $G = SL_{n+1}(\BC)$. 
Now, for $0 \le p \le k \le n+1$, we set
\begin{equation}
\CF^{k}_{p}:=
 \sum_{
   \begin{subarray}{c}
   J \subset [k] \\[1mm]
   |J|=p
   \end{subarray}} \, 
 \prod_{ \begin{subarray}{c} 1 \le j \le k \\[1mm] j,\,j+1 \in J \end{subarray} }
 \frac{1}{1-Q_{j}} 
 \sprod_{j \in J} [\CO_{Fl_{n+1}}(-\eps_{j})] \in QK_{H}(Fl_{n+1}), 
\end{equation}
where $\prod^{\star}$ denotes the product with respect to the quantum multiplication $\star$; 
note that $\CF^{k}_{0} = 1$ for $0 \le k \le n+1$.
Then, from what we have obtained in the previous paragraph, we see easily that 
$\Phi(\CF^{k}_{p}) = \FF^{k}_{p} \in K_{H}(\QG)$ for $0 \le p \le k \le n+1$; 
recall that $\FF^{k}_{p}$ is defined by
\begin{equation}
\FF^{k}_{p} =
 \sum_{
   \begin{subarray}{c}
   J \subset [k] \\[1mm]
   |J|=p
   \end{subarray}}
 \left( \prod_{ j \notin J,\,j+1 \in J }
 (1-\st_{j}) \right)[\CO_{\QG}(\lng \eps_{J})]. 
\end{equation}
As a consequence, we obtain the following from Theorem~\ref{thm:FFn+1}.
\begin{thm} \label{thm:rel}
For $0 \leq l \leq n+1$, the following equality holds in $QK_{H}(Fl_{n+1})$\,{\rm:}
\begin{equation*}
\CF^{n+1}_{l}= 
\sum_{
   \begin{subarray}{c}
   J \subset [n+1] \\[1mm]
   |J|=l
   \end{subarray}}
\be^{\eps_{J}}. 
\end{equation*}
\end{thm}
%
%
\section{Presentation of $QK_{H}(Fl_{n+1})$.} 
\label{sec:presentation}

In this section, we assume that $G$ is of type $A_{n}$, i.e., $G = SL_{n+1}(\BC)$. 
Let $\CI$ denote the ideal of $R(H)[x_{1},\dots,x_{n},x_{n+1}]$ 
generated by 
\begin{equation*}
\sum_{
   \begin{subarray}{c}
   J \subset [n+1] \\[1mm]
   |J|=l
   \end{subarray}} \, 
\prod_{j \in J} (1-x_{j}) - 
\sum_{
   \begin{subarray}{c}
   J \subset [n+1] \\[1mm]
   |J|=l
   \end{subarray}}\be^{\eps_{J}}
\quad
\text{for $1 \le l \le n+1$}. 
\end{equation*}
It is well-known (see \cite[Introduction]{FL}; cf. \cite{PR99}) that there exists an $R(H)$-algebra isomorphism 
$\Psi$ from the quotient $R(H)[x_{1},\dots,x_{n},x_{n+1}]/\CI$ onto 
$K_{H}(Fl_{n+1})$ which maps the residue class of $1 - x_{j}$ modulo $\CI$ to 
$[\CO_{Fl_{n+1}}(-\eps_{j})]$ for $1 \le j \le n+1$; 
for $1 \leq j \leq n+1$, the line bundle $\CO_{Fl_{n+1}}(- \epsilon_{j})$ is just the quotient bundle $\mathcal{U}_{j}/\mathcal{U}_{j-1}$ over $Fl_{n+1}$, where $0 = \mathcal{U}_{0} \subset \mathcal{U}_{1} \subset \cdots \subset \mathcal{U}_{n} \subset \mathcal{U}_{n+1} = Fl_{n+1} \times \BC^{n+1}$ denotes a universal, or tautological, flag of subvector bundles of the trivial bundle $Fl_{n+1} \times \BC^{n+1}$. 

Now, let $R(H)\bra{Q} = R(H)\bra{Q_1, \ldots, Q_{n}}$ denote the ring of formal power series in the Novikov variables $Q_{i} = Q^{\alpha_i^{\vee}}$ corresponding to the simple coroots $\alpha_i^{\vee}$, $1 \leq i \leq n$, 
with coefficients in the representation ring $R(H)$ of $H$. 
Also, let $\CI^{Q}$ denote the ideal of 
$R(H)\bra{Q}[x_{1},\dots,x_{n},x_{n+1}]$ 
generated by 
\begin{equation*}
\sum_{
   \begin{subarray}{c}
   J \subset [n+1] \\[1mm]
   |J|=l
   \end{subarray}} \, 
\prod_{\begin{subarray}{c} 
  1 \le j \le n+1 \\[1mm]
  j \in J,\,j+1 \notin J
  \end{subarray}} (1-Q_{j})
\prod_{j \in J}(1-x_{j}) - 
\sum_{
   \begin{subarray}{c}
   J \subset [n+1] \\[1mm]
   |J|=l
   \end{subarray}}\be^{\eps_{J}}
\quad
\text{for $1 \le l \le n+1$},
\end{equation*}
where we understand that $1 - Q_{n+1} = 1$. 
The following is the main result of this paper.
%
%
\begin{thm} \label{thm:main}
There exists an $R(H)\bra{Q}$-algebra isomorphism 
\begin{equation} \label{eq:PsiQ}
\Psi^{Q}: R(H)\bra{Q}[x_{1},\dots,x_{n},x_{n+1}]/\CI^{Q} 
\stackrel{\sim}{\rightarrow} 
QK_{H}(Fl_{n+1})
\end{equation}
which maps the residue class of $(1-Q_{j})(1 -x_{j})$ modulo $\CI^{Q}$ to 
$[\CO_{Fl_{n+1}}(-\eps_{j})]$ for $1 \le j \le n$, and
the residue class of $1 -x_{n+1}$ modulo $\CI^{Q}$ to 
$[\CO_{Fl_{n+1}}(-\eps_{n+1})]$. 
\end{thm}

\begin{proof}
First, we define an $R(H)\bra{Q}$-algebra homomorphism 
\begin{equation*}
\ha{\Psi}^{Q}: R(H)\bra{Q}[x_{1},\dots,x_{n},x_{n+1}] \rightarrow QK_{H}(Fl_{n+1})
\end{equation*}
by $\ha{\Psi}^{Q}((1-Q_{j})(1 -x_{j})) := [\CO_{Fl_{n+1}}(-\eps_{j})]$ for $1 \leq j \leq n+1$, where 
we understand that $1 - Q_{n+1} = 1$. 
Then, it follows from Theorem~\ref{thm:rel} that $\ha{\Psi}^{Q}$ induces an $R(H)\bra{Q}$-algebra 
homomorphism 
\begin{equation*}
\Psi^{Q} : R(H)\bra{Q}[x_{1},\dots,x_{n},x_{n+1}]/\CI^{Q} \rightarrow QK_{H}(Fl_{n+1})
\end{equation*}
which maps the residue class of $(1-Q_{j})(1 -x_{j})$ modulo $\CI^{Q}$ to 
$[\CO_{Fl_{n+1}}(-\eps_{j})]$ for $1 \leq j \leq n+1$ (we understand that $1 - Q_{n+1} = 1$); 
note that at this point, we do not know whether $\Psi^{Q}$ is surjective or not. 
Here we can show that the quotient algebra $R(H)\bra{Q}[x_{1},\dots,x_{n},x_{n+1}]/\CI^{Q}$ is 
a finitely generated module over the Noetherian integral domain 
$R(H)\bra{Q} = R(H)\bra{Q_1, \ldots, Q_{n}}$, 
which is complete in the $(Q_{1}, \ldots, Q_{n})$-adic topology; 
see Corollary~\ref{cor:fg} of Appendix~\ref{sec:B}. 
Also, observe that under the specialization $Q_{1} = \cdots = Q_{n} = 0$, $\Psi^{Q}$ 
becomes the $R(H)$-algebra isomorphism $\Psi$ mentioned above 
from the quotient $R(H)[x_{1},\dots,x_{n},x_{n+1}]/\CI$ onto $K_{H}(Fl_{n+1})$.
Therefore, we can use Nakayama-type arguments, that is, we can apply \cite[Proposition~A.3]{GMSZ} to $\Psi^{Q}$ to conclude that 
it is an $R(H)\bra{Q}$-algebra isomorphism from 
$R(H)\bra{Q}[x_{1},\dots,x_{n},x_{n+1}]/\CI^{Q}$ onto $QK_{H}(Fl_{n+1})$.
This proves the theorem.
\end{proof}

\begin{rem}
Let $R(H)[Q]_{\mathrm{loc}}$ denote the localization of the polynomial ring $R(H)[Q] := R(H)[Q_1, \ldots, Q_n]$ with respect to the multiplicative set $1 + (Q_1, \ldots, Q_n)$. 
Also, let $\CI^{Q}_{\mathrm{loc}}$ denote the ideal in $R(H)[Q]_{\mathrm{loc}}[x_1, \ldots, x_n, x_{n+1}]$ generated by the same elements as for $\CI^{Q}$. 
By Remark~\ref{rem:locfg} of Appendix~\ref{sec:B}, we see that the quotient ring $R(H)[Q]_{\mathrm{loc}}[x_{1},\ldots,x_{n},x_{n+1}]/\CI^{Q}_{\mathrm{loc}}$ is also a finitely generated module over the Noetherian integral domain $R(H)[Q]_{\mathrm{loc}}$; it follows from \cite[p.~110]{AM} that $R(H)[Q]_{\mathrm{loc}}$ is a subring of $R(H)\bra{Q}$, and that the quotient ring $R(H)[Q]_{\mathrm{loc}}[x_{1},\ldots,x_{n},x_{n+1}]/\CI^{Q}_{\mathrm{loc}}$ is a subring of the ring $R(H)\bra{Q}[x_{1},\ldots,x_{n},x_{n+1}]/\CI^{Q}$. 
Therefore, by using \cite[Remark~A.6]{GMSZ} instead of \cite[Proposition~A.3]{GMSZ}, we can prove that the quotient ring $R(H)[Q]_{\mathrm{loc}}[x_{1},\ldots,x_{n},x_{n+1}]/\CI^{Q}_{\mathrm{loc}}$ is isomorphic to the subring $QK_{H}(Fl_{n+1})_{\mathrm{loc}} := K_{H}(Fl_{n+1}) \otimes_{R(H)} R(H)[Q]_{\mathrm{loc}}$ of $QK_{H}(Fl_{n+1}) = K_{H}(Fl_{n+1}) \otimes_{R(H)} R(H)\bra{Q}$.
\end{rem}

\appendix

\section*{Appendices.}

\section{Demazure operators.}
\label{sec:Demazure}

For $i \in I$, the Demazure operator $D_{i}$ on $\BZ[P]$ is given by: 
\begin{equation*}
D_{i} \be^{\nu} := 
\frac{ \be^{\nu - \rho} - \be^{s_{i}(\nu - \rho)} }{ 1-\be^{\alpha_{i}} } \be^{\rho} = 
\frac{ \be^{\nu}-\be^{\alpha_{i}}\be^{s_{i}\nu} }{ 1-\be^{\alpha_{i}} } \qquad 
\text{for $\nu \in P$}; 
\end{equation*}
note that $D_{i}^{2}=D_{i}$, and
\begin{equation*}
D_{i} \be^{\nu} = 
\begin{cases}
\be^{\nu}(1+\be^{\alpha_{i}}+\be^{2\alpha_{i}}+ \cdots +\be^{-\pair{\nu}{\alpha_{i}^{\vee}}\alpha_{i}}) 
  & \text{\rm if $\pair{\nu}{\alpha_{i}^{\vee}} \le 0$}, \\[1.5mm]
0 & \text{\rm if $\pair{\nu}{\alpha_{i}^{\vee}}=1$}, \\[1.5mm]
-\be^{\nu}(\be^{- \alpha_{i}}+ \be^{- 2\alpha_{i}}+\cdots +\be^{(-\pair{\nu}{\alpha_{i}^{\vee}}+1)\alpha_{i}})
  & \text{\rm if $\pair{\nu}{\alpha_{i}^{\vee}} \ge 2$}.
\end{cases}
\end{equation*}
%
%
\begin{lem}[{\cite[Lemma~6.1]{NOS}}] \label{lem:leib1}
For $\nu,\,\mu \in P$, we have
%
%
\begin{equation} \label{eq:leib1}
D_{i}(\be^{\nu}\be^{\mu}) = 
\frac{\be^{\nu}-\be^{s_{i}\nu}}{1-\be^{\alpha_{i}}}\be^{\mu} + 
\be^{s_{i}\nu} D_{i}(\be^{\mu}).
\end{equation}
\end{lem}

Let $\BH_{0}$ denote the nil-DAHA, which is the $\BZ[\q^{\pm 1}]$-algebra 
generated by $\sT_{i}$, $i \in I_{\af}$, and $\sX^{\nu}$, $\nu \in P$, 
with the defining relations that 
\begin{equation}
  \underbrace{\sT_{i}\sT_{j}\sT_{i} \cdots \cdots}_{\text{$m_{ij}$ factors}} = 
  \underbrace{\sT_{j}\sT_{i}\sT_{j} \cdots \cdots}_{\text{$m_{ij}$ factors}} 
  \qquad \text{for $i,j \in I_{\af}$ with $i \ne j$}, \label{eq:R1}
\end{equation}
where $m_{ij}$ is the order of $s_{i}s_{j}$ (if it is infinite, then 
this relation is omitted), and 
\begin{align}
& \sT_{i}(\sT_{i}+1)=0 \qquad \text{for $i \in I_{\af}$}, \label{eq:R2} \\
& \sX^{0}=1, \quad \sX^{\delta}=\q, \label{eq:R3} \\
& \sX^{\mu}\sX^{\nu}=\sX^{\mu+\nu} \qquad \text{for $\mu,\nu \in P$}, \label{eq:R4} \\
& \sT_{i}\sX^{\nu} = \sX^{s_{i}\nu}\sT_{i} - \frac{\sX^{\nu}-\sX^{s_{i}\nu}}{1-\sX^{\alpha_{i}}}
  \qquad \text{for $\nu \in P$ and $i \in I_{\af}$}. \label{eq:R5}
\end{align}
We set $\sD_{i}:=1+\sT_{i}$ for $i \in I_{\af}$, and also call it the Demazure operator. 
We know from \cite[Theorem 6.5]{KNS} (see also \cite[Proposition~3.4]{KNOS}) that 
there exists a left $\BH_{0}$-action on $\Kr$ 
uniquely determined by: 
\begin{align}
& \q \cdot [\CO_{\QG(t_{\xi})}(\lambda)] = q^{-1} [\CO_{\QG(t_{\xi})}(\lambda)], \label{eq:H1} \\
& \sX^{\nu} \cdot [\CO_{\QG(t_{\xi})}(\lambda)] = \be^{-\nu} [\CO_{\QG(t_{\xi})}(\lambda)], \label{eq:H2} \\
& \sD_{i} \cdot [\CO_{\QG(t_{\xi})}(\lambda)] = [\CO_{\QG(t_{\xi})}(\lambda)], \label{eq:H3} \\
& \sD_{0} \cdot [\CO_{\QG(t_{\xi})}(\lambda)] = [\CO_{\QG(s_{0}t_{\xi})}(\lambda)] \label{eq:H4}
\end{align}
for $i \in I$, $\xi \in Q^{\vee}$, and $\nu,\lambda \in P$. 
%
%
\begin{lem} \label{lem:leib2}
Let $i \in I$, and $\xi \in Q^{\vee}$, $\nu,\lambda \in P$. 
Then, in $\Kr$, we have 
%
%
\begin{equation} \label{eq:leib2}
\begin{split}
\sD_{i} \cdot (\be^{\nu} [\CO_{\QG(t_{\xi})}(\lambda)])
& = \frac{\be^{\nu}-\be^{s_{i}\nu}}{1-\be^{\alpha_{i}}}[\CO_{\QG(t_{\xi})}(\lambda)] + 
\be^{s_{i}\nu} [\CO_{\QG(t_{\xi})}(\lambda)] \\
& = D_{i}(\be^{\nu})[\CO_{\QG(t_{\xi})}(\lambda)].
\end{split}
\end{equation}
\end{lem}

\begin{proof}
In this proof, we omit $\cdot$ to denote the action of $\BH_{0}$. 
We compute as: 
\begin{align*}
\sD_{i} (\be^{\nu} [\CO_{\QG(t_{\xi})}(\lambda)])
 & = \sD_{i} \sX^{-\nu}  [\CO_{\QG(t_{\xi})}(\lambda)] \qquad \text{by \eqref{eq:H2}} \\
 & = (1+\sT_{i})  \sX^{-\nu}  [\CO_{\QG(t_{\xi})}(\lambda)] \\
 & = \sX^{-\nu}  [\CO_{\QG(t_{\xi})}(\lambda)] + \sT_{i}  \sX^{-\nu}  [\CO_{\QG(t_{\xi})}(\lambda)]. 
\end{align*}
Here, by \eqref{eq:R5}, 
\begin{align*}
& \sT_{i}\sX^{-\nu} [\CO_{\QG(t_{\xi})}(\lambda)] \\
& = \left(\sX^{-s_{i}\nu} \sT_{i} - \frac{\sX^{-\nu}-\sX^{-s_{i}\nu}}{1-\sX^{\alpha_{i}}}\right)
    [\CO_{\QG(t_{\xi})}(\lambda)] \\
& = \sX^{-s_{i}\nu} \sT_{i} [\CO_{\QG(t_{\xi})}(\lambda)] - 
    \frac{\sX^{-\nu}-\sX^{-s_{i}\nu}}{1-\sX^{\alpha_{i}}} [\CO_{\QG(t_{\xi})}(\lambda)]. 
\end{align*}
Also, it follows that 
\begin{align*}
\sT_{i} [\CO_{\QG(t_{\xi})}(\lambda)] 
 & = (\sD_{i}-1)[\CO_{\QG(t_{\xi})}(\lambda)] 
   = \sD_{i}[\CO_{\QG(t_{\xi})}(\lambda)] - [\CO_{\QG(t_{\xi})}(\lambda)] \\
 & = [\CO_{\QG(t_{\xi})}(\lambda)] - [\CO_{\QG(t_{\xi})}(\lambda)] \qquad \text{by \eqref{eq:H3}} \\
 & = 0.
\end{align*}
Therefore, we see that 
\begin{align*}
\sT_{i}\sX^{-\nu} [\CO_{\QG(t_{\xi})}(\lambda)] & = 
- \frac{\sX^{-\nu}-\sX^{-s_{i}\nu}}{1-\sX^{\alpha_{i}}} [\CO_{\QG(t_{\xi})}(\lambda)], 
\end{align*}
and hence 
\begin{align*}
& \sD_{i}(\be^{\nu} [\CO_{\QG(t_{\xi})}(\lambda)])
  = \sX^{-\nu} [\CO_{\QG(t_{\xi})}(\lambda)] - 
    \frac{\sX^{-\nu}-\sX^{-s_{i}\nu}}{1-\sX^{\alpha_{i}}} [\CO_{\QG(t_{\xi})}(\lambda)] \\
& = \left(\sX^{-\nu}-\sX^{-s_{i}\nu} - 
    \frac{\sX^{-\nu}-\sX^{-s_{i}\nu}}{1-\sX^{\alpha_{i}}} \right) [\CO_{\QG(t_{\xi})}(\lambda)] 
    + \sX^{-s_{i}\nu}[\CO_{\QG(t_{\xi})}(\lambda)] \\
& = (\sX^{-\nu}-\sX^{-s_{i}\nu}) \left(1-\frac{1}{1-\sX^{\alpha_{i}}}\right)
    [\CO_{\QG(t_{\xi})}(\lambda)] + \sX^{-s_{i}\nu}[\CO_{\QG(t_{\xi})}(\lambda)] \\
& = (\sX^{-\nu}-\sX^{-s_{i}\nu}) \left( \frac{1}{1-\sX^{-\alpha_{i}}}\right)
    [\CO_{\QG(t_{\xi})}(\lambda)] + \sX^{-s_{i}\nu}[\CO_{\QG(t_{\xi})}(\lambda)] \\
& = \frac{\be^{\nu}-\be^{s_{i}\nu}}{1-\be^{\alpha_{i}}}[\CO_{\QG(t_{\xi})}(\lambda)] + 
    \be^{s_{i}\nu} [\CO_{\QG(t_{\xi})}(\lambda)] \qquad \text{by \eqref{eq:H2}}. 
\end{align*}
Thus we have verified the first equality. 
The second equality follows immediately from the definition of 
the Demazure operator $D_{i}$. This proves the lemma. 
\end{proof}
%
%
\begin{prop}[{see, e.g., \cite[(3.19) and (3.20)]{KNOS}}] \label{prop:demop}
Let $x \in W_{\af}$, and write it as $x = wt_{\xi}$, with 
$w \in W$ and $\xi \in Q^{\vee}$. For $i \in I$, we have 
\begin{equation} \label{eq:dem-op1}
\sD_{i} \cdot [\CO_{\QG(x)}] = 
\begin{cases}
[\CO_{\QG(s_{i}x)}] & \text{\rm if $s_{i}w < w$}, \\
[\CO_{\QG(x)}] & \text{\rm if $s_{i}w > w$}.
\end{cases}
\end{equation}
\end{prop}


\section{Some properties of the defining ideals.}
\label{sec:B}

In this appendix, we work in the algebra
$\BZ\bra{ Q }_n[x_i,(1-y_i)^{\pm 1} \mid i=1,\ldots,n ]$, 
which is identified with 
$\BZ\bra{ Q }_n[x_i \mid i = 1,\ldots,n] \otimes R(H)$, 
with $\BZ\bra{ Q }_n := \BZ\bra{ Q_{1}, \ldots, Q_n }$, 
where the algebra  $\BZ[(1-y_i)^{\pm 1} \mid i=1,\ldots,n ]$ is 
identified with the representation ring $R(H)$ of 
the maximal torus $H$ of $SL_{n}(\BC)$ (in type $A_{n-1}$) 
consisting of the diagonal matrices, via $y_i=1-\be^{-\eps_i}$. 
Let $e_p^{(k)}(x) := e_p(x_1, \ldots, x_k)$ be the $p$-th elementary symmetric polynomial in the variables $x_1, \ldots, x_k$ for $0 \leq p \leq k \leq n$, 
and set $e_{p_1 \cdots p_{n-1}}(x) := e^{(1)}_{p_1}(x) \cdots e^{(n-1)}_{p_{n-1}}(x)$ for $0 \leq p_i \leq i$, $1 \leq i \leq n-1$. 
Also, we define
\begin{equation*}
F_p^{(k)}(x) := \sum_{\substack{I \subset [k] \\ |I|=p}} \prod_{i \in I}(1-x_i)
\prod_{\substack{i \in I \\ i+1 \notin I}}(1-Q_i),
\end{equation*}
\begin{equation*}
\ha{E}_p^{(k)}(x) := \sum_{i=0}^p (-1)^i \binom{k-i}{p-i} F^{(k)}_i(x)
\end{equation*}
for $0 \leq p \leq k \leq n$, and set $\ha{E}_{p_1 \cdots p_{n-1}}(x) := \ha{E}^{(1)}_{p_1}(x) \cdots \ha{E}^{(n-1)}_{p_{n-1}}(x)$ for $0 \leq p_i \leq i$, $1 \leq i \leq n-1$. 
Then we define the quantization map (with respect to the $x$-variables) 
\begin{equation*}
\Qq : \BZ\bra{ Q }_n[x_i \mid i=1,\ldots,n ] \otimes R(H) 
 \rightarrow 
\BZ\bra{ Q }_n[x_i \mid i=1,2,\ldots ] \otimes R(H)
\end{equation*}
as a $\BZ\bra{ Q }_n \otimes R(H)$-module homomorphism 
satisfying the condition: 
\begin{equation*}
\Qq(e_{p_1 \cdots p_{n-1}}(x)) =
\ha{E}_{p_1 \cdots p_{n-1}}(x),  \quad  0 \leq p_i \leq i, 1 \leq i \leq n-1, 
\end{equation*} 
as in \cite[Section~3]{LM}, where $\BZ\bra{ Q }_n \otimes R(H)$ is 
identified with $\BZ\bra{ Q }_n[(1-y_i)^{\pm 1} \mid i=1,\ldots,n ]$; 
in this appendix, 
we mainly follow the notation of \cite{LM}.
Here note that the quantization map in \cite[Section~3]{LM} 
is defined 
on the module $L_n^Q := L_n \otimes \BZ\bra{ Q }_{n-1}$, 
where $L_n := \langle x_1^{i_1} \cdots x_{n-1}^{i_{n-1}} \mid 0\leq i_j \leq n-j \rangle_{\BZ}$. 
By taking the inductive limit $n \rightarrow \infty$, 
the module $L_n^Q$ goes to 
$P_{\infty}^Q:=\BZ\bra{ Q_1,Q_2,Q_3, \ldots }[x_1,x_2,x_3,\ldots ]$, and the quantization map extends to a 
$\BZ\bra{ Q_1,Q_2,Q_3, \ldots }\otimes R(H)$-module homomorphism $\Qq_{\infty}$ 
defined on the algebra 
$P_{\infty}^Q \otimes R(H)$. 
The quantization map $\Qq$ we use here is given 
by the restriction of the map $\Qq_{\infty}$. 

In the following, we use the abbreviations such as: 
\begin{equation*}
\BZ[x] := \BZ[x_i \mid i = 1, \ldots ,n],
\end{equation*}
\begin{equation*}
\BZ[x,(1-y)^{\pm 1}]:= \BZ[x_i,(1-y_i)^{\pm 1} \mid i=1,\ldots ,n],
\end{equation*}
\begin{equation*}
\BZ\bra{ Q }_{m}[x] := \BZ\bra{Q_1, \ldots, Q_{m}}[x_i \mid i = 1, \ldots ,n],
\end{equation*}
\begin{equation*}
\BZ\bra{ Q }_m[(1-y)^{\pm 1}]:= \BZ\bra{Q_1, \ldots, Q_m}[(1-y_i)^{\pm 1} \mid i=1,\ldots ,n],
\end{equation*}
\begin{equation*}
\BZ\bra{ Q }_m[x,(1-y)^{\pm 1}]:= \BZ\bra{Q_1, \ldots, Q_m}[x_i, (1-y_i)^{\pm 1} \mid i=1,\ldots ,n],
\end{equation*}
and 
\begin{equation*}
e_i(1-x):=e_i(1-x_1,\ldots,1-x_n), \quad
e_i((1-y)^{-1}):=e_i((1-y_1)^{-1},\ldots,(1-y_n)^{-1}).
\end{equation*}
Define an ideal $I_n$ of 
$\BZ\bra{ Q }_n[x,(1-y)^{\pm 1}] = \BZ\bra{ Q }_n[x] \otimes R(H)$ 
as the one generated by the polynomials $e_i(1-x)-e_i((1-y)^{-1})$, 
$i=1,\ldots,n$. 
Also, define an ideal $\tilde{I}_n^Q$ of 
$P_{\infty}^Q\otimes R(H)$ as the one generated by the polynomials 
$F^{(n)}_i(x)-e_i((1-y)^{-1})$, $i=1,\ldots,n$. 

We introduce the polynomials $\ff^{(k)}_i$ 
for positive integers $i,k$ with $i\leq k$ by 
\begin{equation*}
\ff_i^{(k)}:=
\begin{cases}
e_i(1-x_1,\ldots,1-x_k)-e_i((1-y_1)^{-1},\ldots,(1-y_k)^{-1}) & \text{if $k \leq n$}, \\ 
e_i(1-x_1,\ldots,1-x_k)-e_i((1-y_1)^{-1},\ldots,(1-y_n)^{-1}) & \text{if $i \leq n < k$}, \\ 
e_i(1-x_1,\ldots,1-x_k) & \text{if $i > n$},
\end{cases}
\end{equation*}
and define $\FFF^{(k)}_i:=\Qq_{\infty}(\ff^{(k)}_i)$. 
Note that the ideal $I_n$ (resp., $\tilde{I}_n^Q$) is generated by $\ff^{(n)}_1,\ldots,\ff^{(n)}_n$ 
(resp., $\FFF^{(n)}_1,\ldots,\FFF^{(n)}_n$). 
Also, we define the polynomials $\ff_{i_1\cdots i_k}$ and $\FFF_{i_1\cdots i_k}$ by 
\begin{equation*}
\ff_{i_1\cdots i_k}:=\ff_{i_1}^{(1)}\ff_{i_2}^{(2)}\cdots \ff_{i_k}^{(k)}, \qquad 
\FFF_{i_1\cdots i_k}:=\FFF_{i_1}^{(1)}\FFF_{i_2}^{(2)}\cdots \FFF_{i_k}^{(k)}
\end{equation*}
for $k=1,2,\ldots$ and $0\leq i_j \leq j$, $1 \leq j \leq k$. 
\begin{lem} \label{lemB1}
The sets $\{ \ff_{i_1\cdots i_k} \mid 0\leq i_j \leq j, \; 1\leq j \leq k \}$ and 
$\{ \FFF_{i_1\cdots i_k} \mid 0\leq i_j \leq j, \; 1\leq j \leq k \}$ are both 
$\BZ\bra{ Q }_{k} \otimes R(H)$-linear bases of $L^Q_{k+1} \otimes R(H)$ for any positive integer $k$. 
\end{lem}
\begin{proof} 
The polynomial $\ff_{i_1\cdots i_k}$ is expanded as:
\begin{equation*}
\ff_{i_1\cdots i_k}= f_{i_1\cdots i_k}(x) + \textrm{an $R(H)$-linear combination of 
$f_{j_1\cdots j_k}(x)$'s with $\sum_{a=1}^k j_a < \sum_{a=1}^k i_a$},
\end{equation*}
where $f_{i_1\cdots i_k}(x):=f_{i_1}^{(1)}(x)\cdots f_{i_k}^{(k)}(x)$, 
$f_i^{(j)}:=e_i(1-x_1,\ldots,1-x_j)$. Since 
the set  $\{ f_{i_1\cdots i_k}(x) \mid 0\leq i_j \leq j, \; 1\leq j \leq k \}$ forms a linear basis of $L_{k+1}$, 
the set  $\{ \ff_{i_1\cdots i_k} \mid 0\leq i_j \leq j, \; 1\leq j \leq k \}$ is a linear basis of 
$L^Q_{k+1} \otimes R(H)$. 
Similarly, we have 
\begin{equation*}
\FFF_{i_1\cdots i_k}= F_{i_1\cdots i_k}(x) + \textrm{an $R(H)$-linear combination of 
$F_{j_1\cdots j_k}(x)$'s with $\sum_{a=1}^k j_a < \sum_{a=1}^k i_a$},
\end{equation*}
where $F_{i_1\cdots i_k}(x):=F_{i_1}^{(1)}(x)\cdots F_{i_k}^{(k)}(x)$. 
Hence the set $\{ \FFF_{i_1\cdots i_k} \mid 0\leq i_j \leq j, \; 1\leq j \leq k \}$ is also 
a linear basis of 
$L^Q_{k+1} \otimes R(H)$ by \cite[Proposition 3.26]{LM}. 
This proves the lemma. 
\end{proof}

The lemma above implies that the sets 
$\bigcup_{k=1}^{\infty} \{ \ff_{i_1\cdots i_k} \mid 0\leq i_j \leq j, \; 1\leq j \leq k \}$ and 
$\bigcup_{k=1}^{\infty} \{ \FFF_{i_1\cdots i_k} \mid 0\leq i_j \leq j, \; 1\leq j \leq k \}$ are both 
$\BZ\bra{ Q_1, Q_2,\ldots  }\otimes R(H)$-linear bases of $P^Q_{\infty}\otimes R(H)$. 
\begin{prop} \label{propB2} 
Let $J_n$ denote the ideal of $P^Q_{\infty}\otimes R(H)$ generated by 
the set
\begin{equation*}
\{ 1-x_{n+1},1-x_{n+2},\ldots,Q_{n},Q_{n+1},\ldots\}.
\end{equation*}
Then, the ideal $I_n$ is mapped to $\tilde{I}_n^Q+J_n$ by the quantization map $\Qq$. 
\end{prop}

\begin{proof} 
Let $\tilde{I}_n$ be the ideal of $P^Q_{\infty}\otimes R(H)$ generated by 
$\ff^{(n)}_i,$ $i=1,\ldots,n$. It suffices to prove that 
the ideal $\tilde{I}_n+J_n$ is mapped to $\tilde{I}_n^Q+J_n$ by the quantization map 
$\Qq_{\infty}$. If $k>n,$ then we have
\begin{equation*}
\ff^{(k)}_i |_{x_{n+1}=x_{n+2}=\cdots =x_{k}=1} = 
\begin{cases}
\ff_i^{(n)} & \text{if $i \leq n$}, \\ 
0 & \text{if $i > n$}, 
\end{cases}
\end{equation*}
and hence 
\begin{equation} \label{eq:B2-0}
\ff_{i_1\cdots i_k} |_{x_{n+1}=x_{n+2}=\cdots =x_{k}=1} \in I_n 
\quad \text{ for } \quad (i_{n+1},i_{n+2},\ldots,i_k)\not=(0,0,\ldots,0).
\end{equation}
Here we note that
\begin{equation} \label{eq:B2-1}
(P^Q_{\infty}\otimes R(H))/(\tilde{I}_n+ J_n) \cong 
\BZ\bra{ Q }_{n-1}[x_i,(1-y_i)^{\pm 1} \mid i=1,\ldots,n ]/I_n
\end{equation}
is isomorphic to 
$L_n\otimes \BZ\bra{ Q }_{n-1} \otimes R(H)$ as a $\BZ\bra{ Q }_{n-1} \otimes R(H)$-module. 
Also, Lemma~\ref{lemB1} implies that each element $\alpha$ in $P^Q_{\infty}\otimes R(H)$ 
can be written as: 
\begin{equation*}
\alpha = \sum_{i_1,\ldots, i_N} c_{i_1\cdots i_N}\ff_{i_1\cdots i_N}, \qquad  
c_{i_1\cdots i_N} \in \BZ\bra{ Q_1,Q_2,\ldots  }\otimes R(H),
\end{equation*}
for sufficiently large $N$. 
Since $I_n \subset \tilde{I}_n+J_n$, it follows from \eqref{eq:B2-0} that 
\begin{equation*}
\alpha \equiv \sum_{\substack{i_1,\ldots, i_N \\ i_{n+1}=i_{n+2}=\cdots=i_N=0}}
c_{i_1\cdots i_N}\ff_{i_1\cdots i_N} \mod \tilde{I}_n+J_n.
\end{equation*}
Therefore, we deduce by using Lemma~\ref{lemB1} that 
each element $\alpha$ in $I_n \subset \tilde{I}_n+J_n$ can be written as a linear combination of 
the polynomials $\ff_{i_1\cdots i_N}$ with $(i_{n+1},i_{n+2},\ldots,i_N) \not= (0,0,\ldots,0)$. 
Hence it suffices to show that 
$\Qq( \ff_{i_1\cdots i_k})=\FFF_{i_1\cdots i_k} \in \tilde{I}_n^Q+J_n$ for $k>n$ and 
$(i_{n+1},i_{n+2},\ldots,i_k)\not=(0,0,\ldots,0)$. 
The recurrence relation \cite[(3.7)]{LM} implies that
\begin{equation} \label{eq:B2-2}
\FFF_i^{(j)} |_{ \begin{subarray}{l} x_{n+1}=\cdots =x_k=1 \\ Q_n=Q_{n+1}=\cdots =Q_{k}=0 \end{subarray} } = 
\begin{cases}
\FFF_i^{(n)}|_{Q_n=0} & \text{if $i \leq n$}, \\ 
0 & \text{if $i > n$}
\end{cases}
\end{equation}
for $k\geq j >n$. 
This shows that $\FFF_{i_1\cdots i_k} \in \tilde{I}_n^Q+J_n$ for 
$(i_{n+1},i_{n+2},\ldots,i_k)\not=(0,0,\ldots,0)$, as desired. 
This proves the proposition. 
\end{proof} 

Let $I_n^Q$ denote the ideal of $\BZ\bra{ Q }_{n-1}[x,(1-y)^{\pm 1}]$ generated by the polynomials 
\[ \FFF^{(n)}_1|_{Q_n=0},\ldots,\FFF^{(n)}_n|_{Q_n=0}. \] 
Equation \eqref{eq:B2-2} implies that the image of 
the set $\bigcup_{k=1}^{\infty} \{ \FFF_{i_1\cdots i_k} \mid 0\leq i_j \leq j, \; 1\leq j \leq k \}$ 
in the quotient $(P^Q_{\infty}\otimes R(H))/(\tilde{I}_n^Q+J_n) 
\cong \BZ\bra{ Q }_{n-1}[x,(1-y)^{\pm 1}]/I_n^Q$ coincides with 
that of $\{ \FFF_{i_1\cdots i_{n-1}} \mid 0\leq i_j \leq j, \; 1\leq j \leq n-1 \}$. 
Since the set $\bigcup_{k=1}^{\infty} \{ \FFF_{i_1\cdots i_k} \mid 0\leq i_j \leq j, \; 1\leq j \leq k \}$ 
forms a $\BZ\bra{ Q }_{n-1} \otimes R(H)$-linear basis of $P^Q_{\infty} \otimes R(H)$, 
the image of the set $\{ \FFF_{i_1\cdots i_{n-1}} \mid 0\leq i_j \leq j, \; 1\leq j \leq n-1 \}$ 
generates $\BZ\bra{ Q }_{n-1}[x,(1-y)^{\pm 1}]/I_n^Q$ as a 
$\BZ\bra{ Q }_{n-1}[(1-y)^{\pm 1}]$-module. 
Thus, we obtain the following. 

\begin{cor} \label{cor:fg}
\begin{enu}
\item The quantization map $\Qq$ induces a surjective 
$\BZ\bra{ Q }_{n-1}[(1-y)^{\pm 1}]$-linear map 
from $\BZ\bra{ Q }_{n-1}[x,(1-y)^{\pm 1}]/I_n$ to 
$\BZ\bra{ Q }_{n-1}[x,(1-y)^{\pm 1}]/I_n^Q$.

\item The quotient algebra $\BZ\bra{ Q }_{n-1}[x,(1-y)^{\pm 1}]/I_n^Q$ is 
finitely generated as a module over  $\BZ\bra{ Q }_{n-1}[(1-y)^{\pm 1}]$. 
\end{enu}
\end{cor}  

\begin{rem} \label{rem:locfg} \mbox{}
\begin{enu}
\item
Since the image of the set $\{ \FFF_{i_1 \cdots i_{n-1}} \mid 0 \leq i_j \leq j, \; 1 \leq j \leq n-1 \}$
in $\BZ\bra{ Q }_{n-1}[x,(1-y)^{\pm 1}]/I_n^Q$ is linearly independent over $\BZ\bra{ Q }_{n-1}[(1-y)^{\pm 1}]$ by Lemma~\ref{lemB1}, 
the above map $\BZ\bra{ Q }_{n-1}[x,(1-y)^{\pm 1}]/I_n \rightarrow
\BZ\bra{ Q }_{n-1}[x,(1-y)^{\pm 1}]/I_n^Q$ induced by $\Qq$ is injective, and hence bijective. 

\item
For $m \in \BZ_{\geq 1}$,  let $\BZ[Q]_m := \BZ[Q_{1}, \ldots, Q_m]$ be the ring of polynomials in the variables $Q_1, \ldots, Q_m$ with coefficients in $\BZ$, and $( Q )_m := ( Q_1, \ldots, Q_m )$ the ideal in $\BZ[Q]_m$ generated by $Q_1, \ldots, Q_m$. 
Also, let $(\BZ[Q]_m)_{\mathrm{loc}}$ denote the localization of the ring $\BZ[Q]_m$ with respect to the multiplicative set $1 + ( Q )_m$. 
Then, the same arguments as above show that the assertions of Corollary~\ref{cor:fg} still hold, with $\BZ\bra{ Q }_{n-1}$ replaced by $(\BZ[Q]_{n-1})_{\mathrm{loc}}$.\end{enu}
\end{rem}

The results below are not used in this paper. 
However, they are need for explicit computations of the quantum product of (opposite) Schubert classes in $QK_{H}(Fl_{n})$, since quantum double Grothendieck polynomials represent the corresponding (opposite) Schubert classes in $QK_{H}(Fl_{n})$ (see \cite{MaNS}).

For $w \in S_n$, we consider the double Grothendieck polynomial $\FG_w(x,y)$ defined by 
\begin{equation*}
\FG_w(x,y) := \pi^{(x)}_{w^{-1}w_{\circ}}\FG_{w_{\circ}}(x,y),  \quad  
\FG_{w_{\circ}}(x,y) := \prod_{i+j\leq n}(x_i+y_j-x_iy_j),
\end{equation*} 
and its quantization $\FG^Q_w(x,y):=\Qq(\FG_w(x,y))$, 
called the quantum double Grothendieck polynomial. 
Here $\pi^{(x)}_{w^{-1}\lng}$ denotes the Demazure 
operator acting on the $x$-variables.
\begin{rem}
In \cite[Section~8]{LM}, the quantum double Grothendieck polynomial 
$\FG^q_w(x, y)$ is defined by applying 
the quantization map with respect to the $y$-variables 
to the classical one. Our convention for the polynomial 
$\FG^Q_w(x,y)$ above uses the quantization 
with respect to the $x$-variables unlike that in \cite{LM}.
More precisely, the polynomial $\FG^Q_w(x, y)$ defined above 
coincides with $\FG^q_{w^{-1}}(y, x)$ in the notation of \cite{LM}.
\end{rem} 

In the following, we use the symbols $\eta^0(x):=1$ and 
\begin{equation*}
\eta^j(x):=(x+y_1-xy_1)(x+y_2-xy_2) \cdots (x+y_j-xy_j),
  \quad 1 \leq j \leq n-1.
\end{equation*} 

For a permutation $w \in S_{n}$, 
we set $c_i(w) := \vert \{n \geq j > i \mid w(j) < w(i) \} \vert$ for $1 \leq i \leq n$, 
and call the sequence $(c_1(w), c_2(w), \ldots, c_n(w))$ the code of $w$; 
$w$ is said to be dominant if $c_1(w) \geq c_2(w) \geq \cdots \geq c_n(w)$. 

\begin{lem} \label{dom-gro}
If $w\in S_n$ is a dominant permutation, 
with its code $(c_1(w),\ldots,c_n(w))$, then 
\begin{equation*}
\FG_w(x,y)= \eta^{c_1(w)}(x_1) \cdots \eta^{c_{n-1}(w)}(x_{n-1}).
\end{equation*} 
\end{lem}

\begin{proof}
The assertion holds for $w=\lng$. Let $w\in S_n$ be a dominant permutation 
with $\ell(w)<\ell( \lng)$. Then there exist indices $1 \leq i \leq n-1$ 
such that $c_i(w)=c_{i+1}(w)$. 
Let $i_0:=\min \{ 1 \leq i \leq n-1  \mid c_i(w)=c_{i+1}(w)\}$, and set 
$v:=ws_{i_0}$. Then we have 
\begin{equation*}
c_i(v)= 
\begin{cases}
c_i(w) & \text{\rm if $i \ne i_0$}, \\ 
c_{i_0}(w)+1 & \text{\rm if $i = i_0$}.
\end{cases}
\end{equation*}
Hence, $v$ is also dominant and $\ell(v)= \ell(w)+1$. 
By our induction hypothesis, we get 
\begin{equation*}
\FG_v (x,y)= \eta^{c_1(w)}(x_1) \cdots \eta ^{c_{i_0}(w)+1}(x_{i_0}) 
\eta^{c_{i_0+1}(w)}(x_{i_0+1}) \cdots \eta^{c_{n-1}(w)}(x_{n-1}),
\end{equation*} 
and this implies that 
\begin{equation*}
\FG_w(x,y)= \pi_{i_0}^{(x)} \FG_v(x,y)= 
\eta^{c_1(w)}(x_1) \cdots \eta ^{c_{i_0}(w)}(x_{i_0}) \eta^{c_{i_0+1}(w)}(x_{i_0+1})
\cdots \eta^{c_{n-1}(w)}(x_{n-1}),
\end{equation*} 
as desired. 
\end{proof}

\begin{prop} \label{propB6} 
If $w\in S_{n+1}$ satisfies $w(n+1)\not= n+1$, then $\FG_w(x,y)$ is contained in $I_n$. 
\end{prop}

\begin{proof}
For any $w\in S_{n+1}$, there exists a unique permutation $v\in S_n$ 
such that $wv(1) > wv(2) > \cdots > wv(n)$. 
Here, $wv$ is a dominant permutation in $S_{n+1}$ and 
$wv(1)=n+1$ by the assumption that $w(n+1) \ne n+1$. 
Hence, by Lemma~\ref{dom-gro}, we have 
\begin{equation*}
\FG_{wv}(x,y) = \eta^n(x_1)\eta^{c_2(wv)}(x_2) \cdots \eta^{c_n(wv)}(x_n).
\end{equation*} 
Also, we have the following equalities modulo $I_n$: 
\begin{align*}
e_n((1-y)^{-1})\cdot \eta^n(x_1) 
& = e_n((1-y)^{-1}) \cdot \Big( 1+\sum_{a=1}^n (x_1-1)^ae_a(1-y) \Big) \\ 
& = e_n((1-y)^{-1})+ \sum_{a=1}^n (x_1-1)^ae_{n-a}((1-y)^{-1}) \\ 
& = e_n(1-x)+ \sum_{a=1}^n (x_1-1)^ae_{n-a}(1-x) \\ 
& = e_n(1-x) \cdot \Big( 1+\sum_{a=1}^n (x_1-1)^ae_a((1-x)^{-1}) \Big) \\ 
& = e_n(1-x) \cdot \prod_{j=1}^{n}\Big( 1-\frac{1-x_1}{1-x_j} \Big)=0. 
\end{align*}
Since $e_n((1-y)^{-1})=(1-y_1)^{-1} \cdots (1-y_n)^{-1}$ 
is invertible, $\eta^n(x_1)$ is an element of $I_n$. 
Therefore, $\FG_{wv}(x,y)$ is contained in $I_n$. 
The choice of $v$ implies that  $\ell(wv) = \ell(w)+\ell(v)$, 
and hence we have $\FG_w(x,y) = \pi_v^{(x)} \FG_{wv}(x,y)$. 
Since the action of the Demazure operator $\pi_v^{(x)}$ preserves the
ideal $I_n$, we deduce that $\FG_w(x,y)$ is also an element of $I_n$. 
This proves the proposition. 
\end{proof} 

Here observe that for $w \in S_{n+1}$, we have $\FG^Q_w(x,y) \in \BZ\bra{ Q }_{n}[x, (1-y)^{\pm 1}]$. 
Hence, by applying the quantization map $\Qq$, we obtain the following 
from Propositions~\ref{propB2} and \ref{propB6}; 
recall the isomorphism \eqref{eq:B2-1}: 
$(P^Q_{\infty}\otimes R(H))/(\tilde{I}_n^Q+J_n) \cong 
\BZ\bra{ Q }_{n-1}[x,(1-y)^{\pm 1}]/I_n^Q$. 
\begin{cor}
If $w\in S_{n+1}$ satisfies $w(n+1)\not= n+1$, then $\FG^Q_w(x,y)$ is contained in 
$I_n^Q \otimes \BZ\bra{ Q_n }$. 
\end{cor}


{\small

}

\end{document}